\documentclass[preprint,11pt]{elsarticle}
\usepackage{eurosym}
\usepackage{makeidx}
\usepackage{amsfonts}
\usepackage{amsmath}
\usepackage{amssymb}
\usepackage{graphicx}
\usepackage{bbold}
\usepackage[colorlinks]{hyperref}
\usepackage{palatino}
\usepackage{upgreek}
\usepackage[usenames,dvipsnames]{xcolor}

\hypersetup{colorlinks=true, linkcolor=blue, citecolor=green,
filecolor=black, urlcolor=black }
\providecommand{\U}[1]{\protect\rule{.1in}{.1in}} \journal{\dots}
\newtheorem{theorem}{Theorem}

\newtheorem{corollary}[theorem]{Corollary}

\newtheorem{definition}[theorem]{Definition}
\newtheorem{example}[theorem]{Example}

\newtheorem{lemma}[theorem]{Lemma}

\newtheorem{proposition}[theorem]{Proposition}
\newtheorem{remark}[theorem]{Remark}

\newenvironment{proof}[1][Proof]{\noindent\textbf{#1.} }{\ \rule{0.5em}{0.5em}}

\renewcommand{\thefootnote}{\fnsymbol{footnote}}
\begin{document}

\begin{frontmatter}
\title{C\`{a}dl\`{a}g Skorokhod problem driven by a maximal monotone operator}
\author{Lucian Maticiuc$^{a,b}$, Aurel R\u{a}\c{s}canu$^{a,c}$, Leszek
S\l omi\'{n}ski$^{d}$, Mateusz Topolewski$^{d}$\medskip}
\address{$^{a}$ Faculty of Mathematics, \textquotedblleft Alexandru
Ioan Cuza\textquotedblright\ University of Ia\c{s}i\\ Carol 1 Blvd., no. 11,
Ia\c{s}i, Romania \medskip \\$^{b}$~Department of Mathematics,
\textquotedblleft Gheorghe Asachi\textquotedblright\ Technical University of Ia\c{s}i\\
Carol I Blvd., no. 11, 700506, Romania \medskip \\$^{c}$ \textquotedblleft Octav Mayer\textquotedblright\ Mathematics Institute of the Romanian Academy, Iasi branch \\ Carol I Blvd., no. 8, Iasi, 700506, Romania \medskip \\$^{d}$ Faculty of Mathematics and Computer
Science, Nicolaus Copernicus University \\ ul. Chopina 12/18, 87-100
Toru\'n, Poland}

\begin{abstract}
The article deals with the existence and uniqueness of the solution of
the follo\-wing diffe\-ren\-tial equation (a c\`{a}dl\`{a}g Skorokhod
problem) driven by a maxi\-mal monotone operator, and with singular
input generated by the
c\`{a}dl\`{a}g function $m$:%
\[
\left\{
\begin{array}
[c]{l}%
dx_{t}+A\left(  x_{t}\right)  \left(  dt\right)  +dk_{t}^{d}\ni dm_{t}%
\,,~t\geq0,\\
x_{0}=m_{0},
\end{array}
\right.
\]
where $k^{d}$ is a pure jump function.

The jumps outside of the constrained domain
$\overline{\mathrm{D}(A)}$ are
counteracted through the gene\-ra\-lized projection $\Pi$ by taking $x_{t}%
=\Pi(x_{t-}+\Delta m_{t})$, whenever $x_{t-}+\Delta m_{t}\notin\overline
{\mathrm{D}(A)}\,$.

Approximations of the solution based on discretization and Yosida
penalization are also considered.\bigskip
\end{abstract}
\end{frontmatter}

\textbf{AMS Classification subjects}: 34A60, 47N20.$\smallskip$

\textbf{Keywords or phrases}: Skorokhod problem; maximal monotone operators;
Yosida appro\-ximation.

\renewcommand{\thefootnote}{\fnsymbol{footnote}} \footnotetext{\textit%
{\scriptsize E-mail addresses:} {\scriptsize lucian.maticiuc@ymail.com
(Lucian Maticiuc), aurel.rascanu@uaic.ro (Aurel R\u{a}\c{s}canu),
leszeks@mat.umk.pl (Leszek S\l omi\'{n}ski), woland@mat.umk.pl (Mateusz
Topolewski).}}

\section{Introduction}

Let $\mathbb{H}$ be a separable real Hilbert space and $A:\mathbb{H}%
\rightrightarrows \mathbb{H}$ be a maximal monotone multivalued operator on $%
\mathbb{H}$ with the domain $\mathrm{D}(A)=\{z\in \mathbb{H}:A(z)\neq
\emptyset \}$ and its graph $\mathrm{Gr}(A)=\{(z,y)\in \mathbb{H\times H}%
:z\in \mathbb{H},y\in A(z)\}$.

Let $\Pi:\mathbb{H}\rightarrow\overline{\mathrm{D}(A)}$ be a generalized
projection on $\overline{\mathrm{D}(A)}$ (see Definition \ref{gener proj})
and $m$ a c\`{a}dl\`{a}g function.

The main topic of our paper is the existence and uniqueness of a solution $x$
of the following multivalued differential equation with singular input $%
dm_{t}:$%
\begin{equation}
\left\{
\begin{array}{l}
dx_{t}+A\left( x_{t}\right) \left( dt\right) +dk_{t}^{d}\ni dm_{t}\;,~~t\geq
0,\medskip \\
x_{0}=m_{0},%
\end{array}%
\right.  \label{eq1.1 a}
\end{equation}%
where $x_{t}=\Pi (x_{t-}+\Delta m_{t}),$ if $x_{t-}+\Delta m_{t}\notin
\overline{\mathrm{D}(A)}$ and $k^{d}$ is a pure jump function such that $%
\Delta k_{t}^{d}=\Delta m_{t}-\Delta x_{t}\,.$

The generalized projection $\Pi$ ensures that the jumps from $x_{t-}$ to $%
x_{t-}+\Delta m_{t}$ are counteracted whenever these jumps leave the domain.

By a solution of (\ref{eq1.1 a}) we understand a pair $(x,k)$ of c\`{a}dl%
\`{a}g functions such that%
\begin{equation*}
x_{t}+k_{t}=m_{t}\;,~~t\geq 0,
\end{equation*}%
such that $x_{t}\in \overline{\mathrm{D}(A)}$, for any $t\geq 0$, $%
k=k^{c}+k^{d}$ has locally bounded variation with $k^{c}$ being its
continuous part, $dk_{t}^{c}\in A\left( x_{t}\right) \left( dt\right) $ (see
Definition \ref{definition Annex 14}) and%
\begin{equation*}
k_{t}^{d}=\sum_{0\leq s\leq t}\Delta k_{s}=\sum_{0\leq s\leq t}\left( I-\Pi
\right) \left( x_{s-}+\Delta m_{s}\right) .
\end{equation*}%
We notice that in the particular case when $\Pi $ is the orthogonal
projection equation (\ref{eq1.1 a}) is equivalent with the problem%
\begin{equation}
\left\{
\begin{array}{l}
x_{t}+k_{t}=m_{t}\;,~~t\geq 0,\medskip \\
dk_{t}\in A\left( x_{t}\right) \left( dt\right) .%
\end{array}%
\right.  \label{eq1.1 b}
\end{equation}%
When $m$ has some smoothness properties, equation (\ref{eq1.1 b}) was
intensely studied in the frame of nonlinear analysis (see, e.g. Barbu \cite%
{ba/76, ba/10}, Brezis \cite{br/73} and their references). Strong and genera%
{\footnotesize \-}li{\footnotesize \-}zed solutions were defined. In the
case when the interior of the domain of $A$ is nonempty the generalized
solution can be reformulated in a similar manner as for the Skorohod
problem. In fact, the convex Skorohod problem is a particular case when $A$
is the subdifferential of the indicator function of a closed convex set (see
Subsection \ref{Subsect Introd} for more details). In the recent article
\cite{ga-ra/12} Gassous, R\u{a}\c{s}canu and Rotenstein focused on the
extension of the Skorohod problem to the situation when the multivalued
subdifferential operator is perturbed by considering oblique subgradients.
The study is continued in R\u{a}\c{s}canu, Rotenstein \cite{ra-ro/14}, where
the authors constructed a non-convex framework for the same problem.

Particular cases of the above type equations were considered earlier in many
papers. The existence and uniqueness of solution for (\ref{eq1.1 b}) was
proved independently by R\u{a}\c{s}canu in \cite{ra/96} (for the infinite
dimensional framework) and by C\'{e}pa in \cite{ce/98} (the finite
dimensional case). Barbu and R\u{a}\c{s}canu in \cite{ba-ra/97} studied
parabolic variational inequalities with singular inputs (equation formally
written as $dx_{t}+A\left( x_{t}\right) \left( dt\right) \ni
f_{t}dt+dm_{t}\, $) with $A$ of the form $A=A_{0}+\partial \varphi $, where $%
A_{0}$ is a linear positive defined operator and $\partial \varphi $ is the
subdifferential operator associated to a convex lower semicontinuous
function $\varphi $ (see Annexes \ref{Maximal monotone}). Moreover, using
the Fitzpatrick function, R\u{a}\c{s}canu and Rotenstein provided in \cite%
{ra-ro/11} a one-to-one correspondence between the solutions of multivalued
differential equations and the minimum points of some convex functionals.

We pointing out that the authors of all the mentioned above papers
restricted themselves to processes with continuous trajectories.\medskip

It is well known that for every nonempty closed convex set $\bar{D}\subset
\mathbb{R}^{d},$ its indicator function $\varphi \left( x\right) :={I}_{\bar{%
D}}\left( x\right) =\left\{
\begin{array}{rl}
0, & \text{if \ }x\in \bar{D},\medskip \\
+\infty , & \text{if \ }x\notin \bar{D}.%
\end{array}%
\right. $ is a convex lower semicontinuous function (see Annexes \ref%
{Maximal monotone}). This implies that deterministic equation with
reflecting boundary condition in convex domains are special cases of (\ref%
{eq1.1 a}). Such equations were introduced by Skorokhod \cite{sk/61, sk/62}
in the one-dimensional case for $D=\mathbb{R}^{+}:=[0,\infty )$. The
multidimensional version of Skorohod's equation was studied in detail in
Tanaka \cite{ta/79} in the case of a convex domain $D$. The reflection
problem in non-convex domain was studied by Lions and Sznitman in \cite%
{li-sn/83}.\medskip

We also mention the connection of the Skorohod problem with so called
sweeping process which are the solutions of the next particular problem%
\begin{equation*}
\left\{
\begin{array}{l}
x^{\prime }\left( t\right) +N_{E\left( t\right) }\left( x\left( t\right)
\right) \ni f\left( t,x\left( t\right) \right) ,\quad t\geq 0,\medskip \\
x\left( 0\right) =x_{0},%
\end{array}%
\right.
\end{equation*}%
where $N_{E\left( t\right) }$ is the external normal cone to the moving
domain $E\left( t\right) $ in $x\left( t\right) $ (in particular $N_{E\left(
t\right) }\left( x\left( t\right) \right) =\partial I_{E\left( t\right)
}\left( x\left( t\right) \right) $ if $E\left( t\right) $ is a closed convex
set). This problem was introduced by Moreau (see \cite{mo/77}) for the case $%
f\equiv 0$ and convex sets $E\left( t\right) ,\ t\geq 0$, and it has been
intensively studied since then by several authors; see, e.g., Castaing \cite%
{ca/83}, Monteiro Marques \cite{mo/93}. For the case of a sweeping process
without the assumption of convexity on the sets $E\left( t\right) $, we
refer to \cite{be/00}, \cite{be-ca/96}, \cite{co-go/99}, \cite{co-mo/03} and
\cite{th/03}. The extension to the case $f\not\equiv 0$ was considered,
e.g., in Castaing, Monteiro Marques \cite{ca-mo/96} and Edmond, Thibault
\cite{ed-th/06}. The study of sweeping processes with perturbations was
introduced by Castaing, D\'{u}c Ha and Valadier \cite{ca-du/93} and
Castaing, Monteiro Marques \cite{ca-mo/96}. The problem of sweeping
processes with stochastic perturbations was the subject of Castaing \cite%
{ca/73, ca/76}, Bernicot, Venel \cite{be-ve/11} and, more recently, of
Falkowski, S\l omi\'{n}ski \cite{fa-sl/15}.\medskip

Approximations of solutions of the Skorohod problem, considered in the c\`{a}%
dl\`{a}g case, were studied by \L aukajtys and S\l omi\'{n}ski in \cite%
{la-sl/13} where is involved the classical orthogonal projection on $%
\overline{\mathrm{D}(A)}$, i.e.,%
\begin{equation*}
\Pi _{\overline{\mathrm{D}(A)}}(x)\in \overline{\mathrm{D}(A)}\quad \text{and%
}\quad |\Pi _{\overline{\mathrm{D}(A)}}(x)-x|=\inf \{|z-x|:z\in \overline{%
\mathrm{D}(A)}\}.
\end{equation*}%
Skorohod problem with jumps and two time-depending obstacles was treated by S%
\l omi\'{n}ski and Wojciechowski \cite{sl-wo/10}.

In the present paper we study existence, uniqueness and approximations of
solutions of equation (\ref{eq1.1 a}), provided that%
\begin{equation}
\mathrm{Int}\left( \mathrm{D}\left( A\right) \right) \neq \emptyset
\label{assumpt int}
\end{equation}%
and $\Pi $ is a generalized projection on the $\overline{\mathrm{D}(A)}$.

The results concerning the multivalued deterministic equation with jumps,
considered in our work are new, even in the case of classical projection $%
\Pi _{\overline{\mathrm{D}(A)}}\,$. We should mention that assumption (\ref%
{assumpt int}) is essential for the proof. Although in the finite
dimensional case this assumption is not a strong restriction, in the Hilbert
space framework this assumption become quite restrictive. But the existence
results provided by this research are essential in order to obtain further
existence results for the same problem without condition (\ref{assumpt int})
(we recall article \cite{ra/96} where it was considered the continuous
version of (\ref{eq1.1 b}) and assumption (\ref{assumpt int}) is
suppressed). In addition, the infinite dimensional frame considered here
allows to consider infinite dimensional systems of type (\ref{eq1.1 a}).

We outline that the infinite systems of equations (with or without state
constraints) describe a various types of real problems which appear in
mechanics, in the theory of branching processes, in the theory of neural
networks, etc. (see Zautykov \cite{za/65}, Zautykov \& Valeev \cite{za-va/74}
or reference from \cite{mu-al/12} for a more complete literature scene of
applications). Infinite systems of differential equations also appear in
partial differential equations which are treated through semidiscretization.

Also very important applications of (\ref{eq1.1 a}) are in the stochastic
reflection problems, in which case $dm_{t}$ is interpreted as a random noise
on the system (see \cite%
{be-ra/97,ce/98,li-sn/83,la-sl/03,la-sl/13,re-wu/11,ra/96,si/00,sl-wo/10}%
).\medskip

The paper is organized as follows. In Section 2 we prove existence and
uniqueness of solutions and we give some convergence results in the uniform
norm and in the Skorokhod topology $J_{1}$. We also give examples of
projections $\Pi$.

Section 3 is devoted to the Yosida approximations of solutions. If we
consider the approximating equation of the form
\begin{equation*}
x_{t}^{\varepsilon }+\int_{0}^{t}A_{\varepsilon }(x_{s}^{\varepsilon
})ds=m_{t},\quad t\in \mathbb{R}^{+},
\end{equation*}%
where $A_{\varepsilon }(z)=\frac{1}{\varepsilon }\left( z-J_{\varepsilon
}\left( z\right) \right) $ and $J_{\varepsilon }(z)=(I+\varepsilon
A)^{-1}(z) $, then%
\begin{equation*}
x_{t}^{\varepsilon }\rightarrow x_{t}\mathbf{1}_{\left\vert \Delta
m_{t}\right\vert =0}+\left( x_{t-}+\Delta m_{t}\right) \mathbf{1}%
_{\left\vert \Delta m_{t}\right\vert >0}\,,~\forall t\in \mathbb{R}^{+},%
\text{ as }\varepsilon \searrow 0,
\end{equation*}%
where $x$ is the solution of (\ref{eq1.1 a}) associated to the classical
projection. Also other convergence results are obtained.

When the approximating equation has the form
\begin{equation*}
x_{t}^{\varepsilon }+\int_{0}^{t}A_{\varepsilon }(x_{s}^{\varepsilon
})ds+\sum_{0\leq s\leq t}\left( I-\Pi \right) (x_{s-}^{\varepsilon }+\Delta
m_{s})\mathbf{1}_{\left\vert \Delta m_{s}\right\vert >\varepsilon
}=m_{t},\quad t\in \mathbb{R}^{+},
\end{equation*}%
with a generalized projection $\Pi $, we prove that $x^{\varepsilon }$ is
uniformly convergent on every compact interval to the solution of (\ref%
{eq1.1 a}).\medskip

In the paper we use the following notations: $\mathbb{D}\left( \mathbb{R}%
^{+},\mathbb{H}\right) $ is the space of all c\`{a}dl\`{a}g mappings $y:%
\mathbb{R}^{+}\rightarrow\mathbb{H}$ (i.e., $y$ is right continuous and
admits left-hand limits), endowed with the Skorokhod topology $J_{1}$. For $%
y\in\mathbb{D}\left( \mathbb{R}^{+},\mathbb{H}\right) $, $\delta>0$, $T\in%
\mathbb{R}^{+}$ we denote by $\boldsymbol{\gamma}_{y}(\delta,T)$ the modulus
of c\`{a}dl\`{a}g continuity of $y$ (c\`{a}dl\`{a}g modulus of $y$) on $%
[0,T] $, i.e.%
\begin{equation*}
\boldsymbol{\gamma}_{y}(\delta,T)=\inf\{\max_{i\leq r}\mathcal{O}%
_{y}([t_{i-1},t_{i})):0=t_{0}<\ldots<t_{r}=T,\;\inf_{i<r}(t_{i}-t_{i-1})\geq
\delta\},
\end{equation*}
where $\mathcal{O}_{y}(I):=\sup_{s,t\in I}|y_{s}-y_{t}|$.

If $k\in \mathbb{D}\left( \mathbb{R}^{+},\mathbb{H}\right) $ is a function
with locally bounded variation then $\left\updownarrow {k}\right\updownarrow
{_{\left[ t,T\right] }}$ stands for its variation on $[t,T]$ and $%
\left\updownarrow {k}\right\updownarrow {_{T}}=\left\updownarrow {k}%
\right\updownarrow {_{\left[ 0,T\right] }}$,
\begin{equation*}
k_{t}^{c}:=k_{t}-\sum_{s\leq t}\Delta k_{s}\quad \text{and}\quad
k_{t}^{d}:=k_{t}-k_{t}^{c},\quad t,T\in \mathbb{R}^{+}.
\end{equation*}%
Set $\left\Vert x\right\Vert _{\left[ s,t\right] }:={\sup\limits_{r\in \left[
s,t\right] }|x_{r}|}$ and $\left\Vert x\right\Vert _{t}:=\left\Vert
x\right\Vert _{\left[ 0,t\right] }.$

\subsection{Survey on existence results for (\protect\ref{eq1.1 b}) in the
absolutely continuous case}

\label{Subsect Introd}We recall in this section some classical results
concerning the existence of a solution for problem (\ref{eq1.1 b}) without
jumps.

Throughout this section let $A:\mathbb{H}\rightrightarrows \mathbb{H}\;$be a
maximal monotone operator and%
\begin{equation*}
m_{t}:=m_{0}+\int_{0}^{t}f_{s}ds,
\end{equation*}%
where $m_{0}\in \overline{\mathrm{D}(A)}$ and $f\in L_{loc}^{1}\left(
\mathbb{R}^{+};\mathbb{H}\right) $.

We consider the next Cauchy problem%
\begin{equation}
\left\{
\begin{array}{l}
dx_{t}+A\left( x_{t}\right) \left( dt\right) \ni dm_{t}\;,\quad \text{a.e.}%
\;t>0,\medskip  \\
x_{0}=m_{0}\,.%
\end{array}%
\right.   \label{Cauchy problem 1}
\end{equation}

\begin{definition}
A continuous function $x:\mathbb{R}^{+}\rightarrow \mathbb{H}$ is a strong
solution of (\ref{Cauchy problem 1}) if:%
\begin{equation*}
\begin{array}{rl}
\left( i\right)  & x_{t}\in \mathrm{D}\left( A\right) ,~a.e.\;t>0,\medskip
\\
\left( ii\right)  & \exists h\in L_{loc}^{1}\left( \mathbb{R}^{+};\mathbb{H}%
\right) \text{ such that }h_{t}\in A\left( x_{t}\right) \text{,}\quad \text{%
a.e. }t>0,\text{ and}\medskip  \\
& \quad \displaystyle x_{t}+\int_{0}^{t}h_{s}ds=m_{t}\,,\;\forall ~t\geq 0,%
\end{array}%
\end{equation*}%
and we write $x=\mathcal{S}\left( A,m_{0};f\right) .$
\end{definition}

We introduce the notation%
\begin{equation*}
\begin{array}{l}
W_{loc}^{1,p}\left( \mathbb{R}^{+};\mathbb{H}\right) =\Big\{f:\exists ~a\in%
\mathbb{H},\;g\in L_{loc}^{p}\left( \mathbb{R}^{+};\mathbb{H}\right) \text{
such that}\medskip \\
\quad\quad\quad\quad\quad\quad\quad\quad\quad\quad\quad\quad\quad \quad%
\displaystyle f_{t}=a+\int_{0}^{t}g_{s}ds,\;\forall~t\geq0\Big\}.%
\end{array}%
\end{equation*}

The following three results are recall from Barbu \cite{ba/84} and Brezis
\cite{br/73} respectively. Let us first denote by $A_{\varepsilon }\;$the
Yosida approximation of the operator $A$, i.e.%
\begin{equation*}
A_{\varepsilon }\left( x\right) :=\frac{1}{\varepsilon }\big(%
x-J_{\varepsilon }\left( x\right) \big),\quad \text{where }J_{\varepsilon
}\left( x\right) :=\left( I+\varepsilon A\right) ^{-1}\left( x\right) .
\end{equation*}%
We mention that $J_{\varepsilon },A_{\varepsilon }:\mathbb{H}\rightarrow
\mathbb{H}$ are single valued operators and Lipschitz continuous functions, $%
A_{\varepsilon }$ is a maximal monotone operator operator with $\mathrm{D}%
\left( A_{\varepsilon }\right) =\mathbb{H}$, $\left\vert A_{\varepsilon
}\left( u\right) \right\vert \leq \left\vert v\right\vert $, for all $\left(
u,v\right) \in \mathrm{Gr}\left( A\right) $ and%
\begin{equation*}
\left\vert A_{\varepsilon }\left( x\right) -A_{\varepsilon }\left( y\right)
\right\vert \leq \frac{1}{\varepsilon }\left\vert x-y\right\vert ,\quad
\forall x,y\in \mathbb{H},\;\forall \varepsilon >0.
\end{equation*}%
Hence for any $T>0$ and any $\nu \in C\left( \left[ 0,T\right] ;\mathbb{H}%
\right) $ the approximating equation%
\begin{equation}
x_{t}^{\varepsilon }+\int_{0}^{t}A_{\varepsilon }\left( x_{s}^{\varepsilon
}\right) ds=\nu _{t}\;,\quad t\in \left[ 0,T\right] ,  \label{approxim eq}
\end{equation}%
admits a unique solution $x^{\varepsilon }\in C\left( \left[ 0,T\right] ;%
\mathbb{H}\right) .$

Moreover, if $m_{t}=m_{0}+\int_{0}^{t}f_{s}ds$ and $\hat{m}_{t}=\hat{m}%
_{0}+\int_{0}^{t}\hat{f}_{s}ds$ for $f,\hat{f}\in L_{1}\left( \left[ 0,T%
\right] ;\mathbb{H}\right) $, then, by the monotonicity of $A_{\varepsilon }$%
, we have%
\begin{equation*}
\left\vert x_{t}^{\varepsilon }-u\right\vert \leq \left\vert
m_{0}-u\right\vert +\int_{0}^{t}\left\vert f_{s}-A_{\varepsilon }\left(
u\right) \right\vert ds\,,\quad \forall u\in \mathbb{H}
\end{equation*}%
and%
\begin{equation*}
\left\vert x_{t}^{\varepsilon }-\hat{x}_{t}^{\varepsilon }\right\vert \leq
\left\vert m_{0}-\hat{m}_{0}\right\vert +\int_{0}^{t}|f_{s}-\hat{f}_{s}|ds\,.
\end{equation*}

\begin{proposition}[{see \protect\cite[Theorem 1.12]{ba/84}}]
\label{prop Barbu/84}Let $A$\ be a maximal monotone operator on \thinspace $%
\mathbb{H}$ and $m_{0}\in \mathrm{D}\left( A\right) $. If $f\in
W_{loc}^{1,1}\left( \mathbb{R}^{+};\mathbb{H}\right) ,$ then the Cauchy
problem (\ref{Cauchy problem 1}) has a unique strong solution $x\in
W_{loc}^{1,\infty }\left( \mathbb{R}^{+};\mathbb{H}\right) $ and if $%
x^{\varepsilon }$ is the solution of the approximating problem (\ref%
{approxim eq}), then for all $T\in \mathbb{R}^{+}$,%
\begin{equation*}
\lim_{\varepsilon \searrow 0}\left\Vert x^{\varepsilon }-x\right\Vert _{T}=0.
\end{equation*}%
Moreover, if $m_{0},\hat{m}_{0}\in \mathrm{D}(A)$, $f,\hat{f}\in
W_{loc}^{1,1}\left( \mathbb{R}^{+};\mathbb{H}\right) $ and $x,\hat{x}$ are
the corresponding strong solutions of (\ref{Cauchy problem 1}), then for all
$t\in \mathbb{R}^{+}$%
\begin{equation}
\left\vert x\left( t\right) -\hat{x}\left( t\right) \right\vert \leq
\left\vert m_{0}-\hat{m}_{0}\right\vert +\int_{0}^{t}|f_{s}-\hat{f}_{s}|ds.
\label{continuity}
\end{equation}%
In particular, for all $\left( u_{0},v_{0}\right) \in \mathrm{Gr}\left(
A\right) $ and $t\in \mathbb{R}^{+},$%
\begin{equation}
\left\vert x\left( t\right) -u_{0}\right\vert \leq \left\vert
m_{0}-u_{0}\right\vert +\int_{0}^{t}\left\vert f_{s}-v_{0}\right\vert ds.
\label{bound}
\end{equation}
\end{proposition}

\begin{proposition}[{see \protect\cite[Corollaire 3.2]{br/73}}]
\label{prop Introduction 1}Let $A$\ be a maximal monotone operator on
\thinspace $\mathbb{H}$ such that $\mathrm{Int}\left( \mathrm{D}\left(
A\right) \right) \neq \emptyset .$\ If $m_{0}\in \overline{\mathrm{D}\left(
A\right) }$\ and $f\in W_{loc}^{1,1}\left( \mathbb{R}^{+};\mathbb{H}\right) $%
, then the Cauchy problem (\ref{Cauchy problem 1}) has a unique strong
solution $x\in W_{loc}^{1,1}\left( \mathbb{R}^{+};\mathbb{H}\right) $ and
inequalities (\ref{continuity}) and (\ref{bound}) hold true for any $m_{0},%
\hat{m}_{0}\in \overline{\mathrm{D}(A)}.$
\end{proposition}

By continuity property (\ref{continuity}) we can generalize the notion of
solution for equation (\ref{Cauchy problem 1}) as follows:

\begin{definition}
The function $x:\mathbb{R}^{+}\rightarrow \mathbb{H}$ is a generalized
solution of the Cauchy problem (\ref{Cauchy problem 1}) with $m_{0}\in
\overline{\mathrm{D}\left( A\right) },~f\in L_{loc}^{1}\left( \mathbb{R}^{+};%
\mathbb{H}\right) $ (and we write $x=\mathcal{GS}\left( A,m_{0};f\right) $)
if for all $T>0$,

\noindent $\left( \alpha \right) \quad x\in C\left( \left[ 0,T\right] ;%
\mathbb{H}\right) $ and

\noindent $\left( \beta \right) \quad $there exist $m_{0}^{n}\in \mathrm{D}%
\left( A\right) $, $f^{n}\in W^{1,1}\left( \left[ 0,T\right] ;\mathbb{H}%
\right) $ such that%
\begin{equation*}
\begin{array}{ll}
\left( a\right)  & m_{0}^{n}\rightarrow m_{0}\quad \text{in }\mathbb{H}%
,\medskip  \\
\left( b\right)  & f^{n}\rightarrow f\quad \text{in }L^{1}\left( 0,T;\mathbb{%
H}\right) ,\medskip  \\
\left( c\right)  & x^{n}=\mathcal{S}\left( A,m_{0}^{n};f^{n}\right)
\rightarrow x\quad \text{in\ }C\left( \left[ 0,T\right] ;\mathbb{H}\right) .%
\end{array}%
\end{equation*}
\end{definition}

We have:

\begin{proposition}[{see \protect\cite[Theoreme 3.4]{br/73}}]
If $A$\ is maximal monotone operator on \thinspace $\mathbb{H},$\ $m_{0}\in
\overline{\mathrm{D}\left( A\right) }$\ and $f\in L_{loc}^{1}\left( \mathbb{R%
}^{+};\mathbb{H}\right) ,$\ then the Cauchy problem (\ref{Cauchy problem 1})
has a unique generalized solution $x:\mathbb{R}^{+}\rightarrow \mathbb{H}$.\
If $x=\mathcal{GS}\left( A,m_{0};f\right) $ and $y=\mathcal{GS}(A,\hat{m}%
_{0};\hat{f})$ then%
\begin{equation}
\left\vert x_{t}-y_{t}\right\vert \leq \left\vert m_{0}-\hat{m}%
_{0}\right\vert +\int_{0}^{t}|f_{s}-\hat{f}_{s}|ds,~t\geq 0.
\label{cont-gen}
\end{equation}%
In particular, for all $\left( u_{0},v_{0}\right) \in \mathrm{Gr}\left(
A\right) $ and for all $T\in \mathbb{R}^{+},$%
\begin{equation*}
\left\Vert x-u_{0}\right\Vert _{T}\leq \left\vert m_{0}-u_{0}\right\vert
+\left\Vert f-v_{0}\right\Vert _{L^{1}(0,T;\mathbb{H})}
\end{equation*}%
and therefore%
\begin{equation}
\left\Vert x\right\Vert _{T}\leq \left( \left\vert u_{0}\right\vert
+\left\vert m_{0}-u_{0}\right\vert +\left\vert v_{0}\right\vert T\right)
+\left\Vert f\right\Vert _{L^{1}(0,T;\mathbb{H})}\,.  \label{bound-gen}
\end{equation}
\end{proposition}

In the case when $\mathrm{Int}\left( \mathrm{D}\left( A\right) \right) \neq
\emptyset $ one can give an equivalent formulation for the generalized
solutions which is analogous to the Skorohod problem.

\begin{proposition}
\label{prop 4}Let $A:\mathbb{H}\rightrightarrows\mathbb{H}$ be a maximal
monotone operator such that $\mathrm{Int}\left( \mathrm{D}\left( A\right)
\right) \neq\emptyset$, $m_{0}\in\overline{\mathrm{D}\left( A\right) }$ and $%
f\in L_{loc}^{1}\left( \mathbb{R}^{+};\mathbb{H}\right) $. Then:

\noindent$1.$~there exists a unique couple of continuous functions $(x,k):%
\mathbb{R}^{+}\rightarrow\mathbb{H\times H}$ such that%
\begin{equation}
\left( SP\right) ~~%
\begin{array}{ll}
\left( a\right) & x(t)\in\overline{\mathrm{D}\left( A\right) }\,,\;\forall
t\geq0,\medskip \\
\left( b\right) & k\in\mathrm{BV}_{loc}(\mathbb{R}^{+};\mathbb{H}%
),\;k(0)=0,\medskip \\
\left( c\right) & x_{t}+k_{t}=m_{t},\;\forall t\geq0,\medskip \\
\left( d\right) & \displaystyle\int_{s}^{t}\left\langle x_{r}-\alpha
,dk_{r}-\beta dr\right\rangle \geq0,~\forall\,0\leq s\leq t,\;\forall\left(
\alpha,\beta\right) \in\mathrm{Gr}\left( A\right) .%
\end{array}
\label{SP}
\end{equation}
\noindent$2.$~$x=\mathcal{GS}\left( A,m_{0};f\right) $ if and only if $x$ is
solution of the problem $\left( SP\right) .$

\noindent$3.$~the following estimate holds: for all $T>0$,%
\begin{equation*}
\left\Vert x\right\Vert _{T}^{2}+\left\updownarrow {k}\right\updownarrow {%
_{T}}\leq C_{T}\left( 1+\left\vert m_{0}\right\vert ^{2}+\left\Vert
f\right\Vert _{L^{1}(0,T;\mathbb{H})}^{2}\right) ,
\end{equation*}
where $C_{T}$ is a positive constant independent of $m_{0}$ and $f.$
\end{proposition}

The problem $\left( SP\right) $ will be called generalized Skorohod problem
associated to the maximal monotone operator $A.$

\begin{proposition}[{see Barbu \protect\cite[Chapter III, Section 1]{ba/76}}]

If $m_{t}\equiv m_{0}\in \overline{\mathrm{D}\left( A\right) }$ then the
solution of (\ref{SP}) is given by%
\begin{equation*}
x_{t}=S_{A}\left( t\right) m_{0},
\end{equation*}%
where $\left\{ S_{A}\left( t\right) \right\} _{t\geq 0}$ is the nonlinear
semigroup of contractions generated by $A:$%
\begin{equation*}
S_{A}\left( t\right) y=\lim_{n\rightarrow \infty }\left( I+\frac{t}{n}%
A\right) ^{-1}y,\;\;y\in \overline{\mathrm{D}\left( A\right) }
\end{equation*}%
and the limit is uniform with respect to $t\in \left[ 0,T\right] $, $\forall
T>0.$
\end{proposition}

\begin{proposition}[{see Brezis \protect\cite[Lemma 3.8]{br/73}}]
\label{prop 5}Let $\alpha ^{\varepsilon },\alpha \in \mathbb{H}$ such that $%
\alpha ^{\varepsilon }\rightarrow \alpha $, as $\varepsilon \rightarrow 0$.
If $\alpha \in \overline{\mathrm{D}\left( A\right) },$ then for all $T>0,$%
\begin{equation*}
\lim_{\varepsilon \rightarrow 0}\sup_{t\in \left[ 0,T\right] }\left\vert
S_{A_{\varepsilon }}\left( t\right) \alpha ^{\varepsilon }-S_{A}\left(
t\right) \alpha \right\vert =0,
\end{equation*}%
and, if $\alpha \notin \overline{\mathrm{D}\left( A\right) },$ then%
\begin{equation*}
\lim_{\varepsilon \rightarrow 0}\int_{0}^{T}\left\vert S_{A_{\varepsilon
}}\left( t\right) \alpha ^{\varepsilon }-S_{A}\left( t\right) \Pi _{%
\overline{\mathrm{D}\left( A\right) }}\left( \alpha \right) \right\vert
^{2}dt=0.
\end{equation*}
\end{proposition}

We mention that the classical Skorohod problem corresponds to the case $%
A=\partial I_{\bar{D}}(x)$ where $I_{\bar{D}}$ is the indicator function of
a convex and closed $\bar{D}\subset \mathbb{H}$, i.e. $I_{\bar{D}%
}(x)=\left\{
\begin{array}{rl}
0, & \text{if \ }x\in \bar{D},\medskip \\
+\infty , & \text{if \ }x\notin \bar{D}.%
\end{array}%
\right. $ In this case%
\begin{equation*}
A=\left\{
\begin{array}{ll}
0, & \text{if \ }x\in D,\medskip \\
\mathcal{N}_{\bar{D}}\left( x\right) , & \text{if \ }x\in \mathrm{Bd}\left(
D\right) ,\medskip \\
\emptyset , & \text{if \ }x\in \mathbb{R}^{d}\backslash \bar{D},%
\end{array}%
\right.
\end{equation*}%
where $\mathcal{N}_{\bar{D}}\left( x\right) $ is the outward normal cone to $%
\bar{D}$ at $x\in \mathrm{Bd}\left( D\right) $. In this case the term $%
-\partial I_{\bar{D}}(x_{t})\left( dt\right) $, which is added to the input $%
dm_{t}$, acts, in a minimal way, as an \textit{inward push}\ that prevents $%
x_{t}$ from exiting the domain $\bar{D}$.

The following result can be found in R\u{a}\c{s}canu \cite{ra/96}.

\begin{theorem}
If $A:\mathbb{H}\rightrightarrows\mathbb{H}$ is a maximal monotone operator
such that $\mathrm{Int}\left( \mathrm{D}\left( A\right) \right)
\neq\emptyset $ and $m:\mathbb{R}^{+}\rightarrow\mathbb{H}$ is a continuous
function with $m_{0}\in\overline{\mathrm{D}\left( A\right) }$, then the
convex Skorohod problem (\ref{SP}) has a unique solution $\left( x,k\right) $%
.
\end{theorem}

\section{The Skorokhod problem with maximal monotone operators}

\label{Skorohod problem}

Let $\mathbb{H}$ be Hilbert space with the inner product $\left\langle \cdot
,\cdot \right\rangle $ and the norm induced $\left\vert \cdot \right\vert $.
Let $A:\mathbb{H}\rightrightarrows \mathbb{H}$ be a maximal monotone
operator (notation $\rightrightarrows $ means that $A$ is a multivalued
operator). We denote by $\mathrm{D}\left( A\right) $ the domain of $A$, i.e.
$\mathrm{D}\left( A\right) :=\left\{ x\in \mathbb{H}:\,Ax\neq \emptyset
\right\} .$

\begin{definition}
\label{gener proj}A mapping $\Pi:\mathbb{H}\rightarrow\mathbb{H}$ is a
generalized projection on $\overline{\mathrm{D}(A)}$ if

\begin{enumerate}
\item[$\left( i\right) $] $\Pi\left( \mathbb{H}\right) \subset \overline{%
\mathrm{D}(A)}$ and $\Pi(x)=x,$ for all $x\in\overline {\mathrm{D}(A)}$%
\medskip

\item[$\left( ii\right) $] $|\Pi(x)-\Pi(y)|\leq|x-y|,$ for all $x,y\in%
\mathbb{H}$ (i.e. $\Pi$ is a nonexpansive map).
\end{enumerate}
\end{definition}

The (classical) orthogonal projection $\Pi _{\overline{\mathrm{D}(A)}}:%
\mathbb{H}\rightarrow \overline{\mathrm{D}(A)}$ is defined by: for all $x\in
\mathbb{H}$
\begin{equation}
|x-\Pi _{\overline{\mathrm{D}(A)}}\left( x\right) |=\inf \big\{\left\vert
x-a\right\vert :a\in \overline{\mathrm{D}(A)}\big\}  \label{eq2.1a}
\end{equation}%
$(\Pi _{\overline{\mathrm{D}(A)}}\left( x\right) $ is well defined by (\ref%
{eq2.1a}) since $\overline{\mathrm{D}\left( A\right) }$ is a closed convex
set; see Proposition \ref{prop Anex 2} from the Annexes).

It is easily to prove that $\hat{x}=\Pi _{\overline{\mathrm{D}(A)}}(x)$ if
and only%
\begin{equation}
\hat{x}\in \overline{\mathrm{D}(A)}\quad \text{and}\quad \left\langle \hat{x}%
-x,\hat{x}-a\right\rangle \leq 0,\text{ for all }a\in \overline{\mathrm{D}(A)%
}.  \label{eq2.2'}
\end{equation}%
Hence for all $x,y\in \mathbb{H}$ we have%
\begin{equation*}
\langle x-\Pi _{\overline{\mathrm{D}(A)}}(x),\Pi _{\overline{\mathrm{D}(A)}%
}(y)-\Pi _{\overline{\mathrm{D}(A)}}(x)\rangle \leq 0\quad \text{and}\quad
\langle y-\Pi _{\overline{\mathrm{D}(A)}}(y),\Pi _{\overline{\mathrm{D}(A)}%
}(x)-\Pi _{\overline{\mathrm{D}(A)}}(y)\rangle \leq 0
\end{equation*}%
which yields%
\begin{equation}
|\Pi _{\overline{\mathrm{D}(A)}}(x)-\Pi _{\overline{\mathrm{D}(A)}%
}(y)|^{2}\leq \langle \Pi _{\overline{\mathrm{D}(A)}}(x)-\Pi _{\overline{%
\mathrm{D}(A)}}(y),x-y\rangle ,\quad \text{ for all }x,y\in \mathbb{H}.
\label{eq2.3}
\end{equation}%
If follows that $\Pi _{\overline{\mathrm{D}(A)}}$ is a generalized
projection on $\overline{\mathrm{D}(A)}$.

There exist important examples of projections on $\overline{\mathrm{D}(A)}$
connected with the so-called \textit{elasticity condition} introduced in
one-dimensional case in \cite{ch-ka/80} and \cite{sl-wo/10} (see also \cite%
{si/00}).

Let $\delta \in \lbrack 0,1]$. We will consider the map $\Pi ^{\delta }:%
\mathbb{H}\rightarrow \mathbb{H}$ defined by
\begin{equation*}
\Pi ^{\delta }(z)=\Pi _{\overline{\mathrm{D}(A)}}(z)-\delta \big(z-\Pi _{%
\overline{\mathrm{D}(A)}}(z)\big),\quad z\in \mathbb{H}
\end{equation*}%
and its compositions $\Pi ^{\delta ,n}:\mathbb{H}\rightarrow \mathbb{H}$ of
the form%
\begin{equation*}
\Pi ^{\delta ,n}(z)=\Pi _{n}\circ ...\circ \Pi _{1}(z),\quad z\in \mathbb{H}%
,\;\text{where }\Pi _{1}=...=\Pi _{n}=\Pi ^{\delta },\,\,n\in \mathbb{N}%
^{\ast }.
\end{equation*}

\begin{proposition}
\label{prop2.2}\ \

\begin{enumerate}
\item[$\left( j\right) $] For any $z\in \mathbb{H}$, $a\in \mathrm{Int}%
\left( \mathrm{D}\left( A\right) \right) $ and $\overline{B\left(
a,r_{0}\right) }\subset \overline{\mathrm{D}\left( A\right) }$ with $%
r_{0}>0, $%
\begin{equation*}
|\Pi ^{\delta }(z)-a|^{2}+\left( 1-\delta ^{2}\right) |z-\Pi _{\overline{%
\mathrm{D}(A)}}(z)|^{2}+2r_{0}(1+\delta )|z-\Pi _{\overline{\mathrm{D}(A)}%
}(z)|\;\leq \;|z-a|^{2}.
\end{equation*}

\item[$\left( jj\right) $] The map $\Pi ^{\delta ,n}:\mathbb{H}\rightarrow
\mathbb{H}$ is a nonexpansive map and $\Pi ^{\delta ,n}\left( z\right) =z$,
for any $z\in \overline{\mathrm{D}(A)}$.\medskip

\item[$\left( jjj\right) $] For any $z\in \mathbb{H}$ there exists the limit
$\displaystyle{\lim_{n\rightarrow \infty }\Pi ^{\delta ,n}(z)}$ and $\Pi _{%
\overline{\mathrm{D}(A)}}^{\delta }\left( z\right) :={\lim_{n\rightarrow
\infty }\Pi ^{\delta ,n}(z)}$ is a generalized projection on $\overline{%
\mathrm{D}(A)}.$
\end{enumerate}
\end{proposition}

\begin{proof}
$(j)$ Clearly,%
\begin{equation*}
\begin{array}{l}
|\Pi ^{\delta }(z)-a|^{2}=|(1+\delta )(\Pi _{\overline{\mathrm{D}(A)}%
}(z)-z)+(z-a)|^{2}\medskip \\
=(1+\delta )^{2}|\Pi _{\overline{\mathrm{D}(A)}}(z)-z|^{2}+|z-a|^{2}+2(1+%
\delta )\langle \Pi _{\overline{\mathrm{D}(A)}}(z)-z,z-a\rangle .%
\end{array}%
\end{equation*}%
Since $\overline{B\left( a,r_{0}\right) }\subset \overline{\mathrm{D}\left(
A\right) }$, from inequality (\ref{prop Anex 5_3}) from the Annexes, we have
for all $z\in \mathbb{H}$,%
\begin{equation*}
\langle \Pi _{\overline{\mathrm{D}(A)}}(z)-z,z-a\rangle \leq -r_{0}|z-\Pi _{%
\overline{D\left( A\right) }}\left( z\right) |-|z-\Pi _{\overline{D\left(
A\right) }}\left( z\right) |^{2}.
\end{equation*}%
Consequently,%
\begin{equation*}
|\Pi ^{\delta }(z)-a|^{2}\leq \big((1+\delta )^{2}-2(1+\delta )\big)|\Pi _{%
\overline{\mathrm{D}(A)}}(z)-z|^{2}+|z-a|^{2}-2(1+\delta )r_{0}|\Pi _{%
\overline{\mathrm{D}(A)}}(z)-z|,
\end{equation*}%
which implies the conclusion of $(j)$.\medskip

\noindent $(jj)$ By (\ref{eq2.3}) and Lipschitz property of $\Pi $ we see
that, for all $x,y\in \mathbb{H},$%
\begin{equation*}
\begin{array}{l}
|\Pi ^{\delta }(x)-\Pi ^{\delta }(y)|^{2}=|(1+\delta )(\Pi _{\overline{%
\mathrm{D}(A)}}(x)-\Pi _{\overline{\mathrm{D}(A)}}(y))-\delta \left(
x-y\right) |^{2}\medskip \\
=(1+\delta )^{2}|\Pi _{\overline{\mathrm{D}(A)}}(x)-\Pi _{\overline{\mathrm{D%
}(A)}}(y)|^{2}+\delta ^{2}|x-y|^{2}-2\delta (1+\delta )\langle \Pi _{%
\overline{\mathrm{D}(A)}}(x)-\Pi _{\overline{\mathrm{D}(A)}}(y),x-y\rangle
\medskip \\
\leq \big((1+\delta )^{2}-2\delta (1+\delta )\big)|\Pi _{\overline{\mathrm{D}%
(A)}}(x)-\Pi _{\overline{\mathrm{D}(A)}}(y)|^{2}+\delta ^{2}|x-y|^{2}\leq
|x-y|^{2}.%
\end{array}%
\end{equation*}%
Hence for all $n\in \mathbb{N}$ and $x,y\in \mathbb{H}$
\begin{align*}
|\Pi ^{\delta ,n}(x)-\Pi ^{\delta ,n}(y)|& =|{\Pi }_{n}\circ ...\circ {\Pi }%
_{1}(x)-{\Pi }_{n}\circ ...\circ {\Pi }_{1}(y)| \\
& \leq |{\Pi }_{n-1}\circ ...\circ {\Pi }_{1}(x)-{\Pi }_{n-1}\circ ...\circ {%
\Pi }_{1}(y)| \\
& \leq |x-y|.
\end{align*}%
$(jjj)$ Set $z_{0}=z$, $z_{n}=\Pi ^{n,\delta }(z)$, $n\in \mathbb{N}$.
Clearly, $z_{n}=\Pi ^{\delta }(z_{n-1})$, $n\in \mathbb{N}$. Fix $a\in
\mathrm{Int}\left( \mathrm{D}(A)\right) $ and $r_{0}>0$ such that $\overline{%
B\left( a,r_{0}\right) }\subset \overline{\mathrm{D}\left( A\right) }$. By $%
(j),$ for any $i\in \mathbb{N},$%
\begin{equation*}
|z_{i}-a|^{2}+2r_{0}(1+\delta )|\Pi _{\overline{\mathrm{D}(A)}%
}(z_{i-1})-z_{i-1}|\leq |z_{i-1}-a|^{2}
\end{equation*}%
and consequently%
\begin{equation*}
|z_{n}-a|^{2}+2r_{0}(1+\delta )\sum_{i=0}^{n-1}|\Pi _{\overline{\mathrm{D}(A)%
}}(z_{i})-z_{i}|\leq |z_{0}-a|^{2}.
\end{equation*}%
Hence%
\begin{equation*}
\sum_{i=0}^{\infty }|\Pi _{\overline{\mathrm{D}(A)}}(z_{i})-z_{i}|\leq \frac{%
1}{2r_{0}(1+\delta )}|z_{0}-a|^{2}.
\end{equation*}%
By $(jj),$ for any $n\geq m,$%
\begin{align*}
|z_{n}-z_{m}|& \leq \sum_{k=m+1}^{n}|z_{k}-z_{k-1}|=\sum_{k=m+1}^{n}|\Pi
^{\delta }(z_{k-1})-z_{k-1}| \\
& \leq \sum_{k=m+1}^{n}\big(|\Pi ^{\delta }(z_{k-1})-\Pi _{\overline{\mathrm{%
D}(A)}}(z_{k-1})|+|\Pi _{\overline{\mathrm{D}(A)}}(z_{k-1})-z_{k-1}|\big) \\
& \leq 2\sum_{k=m+1}^{n}|\Pi _{\overline{\mathrm{D}(A)}}(z_{k-1})-z_{k-1}|.
\end{align*}%
Consequently $\{z_{n}\}$ is a Cauchy sequence. Hence there exists the limit $%
z_{\infty }=\lim_{n\rightarrow \infty }z_{n}$. Since
\begin{equation*}
\lim_{n\rightarrow \infty }|\Pi _{\overline{\mathrm{D}(A)}}(z_{n})-z_{n}|=0,
\end{equation*}%
$z_{\infty }\in \overline{\mathrm{D}(A)}$ and the proof is complete.\hfill
\end{proof}

\begin{example}
\label{ex2.3}\noindent $(a)$ If $\overline{\mathrm{D}(A)}$ satisfies the
interior uniform ball condition, then there exists $n_{0}=n_{0}\left(
z\right) \in \mathbb{N}$ such that $\Pi ^{\delta ,n_{0}}(z)\in \overline{%
\mathrm{D}(A)}$ and, in this case,
\begin{equation*}
\Pi _{\overline{\mathrm{D}(A)}}^{\delta }(z)=\left\{
\begin{array}{ll}
z, & \text{if }\,z\in \overline{\mathrm{D}(A)},\medskip \\
\Pi _{n_{0}}\circ ...\circ \Pi _{1}(z), & \text{otherwise,}%
\end{array}%
\right.
\end{equation*}%
where $\Pi _{1}=...=\Pi _{n_{0}}=\Pi ^{\delta }$ and $n_{0}=n_{0}\left(
z\right) =\min \{k:\Pi _{k}\circ ...\circ \Pi _{1}(z)\in \overline{\mathrm{D}%
(A)}\}$.\medskip

We recall first that the set $\overline{\mathrm{D}(A)}\subset \mathbb{H}$
satisfies the interior uniform ball condition if there exists $r_{0}>0$ such
that for any $z\not\in \overline{\mathrm{D}(A)}$,%
\begin{equation*}
B(\Pi _{\overline{\mathrm{D}(A)}}\left( z\right) -r_{0}u_{z},r_{0})\subset
\overline{\mathrm{D}(A)},
\end{equation*}%
where $u_{z}:=\frac{z-\Pi _{\overline{\mathrm{D}(A)}}\left( z\right) }{%
|z-\Pi _{\overline{\mathrm{D}(A)}}\left( z\right) |}\,.$

If $z\in \overline{\mathrm{D}(A)}$ then $n=0$. Let now $z\not\in \overline{%
\mathrm{D}(A)}$ and $u_{z}:=\frac{z-\Pi _{\overline{\mathrm{D}(A)}}\left(
z\right) }{|z-\Pi _{\overline{\mathrm{D}(A)}}\left( z\right) |}\,.$ Since $%
\overline{\mathrm{D}(A)}$ satisfies the $r_{0}-$interior uniform ball
condition,%
\begin{equation*}
\Pi _{\overline{\mathrm{D}(A)}}\left( z\right) -2r_{0}u_{z}\in \bar{B}(\Pi _{%
\overline{\mathrm{D}(A)}}\left( z\right) -r_{0}u_{z},r_{0})\subset \overline{%
\mathrm{D}(A)}.
\end{equation*}%
Let $z_{0}=z$, $z_{1}:=\Pi ^{\delta ,1}\left( z\right) =\Pi _{\overline{%
\mathrm{D}(A)}}\left( z\right) -\delta (z-\Pi _{\overline{\mathrm{D}(A)}%
}\left( z\right) )$. Therefore%
\begin{align*}
|z_{1}-\Pi _{\overline{\mathrm{D}(A)}}\left( z_{1}\right) |& \leq
|z_{1}-(\Pi _{\overline{\mathrm{D}(A)}}\left( z\right)
-2r_{0}u_{z})|=|z_{1}-\Pi _{\overline{\mathrm{D}(A)}}\left( z\right) |-2r_{0}
\\
& =\delta |z-\Pi _{\overline{\mathrm{D}(A)}}\left( z\right) |-2r_{0}
\end{align*}%
and subsequently%
\begin{equation*}
|z_{2}-\Pi _{\overline{\mathrm{D}(A)}}\left( z_{2}\right) |\leq \delta
|z_{1}-\Pi _{\overline{\mathrm{D}(A)}}\left( z_{1}\right) |-2r_{0}\leq
\delta ^{2}|z-\Pi _{\overline{\mathrm{D}(A)}}\left( z\right) |-2\delta
r_{0}-2r_{0}.
\end{equation*}%
Obviously,%
\begin{equation*}
|z_{n}-\Pi _{\overline{\mathrm{D}(A)}}\left( z_{n}\right) |\leq \delta
^{n}|z-\Pi _{\overline{\mathrm{D}(A)}}\left( z\right) |-2r_{0}\left(
1+\delta +\cdots +\delta ^{n-1}\right) =:\alpha _{n}
\end{equation*}%
Now we should consider two cases $\delta =1$ and $\delta \in \left(
0,1\right) $. In the both cases it is easy to check that there exists $%
n_{0}=n_{0}\left( z\right) \in \mathbb{N}^{\ast }$ such that%
\begin{equation*}
\alpha _{n_{0}-1}>0\text{ and }\alpha _{n_{0}}\leq 0,
\end{equation*}%
which yields $z_{n_{0}}=\Pi ^{\delta ,n_{0}}(z)\in \overline{\mathrm{D}(A)}$
and consequently $z_{n}=z_{n_{0}}\in \overline{\mathrm{D}(A)}\,,$ for all $%
n\geq n_{0}.$ Therefore $\Pi _{\overline{\mathrm{D}(A)}}^{\delta }\left(
z\right) :={z}_{n_{0}}\,.$\medskip

\noindent$(b)$ If $\overline{\mathrm{D}(A)}=\overline{B(a,R)}$ , then $%
\overline{\mathrm{D}(A)}$ satisfies the interior uniform ball condition with
$r_{0}=R$ and the same conclusions as in $\left( a\right) $ follows.
\end{example}

In the following lemma we collect basic properties of a generalized
projections $\Pi :\mathbb{H}\rightarrow $ $\overline{\mathrm{D}(A)}\,.$

\begin{lemma}
\label{lem2.4}Let $\Pi:\mathbb{H}\rightarrow\overline{\mathrm{D}(A)}$ be a
generalized projection on $\overline{\mathrm{D}(A)}.$ Then

\begin{enumerate}
\item[$\left( j\right) $] for all $x,y\in\mathbb{H}$
\begin{equation*}
\langle\Pi(x)-\Pi(y),\Pi(x)-x-\Pi(y)+y\rangle\leq\frac{1}{2}|\Pi
(x)-x-\Pi(y)+y|^{2};
\end{equation*}

\item[$\left( jj\right) $] for all $x\in\mathbb{H}$ and $a\in\overline {%
\mathrm{D}(A)}$%
\begin{equation*}
\langle\Pi(x)-a,\Pi(x)-x\rangle\leq\frac{1}{2}|\Pi(x)-x|^{2};
\end{equation*}

\item[$\left( jjj\right) $] if $x\in\mathbb{H}$ and $\overline{B\left(
a,r_{0}\right) }\subset\overline{\mathrm{D}(A)},$ with $r_{0}>0:$%
\begin{equation*}
r_{0}|\Pi(x)-x|\leq\langle a-\Pi(x),\Pi(x)-x\rangle+\frac{1}{2}|\Pi
(x)-x|^{2}.
\end{equation*}
\end{enumerate}
\end{lemma}

\begin{proof}
$(j)$ Indeed for all $x,y\in \mathbb{H}$%
\begin{equation*}
\begin{array}{l}
|\Pi (x)-x-\Pi (y)+y|^{2}-2\langle \Pi (x)-\Pi (y),\Pi (x)-x-\Pi
(y)+y\rangle \medskip \\
=|\Pi (x)-\Pi (y)|^{2}+|x-y|^{2}\medskip \\
\quad +2\langle \Pi (x)-\Pi (y),y-x\rangle -2|\Pi (x)-\Pi (y)|^{2}-2\langle
\Pi (x)-\Pi (y),y-x\rangle \medskip \\
=|x-y|^{2}-|\Pi (x)-\Pi (y)|^{2}\geq 0.%
\end{array}%
\end{equation*}%
\noindent $(jj)$ We take $y=a$ in $(j)$.\medskip

\noindent $(jjj)$ In $(jj)$ we replace $a$ by $a-r_{0}\frac{\Pi (x)-x}{|\Pi
(x)-x|}\,.$\hfill \medskip
\end{proof}

The aim of this section is to prove the existence and uniqueness of a
solution for a Cauchy problem driven by a maximal monotone operator $A$ and
with singular input $dm_{t}$, formally written as:%
\begin{equation}
\left\{
\begin{array}{l}
dx_{t}+A\left( x_{t}\right) \left( dt\right) +dk_{t}^{d}\ni dm_{t}\;,\quad
t\in \mathbb{R}^{+}\,,\medskip \\
x_{0}=m_{0}~.%
\end{array}%
\right.  \label{formal GSP}
\end{equation}%
The basic assumptions are:%
\begin{equation}
\begin{array}{rl}
\left( i\right) & A:\mathbb{H}\rightrightarrows \mathbb{H}\text{ is a
maximal monotone operator,}\medskip \\
\left( ii\right) & \mathrm{Int}\left( \mathrm{D}\left( A\right) \right) \neq
\emptyset \text{,}\medskip \\
\left( iii\right) & m\in \mathbb{D}\left( \mathbb{R}^{+},\mathbb{H}\right)
,~m_{0}\in \overline{\mathrm{D}\left( A\right) }\text{,}\medskip \\
\left( iv\right) & \Pi :\mathbb{H}\rightarrow \mathbb{H}\text{ is a
generalized projection on }\overline{\mathrm{D}(A)}.%
\end{array}
\label{assumpt}
\end{equation}

For $y\in \mathbb{D}\left( \mathbb{R}^{+},\mathbb{H}\right) $ and $\ell \in
\mathbb{D}\left( \mathbb{R}^{+},\mathbb{H}\right) \cap \mathrm{BV}%
_{loc}\left( \mathbb{R}^{+};\mathbb{H}\right) $ we write $d\ell _{t}\in
A\left( y_{t}\right) \left( dt\right) $ if%
\begin{equation*}
\int_{s}^{t}\left\langle y_{r}-\alpha ,\,d\ell _{r}-\beta dr\right\rangle
\geq 0,\quad \forall 0\leq s\leq t,\;\;\forall \,\left( \alpha ,\beta
\right) \in \mathrm{Gr}\left( A\right) .
\end{equation*}

\begin{definition}[Generalized Skorohod problem]
\label{def2.5}We say that function $x\in \mathbb{D}\left( \mathbb{R}^{+},%
\mathbb{H}\right) $ is a solution of equation (\ref{formal GSP}) with its
jumps driven by $\Pi $ if there exists $k\in \mathbb{D}\left( \mathbb{R}^{+},%
\mathbb{H}\right) $ such that:%
\begin{equation}
\begin{array}{rl}
\left( i\right) & x_{t}\in \overline{\mathrm{D}\left( A\right) },\quad
\forall \,t\in \mathbb{R}^{+}\medskip \\
\left( ii\right) & k\in \mathrm{BV}_{loc}\left( \mathbb{R}^{+};\mathbb{H}%
\right) ,\quad k_{0}=0,\medskip \\
\left( iii\right) & x_{t}+k_{t}=m_{t},\quad \forall \text{\ }t\in \mathbb{R}%
^{+},\medskip \\
\left( iv\right) & k=k^{c}+k^{d}\,,\quad k_{t}^{d}=\sum_{0\leq s\leq
t}\Delta k_{s},\medskip \\
\left( v\right) & dk_{t}^{c}\in A\left( x_{t}\right) \left( dt\right)
,\medskip \\
\left( vi\right) & x_{t}=\Pi \left( x_{t-}+\Delta m_{t}\right) ,\quad
\forall \,t\in \mathbb{R}^{+}.%
\end{array}
\label{def sol GSP}
\end{equation}
\end{definition}

We note that formulation (\ref{def sol GSP}) justifies to call the pair $%
\left( x,k\right) $ solution of a generalized Skorokhod problem associated
to $\left( A,\Pi ;m\right) $ and to write $\left( x,k\right) =\mathcal{SP}%
\left( A,\Pi ;m\right) $.

\begin{proposition}
\label{prop-def equiv}An equivalent definition is: $\left( x,k\right) \in%
\mathbb{D}\left( \mathbb{R}^{+},\mathbb{H}\times\mathbb{H}\right) $
satisfies (\ref{def sol GSP}) with $\left( vi\right) $ replaced by%
\begin{equation*}
\begin{array}{rl}
\left( vi^{\prime}\right) & \Delta k_{s}=\left( I-\Pi\right) \left(
x_{s-}+\Delta m_{s}\right)%
\end{array}%
\end{equation*}
\end{proposition}

\begin{proof}
Using the equality $x+k=m$ we infer that $\Delta k_{t}=\left( I-\Pi \right)
\left( x_{t-}+\Delta m_{t}\right) $ is equivalent to $\Delta m_{t}-\Delta
x_{t}=x_{t-}+\Delta m_{t}-\Pi \left( x_{t-}+\Delta m_{t}\right) $, that is $%
x_{t}=\Pi \left( x_{t-}+\Delta m_{t}\right) .$\hfill \medskip
\end{proof}

In the case $\Pi=\Pi_{\overline{\mathrm{D}\left( A\right) }}$ one can give a
simpler equivalent form of Definition \ref{def2.5}.

\begin{proposition}
\label{prop2.14}Let $m\in\mathbb{D}\left( \mathbb{R}^{+},\mathbb{H}\right) $%
, $m_{0}\in\overline{\mathrm{D}\left( A\right) }$. Then $(x,k)=\mathcal{SP}%
(A,\Pi_{\overline{\mathrm{D}\left( A\right) }};m)$ if and only if%
\begin{equation}
\begin{array}{rl}
\left( i\right) & x_{t}\in\overline{\mathrm{D}\left( A\right) },\text{
\thinspace}\forall\,t\in\mathbb{R}^{+}\medskip \\
\left( ii\right) & k\in\mathrm{BV}_{loc}\left( \mathbb{R}^{+};\mathbb{H}%
\right) ,\text{\ }k_{0}=0\text{,}\medskip \\
\left( iii\right) & x_{t}+k_{t}=m_{t},\text{\ }\forall\text{\ }t\in\mathbb{R}%
^{+},\medskip \\
\left( iv\right) & dk_{t}\in A\left( x_{t}\right) \left( dt\right)%
\end{array}
\label{def sol GSP 2}
\end{equation}
\end{proposition}

\begin{proof}
Assume that $(x,k)=\mathcal{SP}(A,\Pi _{\overline{\mathrm{D}\left( A\right) }%
};m)$. Then it holds $x_{t}=m_{t}-k_{t}\in \overline{\mathrm{D}\left(
A\right) }$ and $k$ is a function with locally bounded variation, $k_{0}=0$
such that%
\begin{equation*}
dk_{t}^{c}\in A\left( x_{t}\right) \left( dt\right) .
\end{equation*}%
Since $x_{r}=\Pi _{\overline{\mathrm{D}\left( A\right) }}\left(
x_{r-}+\Delta m_{r}\right) $ and, from Proposition \ref{prop-def equiv}, $%
\Delta k_{r}=\big(I-\Pi _{\overline{\mathrm{D}\left( A\right) }}\big)\left(
x_{r-}+\Delta m_{r}\right) $, we deduce using characterization (\ref{eq2.2'}%
),%
\begin{equation*}
\int_{s}^{t}\langle x_{r}-\alpha ,dk_{r}^{d}\rangle =\sum_{s<r\leq t}\langle
x_{r}-\alpha ,\Delta k_{r}\rangle \geq 0,\quad \forall \alpha \in \overline{%
\mathrm{D}\left( A\right) }.
\end{equation*}%
Consequently, for all $\left( \alpha ,\beta \right) \in \mathrm{Gr}\left(
A\right) ,$%
\begin{equation*}
\int_{s}^{t}\langle x_{r}-\alpha ,dk_{r}-\beta dr\rangle
=\int_{s}^{t}\langle x_{r}-\alpha ,dk_{r}^{c}-\beta dr\rangle
+\int_{s}^{t}\langle x_{r}-\alpha ,dk_{r}^{d}\rangle \geq 0.
\end{equation*}%
Conversely, assume that $(x,k)$ satisfies (\ref{def sol GSP 2}). Set $%
k_{t}^{\left( n\right) }=k_{t}-\sum_{s\leq t}\Delta k_{s}\mathbf{1}_{|\Delta
k_{s}|>1/n}$, $n\in \mathbb{N}$. Clearly, for all $T\geq 0$, $\updownarrow {%
\hspace{-0.1cm}}k^{\left( n\right) }{\hspace{-0.1cm}}\updownarrow _{T}{\leq 2%
}\left\updownarrow {k}\right\updownarrow {_{T}}$ and $||k^{\left( n\right)
}-k^{c}||_{T}\longrightarrow 0$, as $n\rightarrow \infty $. In every
interval $(s,t]$ there exist finite number of $u_{1}<u_{2}<...<u_{m}$ such
that $|\Delta k_{u_{i}}|>1/n$. Set $u_{0}=s$ and $u_{m+1}=t$ and observe
that, for all $0\leq s<t$ and $(\alpha ,\beta )\in \mathrm{Gr}\left(
A\right) $%
\begin{equation*}
\int_{s}^{t}\langle x_{u}-\alpha ,dk_{u}^{\left( n\right) }-\beta
\,du\rangle =\sum_{i=1}^{m+1}\int_{(u_{i-1},u_{i})}\langle x_{u}-\alpha
,dk_{u}-\beta \,du\rangle .
\end{equation*}%
Using next auxiliary result:

\begin{lemma}
Let $y,z\in \mathbb{D}\left( \mathbb{R}^{+},\mathbb{H}\right) $ such that $%
z\in \mathrm{BV}_{loc}\left( \mathbb{R}^{+};\mathbb{H}\right) $. Then the
following conditions are equivalent:%
\begin{equation*}
\begin{array}{rl}
\left( i\right) & \displaystyle\int_{(s,t]}\left\langle
y_{r},dz_{r}\right\rangle \geq 0,\quad \forall \;0\leq s<t,\medskip \\
\left( ii\right) & \displaystyle\int_{(s,t)}\left\langle
y_{r},dz_{r}\right\rangle \geq 0,\quad \forall \;0\leq s<t,\medskip \\
\left( iii\right) & \displaystyle\int_{\left[ s,t\right] }\left\langle
y_{r},dz_{r}\right\rangle \geq 0,\quad \forall \;0\leq s\leq t,%
\end{array}%
\end{equation*}
\end{lemma}

we deduce that%
\begin{equation}
\int_{s}^{t}\langle x_{r}-\alpha ,dk_{r}^{\left( n\right) }-\beta
\,dr\rangle \geq 0,\quad \forall 0\leq s<t,\quad \forall (\alpha ,\beta )\in
\mathrm{Gr}\left( A\right) .  \label{eq2.20}
\end{equation}%
and letting $n\rightarrow \infty $ we obtain $dk_{t}^{c}\in A\left(
x_{t}\right) \left( dt\right) $. For any $t\in \mathbb{R}^{+}$ we have $%
x_{t}\in \overline{\mathrm{D}\left( A\right) }$ and for all $\left( \alpha
,\beta \right) \in \mathrm{Gr}\left( A\right) $%
\begin{equation*}
0\leq \int_{\left\{ t\right\} }\langle x_{r}-\alpha ,dk_{r}-\beta
\,dr\rangle =\langle x_{t}-\alpha ,\Delta k_{t}\rangle =\langle x_{t}-\alpha
,x_{t-}+\Delta m_{t}-x_{t}\rangle .
\end{equation*}%
Then by (\ref{eq2.2'}) it follows $x_{t}=\Pi _{\overline{\mathrm{D}\left(
A\right) }}(x_{t-}+\Delta m_{t})$ which completes the proof.\hfill
\end{proof}

\begin{remark}
\label{rem2.6}Let $\left( x,k\right) =\mathcal{SP}\left( A,\Pi ;m\right) .$
We have
\begin{equation*}
\left\vert \Delta x_{t}\right\vert \leq |\Delta m_{t}|\quad \text{and}\quad
|\Delta k_{t}|\leq 2|\Delta m_{t}|,\quad t\in \mathbb{R}^{+},
\end{equation*}%
since $x_{t-}\in \mathrm{D}\left( A\right) $ and%
\begin{align*}
\Delta x_{t}& =\Pi \left( x_{t-}+\Delta m_{t}\right) -\Pi \left(
x_{t-}\right) \\
\Delta k_{t}& =\Delta m_{t}+\Pi \left( x_{t-}\right) -\Pi (x_{t-}+\Delta
m_{t}).
\end{align*}%
Consequently if $m$ is continuos, then $x$ and $k$ are continuous functions
and independent of $\Pi $ ($k^{d}=0$).
\end{remark}

\begin{remark}
\label{lem2.4'}If $\left( x,k\right) =\mathcal{SP}\left( A,\Pi;m\right) $
and $(\hat{x},\hat{k})=\mathcal{SP}\left( A,\Pi;\hat{m}\right) $ then,
taking $x:=x_{r-}+\Delta m_{r}$ and $y:=\hat{x}_{r-}+\Delta\hat{m}_{r}$ in
Lemma \ref{lem2.4}, we see that, for any $a\in\overline{\mathrm{D}(A)}$,%
\begin{equation*}
\begin{array}{rl}
\left( i\right) & \displaystyle\langle x_{r}-\hat{x}_{r},\Delta k_{r}-\Delta%
\hat{k}_{r}\rangle+\frac{1}{2}|\Delta k_{r}-\Delta\hat{k}_{r}|^{2}\geq0%
\medskip \\
\left( ii\right) & \displaystyle r_{0}|\Delta k_{r}|\leq\langle
x_{r}-a,\Delta k_{r}\rangle+\frac{1}{2}|\Delta k_{r}|^{2},%
\end{array}%
\end{equation*}
where $r_{0}\geq0$ is such that $\overline{B\left( a,r_{0}\right) }\subset%
\overline{\mathrm{D}(A)}$.
\end{remark}

Let $(x,k)=\mathcal{SP}\left( A,\Pi;m\right) $. We will use the notation $x=%
\mathcal{SP}^{(1)}\left( A,\Pi;m\right) $ and $k=\mathcal{SP}^{(2)}\left(
A,\Pi;m\right) $. If $m$ is continuous, since $x,k$ does not depend on $\Pi$%
, we will write $(x,k)=\mathcal{SP}\left( A;m\right) $ and $x=\mathcal{SP}%
^{(1)}\left( A;m\right) $, $k=\mathcal{SP}^{(2)}\left( A;m\right) $.

The following version of the Tanaka's estimate of the distance between two
solutions of the Skorokhod problem will prove be useful in what follows.

\begin{lemma}
\label{lem2.7}We assume that $m,\hat{m}\in\mathbb{D}\left( \mathbb{R}^{+},%
\mathbb{H}\right) $ with $m_{0},\hat{m}_{0}\in\overline{\mathrm{D}\left(
A\right) }$. If $(x,k)=\mathcal{SP}\left( A,\Pi;m\right) $ and $(\hat{x},%
\hat{k})=\mathcal{SP}\left( A,\Pi;\hat{m}\right) $ then

\noindent$\left( i\right) $ for all $0\leq s<t$, $s,t\in\mathbb{R}^{+}$%
\begin{equation*}
\int_{s}^{t}\langle x_{r}-\hat{x}_{r},dk_{r}-d\hat{k}_{r}\rangle+\frac{1}{2}%
\sum_{s<r\leq t}|\Delta k_{r}-\Delta\hat{k}_{r}|^{2}\geq0\;;
\end{equation*}
\noindent$\left( ii\right) $ for all $t\in\mathbb{R}^{+}$,%
\begin{equation*}
|x_{t}-\hat{x}_{t}|^{2}\leq|m_{t}-\hat{m}_{t}|^{2}-2\int_{0}^{t}\langle
m_{t}-\hat{m}_{t}-m_{s}+\hat{m}_{s},dk_{s}-d\hat{k}_{s}\rangle.
\end{equation*}
\end{lemma}

\begin{proof}
$(i)$ Let $0\leq s<t$ be arbitrary chosen. From Proposition \ref{prop Anex
13} from Annexes, we see that, for all c\`{a}dl\`{a}g functions $(x,k),(\hat{%
x},\hat{k})$ satisfying Definition \ref{def2.5}, we have%
\begin{equation}
\int_{s}^{t}\langle x_{r}-\hat{x}_{r},dk_{r}^{c}-d\hat{k}_{r}^{c}\rangle
\geq 0,  \label{eq2.5}
\end{equation}%
since $(x,k^{c}),(\hat{x},\hat{k}^{c})\in \mathcal{A}$ (see the Annexes,
Proposition \ref{prop Anex 13}). On the other hand, using inequality $\left(
i\right) $ from Remark \ref{lem2.4'} and also the definition of the
Lebesgue-Stieltjes integral (see the Annexes), we deduce that%
\begin{equation}
\int_{s}^{t}\langle x_{r}-\hat{x}_{r},dk_{r}^{d}-d\hat{k}_{r}^{d}\rangle +%
\frac{1}{2}\sum_{s<r\leq t}|\Delta k_{r}-\Delta \hat{k}_{r}|^{2}\geq 0,\quad
0\leq s<t,\,\,s,t\in \mathbb{R}^{+}.  \label{eq2.6}
\end{equation}%
By combining (\ref{eq2.5}) with (\ref{eq2.6}) we obtain $(i)$.

\noindent $(ii)$ By Lemma \ref{lemma Annex 9}%
\begin{equation*}
|k_{t}-\hat{k}_{t}|^{2}=-\sum_{s\leq t}|\Delta k_{s}-\Delta \hat{k}%
_{s}|^{2}+2\int_{0}^{t}\langle k_{s}-\hat{k}_{s},dk_{s}-d\hat{k}_{s}\rangle .
\end{equation*}%
Therefore, using step $(i),$%
\begin{align*}
& |x_{t}-\hat{x}_{t}|^{2}-|m_{t}-\hat{m}_{t}|^{2}+2\int_{0}^{t}\langle m_{t}-%
\hat{m}_{t}-m_{s}+\hat{m}_{s},dk_{s}-d\hat{k}_{s}\rangle \\
& =-2\int_{0}^{t}\langle x_{s}-\hat{x}_{s},dk_{s}-d\hat{k}_{s}\rangle
-\sum_{s\leq t}|\Delta k_{s}-\Delta \hat{k}_{s}|^{2}\leq 0.
\end{align*}%
\hfill \medskip
\end{proof}

The next result it is an immediate consequence of Lemma \ref{lem2.7}.

\begin{corollary}
\label{cor2.8}Under the assumptions (\ref{assumpt}), differential equation (%
\ref{formal GSP}) admits at most one solution $(x,k)=\mathcal{SP}\left(
A,\Pi ;m\right) $.
\end{corollary}

In R\u{a}\c{s}canu \cite{ra/96} (in the context of Hilbert spaces which is
considered here) and C\'{e}pa \cite{ce/98} (finite dimensional case) it is
proved that, for any continuous $m$ such that $m_{0}\in \overline{\mathrm{D}%
\left( A\right) }\,$, there exists a unique continuous solution $\left(
x,k\right) =\mathcal{SP}\left( A;m\right) $. In particular, for any constant
function $\alpha \in \overline{\mathrm{D}\left( A\right) }$ there exists a
unique solution $\mathcal{SP}\left( A;m\right) $. In fact, for $m_{t}\equiv
\alpha \in \overline{\mathrm{D}\left( A\right) }$ we are in the absolutely
continuous case and by Proposition \ref{prop 4} the solution $%
x_{t}=S_{A}\left( t\right) \alpha $ where $S_{A}\left( t\right) $ is the
nonlinear semigroup generated by $A$.

\begin{lemma}
\label{lem2.9}Let assumptions (\ref{assumpt}) be satisfied and $m\in \mathbb{%
D}\left( \mathbb{R}^{+},\mathbb{H}\right) $ be a step function of the form%
\begin{equation*}
m_{t}=\sum_{r\in\pi}m_{r}\mathbf{1}_{[r,r^{\prime})}\left( t\right) ,
\end{equation*}
where $\pi\in\mathcal{P}_{\mathbb{R}^{+}}$ and $r^{\prime}$ is the successor
of $r$ in the partition $\pi$. Then there exists a unique solution $(x,k)=%
\mathcal{SP}\left( A,\Pi;m\right) $ and the solution is given by%
\begin{equation}
(x_{t},k_{t})=\mathcal{SP}\left( A,\Pi\left( x_{r-}+\Delta m_{r}\right)
\right) _{t-r}=S_{A}\left( t-r\right) \left( \Pi\left( x_{r-}+\Delta
m_{r}\right) \right) ,  \label{eq2.7}
\end{equation}
for all $t\in\lbrack r,r^{\prime})$, $r\in\pi$.
\end{lemma}

\begin{proof}
Clearly%
\begin{equation*}
x_{r}=\Pi (x_{{r}-}+\Delta m_{r}),\;\forall r\in \pi ,
\end{equation*}%
and $x_{t}=m_{t}-k_{t}\in \overline{\mathrm{D}\left( A\right) }$,\thinspace\
$t\in \mathbb{R}^{+}$. Function $k$ is with locally bounded variation, $%
k_{0}=0$, such that for any $(\alpha ,\beta )\in \mathrm{Gr}(A)$%
\begin{equation*}
\int_{s}^{t}\langle x_{r}-\alpha ,dk_{r}^{c}-\beta \,du\rangle \geq 0,\quad
0\leq s<t,~s,t\in \mathbb{R}^{+}
\end{equation*}%
and the proof is complete.\hfill
\end{proof}

\begin{lemma}
\label{lem2.10}Let $(x,k)=\mathcal{SP}\left( A,\Pi ;m\right) $. Let $a\in
\mathrm{Int}\left( \mathrm{D}\left( A\right) \right) $ and $r_{0}>0$ be such
that $\overline{B\left( a,r_{0}\right) }\subset \mathrm{D}\left( A\right) .$
Let%
\begin{equation*}
A_{a,r_{0}}^{\#}:=\sup \left\{ \left\vert \hat{u}\right\vert :\hat{u}\in
Au,\;u\in \overline{B\left( a,r_{0}\right) }\right\} \leq \mu <\infty .
\end{equation*}%
Then for every $0\leq s<t$%
\begin{equation*}
r_{0}\left\updownarrow k\right\updownarrow _{\left[ s,t\right] }\leq
\int_{s}^{t}\langle x_{r}-a,dk_{r}\rangle +\frac{1}{2}\sum_{s<r\leq
t}|\Delta k_{r}|^{2}+\mu \int_{s}^{t}|x_{r}-a|dr+(t-s)r_{0}\mu \,.
\end{equation*}
\end{lemma}

\begin{proof}
Using Proposition \ref{prop Anex 15} from the Annexes we see that%
\begin{equation*}
r_{0}\left\updownarrow k^{c}\right\updownarrow _{\left[ s,t\right] }\leq
\int_{s}^{t}\langle x_{u}-a,dk_{u}^{c}\rangle +\mu
\int_{s}^{t}|x_{u}-a|du+(t-s)r_{0}\mu \,.
\end{equation*}%
From Remark \ref{lem2.4'}-$(ii)$ we deduce that%
\begin{equation*}
r_{0}\updownarrow {\hspace{-0.1cm}}k^{d}{\hspace{-0.1cm}}\updownarrow _{%
\left[ s,t\right] }\leq \int_{s}^{t}\langle x_{u}-a,dk_{u}^{d}\rangle +\frac{%
1}{2}\sum_{s<u\leq t}|\Delta k_{u}|^{2}
\end{equation*}%
and the result follows, since $\updownarrow {\hspace{-0.1cm}}k{\hspace{-0.1cm%
}}\updownarrow _{\left[ s,t\right] }\,\leq \,\updownarrow {\hspace{-0.1cm}}%
k^{c}{\hspace{-0.1cm}}\updownarrow _{\left[ s,t\right] }+\updownarrow {%
\hspace{-0.1cm}}k^{d}{\hspace{-0.1cm}}\updownarrow _{\left[ s,t\right] }.$%
\hfill
\end{proof}

\begin{theorem}[Convergence results]
\label{prop2.11}Assume that:$\smallskip $

\noindent $\left( i\right) $ assumptions (\ref{assumpt}) are satisfied;$%
\smallskip $

\noindent $\left( ii\right) $ $m,m^{\left( n\right) }\in \mathbb{D}\left(
\mathbb{R}^{+},\mathbb{H}\right) $ and $(x^{\left( n\right) },k^{\left(
n\right) })=\mathcal{SP}\left( A,\Pi ;m^{\left( n\right) }\right) $, $n\in
\mathbb{N}^{\ast }.\smallskip $

The following assertions hold true:\medskip

\noindent $\left( \mathbf{I}\right) \quad $If for every $T>0$, $\left\Vert
m^{\left( n\right) }-m\right\Vert _{T}\rightarrow 0$, then$\smallskip $

\noindent $\quad \left( j\right) $ for any $a\in \mathrm{Int}\left( \mathrm{D%
}\left( A\right) \right) $, $r_{0}>0$ such that $\overline{B\left(
a,r_{0}\right) }\subset \overline{\mathrm{D}\left( A\right) }$, for any $\mu
\geq A_{a,r_{0}}^{\#}$, there exist a constant $C=C(a,r_{0},\mu ,T,$\textbf{$%
\boldsymbol{\gamma }$}$_{m}(\cdot ,T))>0$ and $n_{0}\in \mathbb{N}^{\ast }$
such that for all $n\geq n_{0}:$%
\begin{equation}
||x^{\left( n\right) }||_{T}^{2}+\updownarrow {\hspace{-0.1cm}}k^{\left(
n\right) }{\hspace{-0.1cm}}\updownarrow _{T}\leq ||x^{\left( n\right)
}||_{T}^{2}+\updownarrow {\hspace{-0.1cm}}k^{\left( n\right) ,c}{\hspace{%
-0.1cm}}\updownarrow _{T}{+\updownarrow {\hspace{-0.1cm}}k^{\left( n\right)
,d}{\hspace{-0.1cm}}\updownarrow _{T}\leq C}(1+\left\Vert m\right\Vert
_{T}^{2});  \label{estim xk}
\end{equation}%
\noindent $\quad \left( jj\right) $ there exist $x,k\in \mathbb{D}\left(
\mathbb{R}^{+},\mathbb{H}\right) $ such that%
\begin{equation*}
||x^{\left( n\right) }-x||_{T}+||k^{\left( n\right) }-k||_{T}\rightarrow
0,\quad \text{as }n\rightarrow \infty ,
\end{equation*}%
and%
\begin{equation}
\left\Vert x\right\Vert _{T}^{2}+\left\updownarrow {k}\right\updownarrow {%
_{T}}\leq {C}(1+\left\Vert m\right\Vert _{T}^{2});  \label{estim xk1}
\end{equation}%
\noindent $\quad \left( jjj\right) $ $(x,k)=\mathcal{SP}\left( A,\Pi
;m\right) .$\medskip

\noindent $\left( \mathbf{II}\right) $ If $m^{\left( n\right)
}\longrightarrow m$ in $\mathbb{D}\left( \mathbb{R}^{+},\mathbb{H}\right) $
then there exist $x,k\in \mathbb{D}\left( \mathbb{R}^{+},\mathbb{H}\right) $
such that
\begin{equation*}
(x^{\left( n\right) },k^{\left( n\right) },m^{\left( n\right)
})\longrightarrow (x,k,m)\quad \text{in }\mathbb{D}\left( \mathbb{R}^{+},%
\mathbb{H}\times \mathbb{H\times H}\right) .
\end{equation*}%
and $(x,k)=\mathcal{SP}\left( A,\Pi ;m\right) .$\medskip

\noindent$\left( \mathbf{III}\right) $ Let $m_{0}\in\overline{\mathrm{D}%
\left( A\right) }$ and $\pi_{n}\in\mathcal{P}_{\mathbb{R}^{+}}$ be a
partition such that $\left\Vert \pi_{n}\right\Vert \rightarrow0$. If $%
m_{t}^{(n)}=\sum_{r\in\pi_{n}}m_{r}\mathbf{1}_{[r,r^{\prime})}\left(
t\right) $ denotes the $\pi_{n}-$discretization of $m$ and $(x^{\left(
n\right) },k^{\left( n\right) })=\mathcal{SP}\left( A,\Pi;m^{(n)}\right) $, $%
n\in\mathbb{N}$ then there exist $x,k\in\mathbb{D}\left( \mathbb{R}^{+},%
\mathbb{H}\right) $ such that%
\begin{equation*}
{(x^{\left( n\right) },k^{\left( n\right) },m^{(n)})\longrightarrow
(x,k,m)\quad}\text{in }\mathbb{D}\left( \mathbb{R}^{+},\mathbb{H}\times%
\mathbb{H\times H}\right)
\end{equation*}
and $(x,k)=\mathcal{SP}\left( A,\Pi;m\right) .$
\end{theorem}

\begin{proof}
$(\mathbf{I-}j)$ Let $0\leq t\leq T.$ We have%
\begin{equation*}
|k_{t}^{\left( n\right) }|^{2}=2\int_{0}^{t}\langle k_{r}^{\left( n\right)
},dk_{r}^{\left( n\right) }\rangle -\sum_{r\leq t}|\Delta k_{r}^{\left(
n\right) }|^{2}\,.
\end{equation*}%
Since $k_{t}=m_{t}-x_{t}$~,%
\begin{equation*}
|m_{t}^{\left( n\right) }-x_{t}^{\left( n\right) }|^{2}=2\int_{0}^{t}\langle
m_{r}^{\left( n\right) }-x_{r}^{\left( n\right) },dk_{r}^{\left( n\right)
}\rangle -\sum_{r\leq t}|\Delta k_{r}^{\left( n\right) }|^{2}
\end{equation*}%
and moreover%
\begin{equation*}
|m_{t}^{\left( n\right) }-x_{t}^{\left( n\right) }|^{2}+2\int_{0}^{t}\langle
x_{r}^{\left( n\right) }-a,dk_{r}^{\left( n\right) }\rangle
=2\int_{0}^{t}\langle m_{r}^{\left( n\right) }-a,dk_{r}^{\left( n\right)
}\rangle -\sum_{r\leq t}|\Delta k_{r}^{\left( n\right) }|^{2}\,.
\end{equation*}%
Using Lemma \ref{lem2.10} we obtain%
\begin{equation}
|x_{t}^{\left( n\right) }-m_{t}^{\left( n\right) }|^{2}+2r_{0}\updownarrow {%
\hspace{-0.1cm}}k^{\left( n\right) }{\hspace{-0.1cm}}\updownarrow _{t}\leq
2\mu t||x^{\left( n\right) }-a||_{t}+2r_{0}\mu t+2\int_{0}^{t}\langle
m_{r}^{\left( n\right) }-a,dk_{r}^{\left( n\right) }\rangle \,.
\label{prop2.11_1}
\end{equation}%
Since $m\in \mathbb{D}\left( \mathbb{R}^{+},\mathbb{H}\right) $, there
exists a partition $\pi \in \mathcal{P}_{\mathbb{R}^{+}}$ such that $%
\max_{r\in \pi }\mathcal{O}_{m}\left( [r,r^{\prime })\right) <r_{0}/2$ (see
Remark \ref{remark Annex 6'} in the Annexes). We can assume that $T\in \pi $
(if not, we can replace $\pi $ by $\tilde{\pi}=\pi \cup \left\{ T\right\} $
which also satisfies the above condition). Let $N_{0}=\mathrm{card}\left\{
r\in \pi :r\leq T\right\} $ . Since $||m^{\left( n\right)
}-m||_{T}\rightarrow 0$, there exists $n_{0}\geq N_{0}$ such that $%
||m^{\left( n\right) }||_{T}\leq 1+\left\Vert m\right\Vert _{T}$ and $%
\max_{r\in \pi }\mathcal{O}_{m^{\left( n\right) }}\left( [r,r^{\prime
})\right) <r_{0}/2$ for all $n\geq n_{0}.$

Let $m_{s}^{\left( n\right) ,\pi }=\sum_{r\in \pi }m_{r}^{\left( n\right) }%
\mathbf{1}_{[r,r^{\prime })}\left( s\right) ,~s\geq 0.$ We have $\left\vert
m_{s}^{\left( n\right) }-m_{s}^{\left( n\right) ,\pi }\right\vert <r_{0}/2$
for all $s\geq 0$ and $n\geq n_{0}.$%
\begin{align*}
2\int_{0}^{t}\langle m_{s}^{\left( n\right) }-a,dk_{s}\rangle &
=2\int_{0}^{t}\langle m_{s}^{\left( n\right) }-m_{s}^{\left( n\right) ,\pi
},dk_{s}\rangle +2\int_{0}^{t}\langle m_{s}^{\left( n\right) ,\pi
}-a,dk_{s}^{\left( n\right) }\rangle \\
& \leq r_{0}\updownarrow {\hspace{-0.1cm}}k^{\left( n\right) }{\hspace{-0.1cm%
}}\updownarrow _{t}+2\sum_{r\in \pi ,\;r<t}\langle m_{r}^{\left( n\right)
}-a,k_{r^{\prime }\wedge t}^{\left( n\right) }-k_{r}^{\left( n\right)
}\rangle +2\langle m_{t}^{\left( n\right) }-a,\Delta k_{t}^{\left( n\right)
}\rangle \\
& \leq r_{0}\updownarrow {\hspace{-0.1cm}}k^{\left( n\right) }{\hspace{-0.1cm%
}}\updownarrow _{t}+4N_{0}||m^{\left( n\right) }-a||_{t}\,||k^{\left(
n\right) }||_{t} \\
& \leq r_{0}\updownarrow {\hspace{-0.1cm}}k^{\left( n\right) }{\hspace{-0.1cm%
}}\updownarrow _{t}+4N_{0}||m^{\left( n\right) }-a||_{t}\,||m^{\left(
n\right) }-x^{\left( n\right) }||_{t}
\end{align*}%
Using this last estimate in (\ref{prop2.11_1}) we infer for all $t\in \left[
0,T\right] $ and $n\geq n_{0}$%
\begin{equation*}
\begin{array}{l}
|x_{t}^{\left( n\right) }-m_{t}^{\left( n\right) }|^{2}+r_{0}\updownarrow {%
\hspace{-0.1cm}}k^{\left( n\right) }{\hspace{-0.1cm}}\updownarrow _{t}\leq
2\mu T\left\Vert (x^{\left( n\right) }-m^{\left( n\right) })+(m^{\left(
n\right) }-a)\right\Vert _{T}+2r_{0}\mu T\medskip \\
\quad +4N_{0}||m^{\left( n\right) }-a||_{T}\,||x^{\left( n\right)
}-m^{\left( n\right) }||_{T}\leq \frac{1}{2}\left\Vert x^{\left( n\right)
}-m^{\left( n\right) }\right\Vert _{T}^{2}+C(1+\left\Vert m\right\Vert
_{T}^{2})%
\end{array}%
\end{equation*}%
with $C=C\left( \mu ,\left\vert a\right\vert ,T,N_{0},r_{0}\right) $ a
positive constant.

Taking into account that%
\begin{equation*}
\updownarrow {\hspace{-0.1cm}}k^{d}{\hspace{-0.1cm}}\updownarrow
_{t}=\sum_{r\leq t}\left\vert \Delta k_{r}\right\vert \leq \left\updownarrow
k\right\updownarrow _{t}\quad \text{and}\quad \left\updownarrow
k^{c}\right\updownarrow _{t}=\updownarrow {\hspace{-0.1cm}}k-k^{d}{\hspace{%
-0.1cm}}\updownarrow _{t}\leq 2\left\updownarrow k\right\updownarrow _{t}\,,
\end{equation*}%
the inequality (\ref{estim xk}) follows.

$(\mathbf{I-}jj)$ By Lemma~\ref{lem2.7} for any $n,\ell \in \mathbb{N}^{\ast
},$ $n,\ell \geq n_{0}$%
\begin{align*}
||x^{\left( n\right) }-x^{\left( \ell \right) }||_{T}^{2}& \leq ||m^{\left(
n\right) }-m^{\left( \ell \right) }||_{T}^{2}+4||m^{\left( n\right)
}-m^{\left( \ell \right) }||_{T}(\updownarrow {\hspace{-0.1cm}}k^{\left(
n\right) }{\hspace{-0.1cm}}\updownarrow _{T}+\updownarrow {\hspace{-0.1cm}}%
k^{\left( l\right) }{\hspace{-0.1cm}}\updownarrow _{T}) \\
& \leq \Big(\frac{1}{n}+\frac{1}{\ell }\Big)^{2}+8\Big(\frac{1}{n}+\frac{1}{%
\ell }\Big){C}(1+\left\Vert m\right\Vert _{T}^{2}).
\end{align*}%
Hence $x^{(n)}$ and $k^{(n)}=m^{\left( n\right) }-x^{(n)}$ are Cauchy
sequences in the spaces of c\`{a}dl\`{a}g functions on $[0,T]$. Thus there
exist functions $x,k\in \mathbb{D}\left( \mathbb{R}^{+},\mathbb{H}\right) $
such that
\begin{equation*}
||x^{\left( n\right) }-x||_{T}+||k^{\left( n\right) }-k||_{T}\rightarrow
0,\quad \text{as }n\rightarrow \infty
\end{equation*}%
and by Lemma \ref{lemma Annex 10} passing to $\liminf_{n\rightarrow +\infty
} $ in (\ref{estim xk}) we infer (\ref{estim xk1}).

$(\mathbf{I-}jjj)$ We prove now that $(x,k)$ satisfies conditions (\ref{def
sol GSP}). Since for all $t\in \left[ 0,T\right] :$ $x_{t}^{(n)}\in
\overline{\mathrm{D}\left( A\right) }$,\ $x_{t}^{\left( n\right)
}+k_{t}^{\left( n\right) }=m_{t}^{\left( n\right) }$, $x_{t}^{\left(
n\right) }=\Pi (x_{t-}^{\left( n\right) }+\Delta m_{t}^{\left( n\right) }),$
passing to limit the same conditions are satisfied by $(x,k).$ It remains to
show that $dk_{t}^{c}\in A\left( x_{t}\right) \left( dt\right) .$\newline
Let $(z,\hat{z})\in \mathrm{Gr}(A)$ be arbitrary. From Helly-Bray Theorem %
\ref{Helly-Bray} there exist a subsequence $n_{0}\leq
n_{1}<n_{2}<n_{3}<\cdots $, $n_{i}\rightarrow \infty $ and a sequence $%
\delta _{i}\searrow 0$ as $i\rightarrow \infty $ such that uniformly with
respect to $s,t\in \left[ 0,T\right] ,$ $s\leq t$%
\begin{equation*}
\int_{s}^{t}\mathbf{1}_{|\Delta k_{r}^{(n_{i})}|>\delta _{i}}\langle
x_{r}^{(n_{i})}-z,dk_{r}^{(n_{i}),d}\rangle \rightarrow \int_{s}^{t}\langle
x_{r}-z,dk_{r}^{d}\rangle ,\quad \text{as }i\rightarrow \infty .
\end{equation*}%
By Remark \ref{lem2.4'}-$(ii)$ we infer%
\begin{align*}
\int_{s}^{t}\mathbf{1}_{|\Delta k_{r}^{(n_{i})}|\leq \delta _{i}}\langle
x_{r}^{(n_{i})}-z,dk_{r}^{(n_{i}),d}\rangle & =\sum_{r\in (s,t]}\langle
x_{r}^{(n_{i})}-\alpha ,\Delta k_{r}^{(n_{i})}\rangle \mathbf{1}_{|\Delta
k_{r}^{(n_{i})}|\leq \delta _{i}} \\
& \geq -\frac{1}{2}\sum_{r\in (s,t]}|\Delta k_{r}^{(n_{i})}|^{2}\mathbf{1}%
_{|\Delta k_{r}^{(n_{i})}|\leq \delta _{i}}\, \\
& \geq -\frac{1}{2}\delta _{i}\updownarrow {\hspace{-0.1cm}}k^{\left(
n_{i}\right) }{\hspace{-0.1cm}}\updownarrow _{\left[ s,t\right] } \\
& \geq -\frac{1}{2}\delta _{i}{C}(1+\left\Vert m\right\Vert _{T}^{2}).
\end{align*}%
From (\ref{def sol GSP}-$iv$) for $(x^{(n_{i})},k^{(n_{i})})$ we have
\begin{equation}
\int_{s}^{t}\langle x_{r}^{(n_{i})}-z,dk_{r}^{\left( n_{i}\right) }-\hat{z}%
dr\rangle \geq \int_{s}^{t}\langle x_{r}^{(n_{i})}-z,dk_{r}^{\left(
n_{i}\right) ,d}\rangle  \label{eq 2.15a}
\end{equation}%
which yields
\begin{equation*}
\int_{s}^{t}\langle x_{r}^{(n_{i})}-z,dk_{r}^{\left( n_{i}\right) }-\hat{z}%
dr\rangle \geq \int_{s}^{t}\mathbf{1}_{|\Delta k_{r}^{(n_{i})}|>\delta
_{i}}\langle x_{r}^{(n_{i})}-z,dk_{r}^{(n_{i}),d}\rangle -\frac{1}{2}\delta
_{i}{C}(1+\left\Vert m\right\Vert _{T}^{2}).
\end{equation*}%
Passing to $\lim_{i\rightarrow \infty }$ in this last inequality we deduce
using Helly-Bray Theorem \ref{Helly-Bray} and (\ref{eq 2.15a}):%
\begin{equation*}
\int_{s}^{t}\langle x_{r}-z,dk_{r}-\hat{z}dr\rangle \geq \int_{s}^{t}\langle
x_{r}-z,dk_{r}^{d}\rangle
\end{equation*}%
that is%
\begin{equation*}
\int_{s}^{t}\langle x_{r}-z,dk_{r}^{c}-\hat{z}dr\rangle \geq 0,
\end{equation*}%
for all $(z,\hat{z})\in \mathrm{Gr}(A),$ $T>0$ and $0\leq s\leq t\leq T.$%
\medskip

$\left( \mathbf{II}\right) $ By the definition of convergence in the
Skorokhod topology $J_{1}$, there exists a sequence of strictly increasing
continuous changes of time $\{\lambda ^{\left( n\right) }\}$ such that $%
\lambda _{0}^{\left( n\right) }=0$, $\lambda _{\infty }^{\left( n\right)
}=+\infty $, $\sup_{t\geq 0}|\lambda _{t}^{\left( n\right) }-t|\rightarrow 0$
and $\sup_{t\in \left[ 0,T\right] }|m_{\lambda _{t}^{\left( n\right)
}}^{\left( n\right) }-m_{t}|\rightarrow 0$, for all $T\in \mathbb{R}^{+}$.
Since $(x_{\lambda ^{\left( n\right) }}^{\left( n\right) },k_{\lambda
^{\left( n\right) }}^{\left( n\right) })=\mathcal{SP}\big(A,\Pi ;m_{\lambda
^{\left( n\right) }}^{\left( n\right) }\big)$, the first part $(\mathbf{I})$
of the proof yields there exist $x,k\in \mathbb{D}\left( \mathbb{R}^{+},%
\mathbb{H}\right) $ such that for any $T\in \mathbb{R}^{+}$,
\begin{equation*}
||x_{\lambda ^{\left( n\right) }}^{\left( n\right) }-x||_{T}\rightarrow
0\quad \text{and}\quad ||k_{\lambda ^{\left( n\right) }}^{\left( n\right)
}-k||_{T}\rightarrow 0,
\end{equation*}%
and $(x,k)=\mathcal{SP}\left( A,\Pi ;m\right) .$ The proof is
complete.\medskip

$\left( \mathbf{III}\right) $ Applying Proposition \ref{proposition Annex 6}
from Annexes, we see that $m^{(n)}\longrightarrow m$ in $\mathbb{D}\left(
\mathbb{R}^{+},\mathbb{H}\right) $ and, using $\left( \mathbf{II}\right) ,$
the result follows.\hfill \medskip
\end{proof}

We state now the main result of this section:

\begin{theorem}
\label{thm2.13}Under the assumptions (\ref{assumpt}) differential equation (%
\ref{formal GSP}) has a unique solution $(x,k)=\mathcal{SP}\left( A,\Pi
;m\right) $.
\end{theorem}

\begin{proof}
Uniqueness follows from Corollary \ref{cor2.8}. Since $m\in \mathbb{D}\left(
\mathbb{R}^{+},\mathbb{H}\right) $, there exists a partition $\pi _{n}\in
\mathcal{P}_{\mathbb{R}^{+}}$ such that $\max_{r\in \pi _{n}}\mathcal{O}%
_{m}\left( [r,r^{\prime })\right) <\frac{1}{n},$ $n\in \mathbb{N}^{\ast }.$
Let $m_{s}^{\left( n\right) }=\sum_{r\in \pi _{n}}m_{r}\mathbf{1}%
_{[r,r^{\prime })}\left( s\right) ,~s\geq 0.$ Let $T>0$ be arbitrary. Then
for all $n\in \mathbb{N}^{\ast }:$
\begin{equation}
||m^{\left( n\right) }-m||_{T}\leq \frac{1}{n}\,.  \label{eq2.13}
\end{equation}%
By Lemma \ref{lem2.9} for the step function $m^{(n)}$, with $%
m_{0}^{(n)}=m_{0}\in \overline{\mathrm{D}\left( A\right) },$ there exists a
unique solution $(x^{(n)},k^{(n)})=\mathcal{SP}\left( A,\Pi ;m^{(n)}\right) $%
, $n\in \mathbb{N}^{\ast }.$ Since $\left\Vert m^{\left( n\right)
}-m\right\Vert _{T}\rightarrow 0$ , by Theorem \ref{prop2.11} there exist $%
x,k\in \mathbb{D}\left( \mathbb{R}^{+},\mathbb{H}\right) $ such that%
\begin{equation*}
||x^{\left( n\right) }-x||_{T}+||k^{\left( n\right) }-k||_{T}\rightarrow
0,\quad \text{as }n\rightarrow \infty ,
\end{equation*}%
and $(x,k)=\mathcal{SP}\left( A,\Pi ;m\right) .$\hfill
\end{proof}

\begin{remark}
\label{rem2.16}If $m$ is continuous then the solution of the Skorokhod
problem $(x,k)=\mathcal{SP}\left( A,\Pi ;m\right) $ is also continuous and
is not depending on projections $\Pi $ (see Remark \ref{rem2.6}). By Theorem %
\ref{prop2.11}-$\left( \mathbf{I}\right) $ for any projection $\Pi $, if $%
(x^{n},k^{n})=\mathcal{SP}\left( A,\Pi ;m^{n}\right) $, $n\in \mathbb{N}$
and $||m^{n}-m||_{T}\rightarrow 0$, $T\in \mathbb{R}^{+}$, then%
\begin{equation*}
||x^{n}-x||_{T}\rightarrow 0\quad \text{and}\quad ||k^{n}-k||_{T}\rightarrow
0,~T\in \mathbb{R}^{+}.
\end{equation*}
\end{remark}

\begin{example}
\textbf{1.} We consider $\mathbb{H}$ a Hilbert space with $\left\{
e_{i}\right\} _{i\in\mathbb{N}^{\ast}}$ an orthonormal basis and the set of
constraints $\mathcal{O}=\overline{\mathcal{O}}\subset\mathbb{H}$ given by%
\begin{equation*}
\overline{\mathcal{O}}=\left\{ y\in\mathbb{H}:y^{i}\geq0,~\forall i=%
\overline{1,N}\right\} .
\end{equation*}
We propose the following example of upper bidiagonal infinite dimensional
system with constraints and with singular inputs generated by c\`{a}dl\`{a}g
functions:%
\begin{equation}
\left\{
\begin{array}{l}
dx_{t}^{1}+\left( x_{t}^{1}+x_{t}^{2}\right)
dt+dk_{t}^{1}=dm_{t}^{1}\,,\medskip \\
dx_{t}^{2}+\left( x_{t}^{2}+x_{t}^{3}\right)
dt+dk_{t}^{2}=dm_{t}^{2}\,,\medskip \\
\cdots\cdots\cdots\cdots\cdots\cdots\medskip \\
dx_{t}^{i}+\left( x_{t}^{i}+x_{t}^{i+1}\right)
dt+dk_{t}^{i}=dm_{t}^{i}\,,\medskip \\
\cdots\cdots\cdots\cdots\cdots\cdots%
\end{array}
\right.  \label{exam}
\end{equation}
where $m_{t}=\sum\limits_{i\in\mathbb{N}^{\ast}}m_{t}^{i}e_{i}$ with $%
m_{t}^{i}\in\mathbb{D}\left( \mathbb{R}^{+},\mathbb{R}\right) $, $i\in%
\mathbb{N}^{\ast}$.

This system can be written in the abstract form%
\begin{equation*}
\left\{
\begin{array}{l}
dx_{t}+A\left( x_{t}\right) dt\ni dm_{t}\,,\medskip \\
x_{0}=m_{0}\,,%
\end{array}
\right.
\end{equation*}
with $A\left( x\right) =x+\hat{x}+\partial I_{\overline{\mathcal{O}}}\left(
x\right) $, where, for $x=\sum\limits_{i\in\mathbb{N}^{\ast}}x^{i}e_{i}$, $%
\hat{x}=\sum\limits_{i\in\mathbb{N}^{\ast}}x^{i+1}e_{i}.$ Clearly $\overline{%
\mathrm{D}\left( A\right) }=\overline{\mathcal{O}}$ and $A$ is a maximal
monotone operator (as a sum of the linear continuous non-negative operator $%
L\left( x\right) =x+\hat{x}$ and the maximal monotone operator $\partial I_{%
\overline{\mathcal{O}}}$).

Applying Theorem \ref{thm2.13}, we infer that there exists a unique solution
(see definition (\ref{def sol GSP 2})) $\left( x,k\right) =\mathcal{SP}%
\left( A,\Pi_{\overline{\mathcal{O}}};m\right) $. We outline that solution $%
\left( x,k\right) $ satisfies%
\begin{equation*}
\left\{
\begin{array}{l}
\displaystyle x_{t}^{i}+%
\int_{0}^{t}(x_{s}^{i}+x_{s}^{i+1})ds+k_{t}^{i}=m_{t}^{i}\,,~\forall i\in%
\mathbb{N}^{\ast},\medskip \\
\displaystyle\sum\nolimits_{j=1}^{N}\int_{s}^{t}\langle
x_{r}^{j}-a^{j},dk_{r}^{j}\rangle\leq0\,,\;\forall~a^{j}\geq0,\medskip \\
k_{t}^{i}=0\,,~\forall i>N.%
\end{array}
\right.
\end{equation*}
\noindent\textbf{2.} In the above example we can consider a generalized
projection $\Pi:\mathbb{H\rightarrow}\overline{\mathcal{O}}$,%
\begin{equation*}
\Pi\left( x\right) =\sum\limits_{i\in\mathbb{N}^{\ast}}p^{i}e_{i}\quad\text{%
with}\quad p^{i}:=(x^{i})^{+}+\left[ g((x^{i})^{-})-g\left( 0\right) \right]
^{+},
\end{equation*}
where $g:\mathbb{R}\rightarrow\mathbb{R}$ is a $1-$Lipschitz function. In
this case $\left( x,k\right) =\mathcal{SP}\left( A,\Pi;m\right) $ and we
mention%
\begin{equation*}
\left\{
\begin{array}{l}
\displaystyle x_{t}^{i}+%
\int_{0}^{t}(x_{s}^{i}+x_{s}^{i+1})ds+k_{t}^{i,c}+k_{t}^{i,d}=m_{t}^{i}\,,~%
\forall i\in\mathbb{N}^{\ast},\medskip \\
\displaystyle\sum\nolimits_{j=1}^{N}\int_{s}^{t}\langle
x_{r}^{j}-a^{j},dk_{r}^{j,c}\rangle\leq0\,\;\forall~a^{j}\geq0,\medskip \\
x_{t}=\Pi\left( x_{t-}+\Delta m_{t}\right) \,,~t>0,\medskip \\
k_{t}^{i}=0\,,~\forall i>N.%
\end{array}
\right.
\end{equation*}
\noindent\textbf{3.} If we ask to the feedback law $dk_{t}$ to produce, in a
minimal way, a boundedness of the form%
\begin{equation*}
\sum\limits_{i\in\mathbb{N}^{\ast}}(x_{t}^{i})^{2}\leq R^{2}
\end{equation*}
then we should take $A\left( x\right) =x+\hat{x}+\partial I_{\overline {%
\mathcal{O}}}\left( x\right) +\partial I_{\overline{B\left( 0,R\right) }%
}\left( x\right) $ with $\overline{\mathrm{D}\left( A\right) }=\overline{%
\mathcal{O}}\cap\overline{B\left( 0,R\right) }$.

\noindent \textbf{4.} An interacting particles system with singular input
can be modeled by the following infinite dimensional system%
\begin{equation*}
\left\{
\begin{array}{l}
dx_{t}+L\left( x_{t}\right) dt+A\left( x_{t}\right) \left( dt\right) \ni
dm_{t}\,,\medskip \\
x_{0}=m_{0}\,,%
\end{array}%
\right.
\end{equation*}%
where $L:\mathbb{H}\rightarrow \mathbb{H}$ is a linear continuous
non-negative operator and $A=\partial \varphi $ with $\varphi $ a proper
convex lower semicontinuous function given by%
\begin{equation*}
\varphi \left( x\right) =\sum\limits_{i\neq j}\psi \left( x^{j}-x^{i}\right)
\end{equation*}%
(see also the examples from Sznitman \cite{sz/91} and C\'{e}pa \cite{ce/98}).
\end{example}

\section{The penalized problem}

For all $\varepsilon >0$ and $z\in \mathbb{H}$ let us define (see the
Annexes, subsection \ref{Maximal monotone})%
\begin{equation*}
J_{\varepsilon }(z)=\big(I+\varepsilon A\big)^{-1}(z),\quad A_{\varepsilon
}(z)=\frac{1}{\varepsilon }(z-J_{\varepsilon }(z)),
\end{equation*}%
where the sequence $A_{\varepsilon }$ is the Yosida approximation of the
operator $A$. For the properties of $J_{\varepsilon }$ and $A_{\varepsilon }$
see Proposition \ref{prop Anex 4} in the Annexes.

\subsection{Approximation with unamortized jumps}

Let $m^{\varepsilon },m\in {\mathbb{D}}\left( {\mathbb{R}^{+}},{{\mathbb{H}}}%
\right) $ such that, for all $T>0$, $\left\Vert m^{\varepsilon
}-m\right\Vert _{T}\rightarrow 0$, as $\varepsilon \rightarrow 0$. We will
consider equations of the form%
\begin{equation}
x_{t}^{\varepsilon }+\int_{0}^{t}A_{\varepsilon }(x_{s}^{\varepsilon
})ds=m_{t}^{\varepsilon },~t\in {\mathbb{R}^{+}},\,\varepsilon >0.
\label{eq3.1}
\end{equation}%
Since $A_{\varepsilon }$ is Lipschitz continuous it is well known that there
exists a unique solution $x^{\varepsilon }$ of (\ref{eq3.1}). On the other
hand, $A_{\varepsilon }:\mathbb{H}\rightarrow \mathbb{H}$ is a maximal
monotone operator (is monotone and continuous) with $\overline{\mathrm{D}%
\left( A_{\varepsilon }\right) }=\mathbb{H}$. Therefore any generalized
projection $\Pi =I=\Pi _{\overline{\mathrm{D}\left( A_{\varepsilon }\right) }%
}$, where $I$ is the identity operator on $\mathbb{H}$ and from Theorem \ref%
{thm2.13} we deduce that there exists a unique solution $\left(
x^{\varepsilon },k^{\varepsilon }\right) =\mathcal{SP}\left( A_{\varepsilon
},I;m^{\varepsilon }\right) =\mathcal{SP}\left( A_{\varepsilon
};m^{\varepsilon }\right) $, where $k_{t}^{\varepsilon
}:=\int_{0}^{t}A_{\varepsilon }(x_{s}^{\varepsilon })ds.$

The following inequality holds (see Theorem \ref{prop2.11}, inequality (\ref%
{estim xk})): there exits $\varepsilon _{0}>0$ such that for all $%
0<\varepsilon \leq \varepsilon _{0},$%
\begin{equation}
\left\Vert x^{\varepsilon }\right\Vert _{T}^{2}+\left\updownarrow {k}%
^{\varepsilon }\right\updownarrow {_{T}}\leq C(1+\left\Vert m\right\Vert
_{T}^{2}),  \label{prop3.4}
\end{equation}%
since $\left\vert A_{\varepsilon }\left( u\right) \right\vert \leq
\left\vert A^{0}\left( u\right) \right\vert $, for all $u\in \mathrm{D}%
\left( A\right) $ and consequently%
\begin{equation*}
\sup \left\{ \left\vert A_{\varepsilon }u\right\vert :u\in \overline{B\left(
a,r_{0}\right) }\right\} \leq A_{a,r_{0}}^{\#}\leq \mu .
\end{equation*}

\begin{remark}
\label{lem3.2}If $m$ is a step function of the form $m_{t}=\sum_{r\in \pi
}m_{r}\mathbf{1}_{[r,r^{\prime })}\left( t\right) $, where $\pi $ is a
partition from $\mathcal{P}_{\mathbb{R}^{+}}$, then the solution $%
x^{\varepsilon }=\mathcal{SP}^{\left( 1\right) }\left( A_{\varepsilon
};m\right) $ is given by%
\begin{equation}
x_{t}^{\varepsilon }=\mathcal{SP}^{\left( 1\right) }\left( A_{\varepsilon
};x_{r-}^{\varepsilon }+\Delta m_{r}\right) _{t-r},\quad \forall t\in
\lbrack r,r^{\prime }),~r\in \pi .  \label{eq3.2}
\end{equation}
\end{remark}

\begin{lemma}
\label{lem3.3} Let $m,\hat{m}\in {\mathbb{D}}\left( {\mathbb{R}^{+}},{{%
\mathbb{H}}}\right) $. If $x^{\varepsilon }=\mathcal{SP}^{\left( 1\right)
}\left( A_{\varepsilon };m\right) $, $\hat{x}^{\varepsilon }=\mathcal{SP}%
^{\left( 1\right) }\left( A_{\varepsilon };\hat{m}\right) $ and%
\begin{equation*}
k_{t}^{\varepsilon }=\int_{0}^{t}A_{\varepsilon }(x_{s}^{\varepsilon
})ds\,,\quad \hat{k}_{t}^{\varepsilon }=\int_{0}^{t}A_{\varepsilon }(\hat{x}%
_{s}^{\varepsilon })ds,\quad t\in {\mathbb{R}^{+}},\quad \varepsilon >0,
\end{equation*}%
then for all $0\leq s<t$ and $\varepsilon >0$%
\begin{equation*}
\begin{array}{l}
|x_{t}^{\varepsilon }-\hat{x}_{t}^{\varepsilon }|^{2}-|x_{s}^{\varepsilon }-%
\hat{x}_{s}^{\varepsilon }|^{2}\leq \left[ |m_{t}-\hat{m}_{t}|^{2}-|m_{s}-%
\hat{m}_{s}|^{2}\right] \medskip \\
\displaystyle-2\int_{0}^{t}\langle m_{t}-\hat{m}_{t}-m_{r}+\hat{m}%
_{r},dk_{r}^{\varepsilon }-d\hat{k}_{r}^{\varepsilon }\rangle
+2\int_{0}^{s}\langle m_{s}-\hat{m}_{s}-m_{r}+\hat{m}_{r},dk_{r}^{%
\varepsilon }-d\hat{k}_{r}^{\varepsilon }\rangle .%
\end{array}%
\end{equation*}
\end{lemma}

\begin{proof}
Since $A_{\varepsilon }$ is monotone,%
\begin{equation}
\int_{s}^{t}\langle x_{r}^{\varepsilon }-\hat{x}_{r}^{\varepsilon
},A_{\varepsilon }(x_{r}^{\varepsilon })-A_{\varepsilon }(\hat{x}%
_{r}^{\varepsilon })\rangle dr\geq 0.  \label{eq3.3}
\end{equation}%
Consequently, the result follows by the arguments from the proof of Lemma %
\ref{lem2.7}-$(ii)$.\hfill \medskip
\end{proof}

We will study the problem of convergence of $\{x^{\varepsilon}\}$. It is
well known that if $m$ is continuous and $m_{0}\in\overline{D(A)}$ then, for
all $T\geq0,$%
\begin{equation*}
\left\Vert x^{\varepsilon}-x\right\Vert _{T}\longrightarrow0,~\text{as }%
\varepsilon\rightarrow0,
\end{equation*}
where $x=\mathcal{SP}^{(1)}\left( A;m\right) $ (see, e.g., R\u{a}\c{s}canu
\cite{ra/96}). In the case where $m$ has jumps the problem is more delicate.
We start with the simplest case where $m_{t}=\alpha\in\mathbb{H}$, $t\in{%
\mathbb{R}^{+}}$.

\begin{lemma}
\label{lem3.5}Let $\varepsilon>0,$ $\alpha_{\varepsilon},\alpha\in\mathbb{H}$
and $x_{t}^{\varepsilon}=\mathcal{SP}^{\left( 1\right) }\left(
A_{\varepsilon};\alpha_{\varepsilon}\right) _{t}=S_{A_{\varepsilon}}\left(
t\right) \alpha_{\varepsilon}$ with $\alpha_{\varepsilon}\rightarrow\alpha$
as $\varepsilon\rightarrow0$.

\noindent$\left( i\right) $ If $\alpha\in\overline{\mathrm{D}\left( A\right)
}$ and $x_{t}=\mathcal{SP}^{(1)}\left( A;\alpha\right) _{t}=S_{A}\left(
t\right) \alpha$ then, for all $T\geq0$,
\begin{equation*}
\left\Vert x^{\varepsilon}-x\right\Vert _{T}\longrightarrow0,~\text{as }%
\varepsilon\rightarrow0.
\end{equation*}

\noindent$\left( ii\right) $ If $\alpha\notin\overline{\mathrm{D}\left(
A\right) }$ and $x=\mathcal{SP}^{(1)}(A;\Pi_{\overline{\mathrm{D}\left(
A\right) }}\left( \alpha\right) )$ then for any $0<\delta\leq T$,%
\begin{equation*}
\left\Vert x^{\varepsilon}-x\right\Vert _{\left[ \delta,T\right]
}\longrightarrow0,~\text{as }\varepsilon\rightarrow0.
\end{equation*}
\end{lemma}

\begin{proof}
$(i)$ We have%
\begin{align*}
\left\vert x_{t}^{\varepsilon }-x_{t}\right\vert & =\left\vert
S_{A_{\varepsilon }}\left( t\right) \alpha _{\varepsilon }-S_{A}\left(
t\right) \alpha \right\vert \\
& \leq \left\vert \alpha _{\varepsilon }-\alpha \right\vert +\left\vert
S_{A_{\varepsilon }}\left( t\right) \alpha -S_{A}\left( t\right) \alpha
\right\vert .
\end{align*}%
The result follows using Proposition \ref{prop Barbu/84}.

$(ii)$ By Proposition \ref{prop 5} we know that for all $T>0$, $%
x^{\varepsilon }$ tends to $x$ in $L^{2}([0,T])$, as $\varepsilon
\rightarrow 0$. Hence $x_{\varepsilon }$ tends almost every where to $x$
with respect to the Lebesgue measure. Consequently, for any $\delta >0$
there is $t_{0}\in (0,\delta ]$ such that $x_{t_{0}}^{\varepsilon
}\rightarrow x_{t_{0}}$. Let $\hat{x}^{\varepsilon }=\mathcal{SP}^{\left(
1\right) }(A_{\varepsilon };\Pi _{\overline{\mathrm{D}\left( A\right) }%
}\left( \alpha \right) )$, $\varepsilon >0$. By Lemma \ref{lem3.3}%
\begin{equation*}
|x_{t}^{\varepsilon }-\hat{x}_{t}^{\varepsilon }|^{2}\leq
|x_{t_{0}}^{\varepsilon }-\hat{x}_{t_{0}}^{\varepsilon }|^{2},~t\geq
t_{0},~\varepsilon >0.
\end{equation*}%
Finally, from $(i)$, $\left\Vert \hat{x}^{\varepsilon }-x\right\Vert
_{T}\longrightarrow 0$, as $\varepsilon \rightarrow 0$, which completes the
proof.\hfill
\end{proof}

\begin{theorem}
\label{thm3.6}Assume that $m^{\varepsilon},m\in{\mathbb{D}}\left( {\mathbb{R}%
^{+}},{{\mathbb{H}}}\right) $, $m_{0}\in\overline{D(A)}$ and let $x=\mathcal{%
SP}^{(1)}(A,\Pi_{\overline{\mathrm{D}\left( A\right) }};m)$ and $%
x^{\varepsilon}=\mathcal{SP}^{\left( 1\right) }(A_{\varepsilon
};m^{\varepsilon})$ for each $\varepsilon>0$.

If for all $T>0,$%
\begin{equation*}
\left\Vert {m^{\varepsilon}-m}\right\Vert _{T}{\longrightarrow0},~\text{as }%
\varepsilon\rightarrow0,
\end{equation*}
then

\noindent$\left( j\right) $ for any $T>0$,
\begin{equation}
\lim_{\varepsilon\rightarrow0}\left\vert |x^{\varepsilon}-\bar{x}\right\vert
|_{T}=0,\text{ where }\bar{x}_{t}=x_{t}\mathbf{1}_{\left\vert \Delta
m_{t}\right\vert =0}+\left( x_{t-}+\Delta m_{t}\right) \mathbf{1}%
_{\left\vert \Delta m_{t}\right\vert >0}  \label{thm3.6_1}
\end{equation}
and consequently,%
\begin{equation*}
\lim_{\varepsilon\rightarrow0}\Big(\sup_{t\in\left[ 0,T\right] }\left\vert
x_{t-}^{\varepsilon}-x_{t-}\right\vert \Big)=0\,;
\end{equation*}
\noindent$\left( jj\right) $ for all $T>0,$
\begin{equation*}
\lim_{\varepsilon\rightarrow0}J_{\varepsilon}(x_{t}^{\varepsilon})=\Pi_{%
\overline{D(A)}}(x_{t-}+\Delta m_{t})=x_{t}\;,~\text{uniformly for }t\in%
\left[ 0,T\right] ;
\end{equation*}
\noindent$\left( jjj\right) $ for any sequence $t_{\varepsilon}\searrow t$
we have $x_{t_{\varepsilon}}^{\varepsilon}\longrightarrow x_{t}$ and for $%
t_{\varepsilon}\nearrow t$ we have $x_{t_{\varepsilon}}^{\varepsilon
}\longrightarrow x_{t-}\,.$
\end{theorem}

\begin{proof}
\noindent $\left( j\right) $ Since $m\in \mathbb{D}\left( \mathbb{R}^{+},%
\mathbb{H}\right) $, there exists a partition $\pi _{n}\in \mathcal{P}_{%
\mathbb{R}^{+}}$ such that $\max_{r\in \pi _{n}}\mathcal{O}_{m}\left(
[r,r^{\prime })\right) <1/n$ (see Remark \ref{remark Annex 6'}). Let $%
m_{t}^{\left( n\right) }=\sum_{r\in \pi _{n}}m_{r}\mathbf{1}_{[r,r^{\prime
})}\left( t\right) $ and $m_{t}^{\varepsilon ,\left( n\right) }=\sum_{r\in
\pi _{n}}m_{r}^{\varepsilon }\mathbf{1}_{[r,r^{\prime })}\left( t\right) $.
We have%
\begin{equation*}
||m-m^{\left( n\right) }||_{T}<\frac{1}{n}\text{ and }||m^{\varepsilon
,\left( n\right) }-m^{\left( n\right) }||_{T}\leq ||m^{\varepsilon }-m||_{T}
\end{equation*}%
and therefore%
\begin{equation*}
||m^{\varepsilon ,\left( n\right) }-m||_{T}\leq \left\Vert m^{\varepsilon
}-m\right\Vert _{T}+\frac{1}{n}\quad \text{and}\quad ||m^{\varepsilon
,\left( n\right) }-m^{\varepsilon }||_{T}\leq 2\left\Vert m^{\varepsilon
}-m\right\Vert _{T}+\frac{1}{n}\,.
\end{equation*}%
Let $x^{\varepsilon ,(n)}=\mathcal{SP}^{\left( 1\right) }(A_{\varepsilon
};m^{\varepsilon ,(n)})$ and $(x^{(n)},k^{(n)})=\mathcal{SP}(A;m^{(n)})$, $%
n\in {\mathbb{N}}$, $\varepsilon >0$. From Convergence Theorem \ref{prop2.11}
it follows%
\begin{equation}
||x^{(n)}-x||_{T}+||k^{(n)}-k||_{T}\rightarrow 0,\;\text{as }n\rightarrow
\infty ,  \label{thm3.6_2}
\end{equation}%
where $(x,k)=\mathcal{SP}(A;m).$

By Lemma \ref{lem3.3} (for $s=0$) we obtain%
\begin{equation}
|x_{t}^{\varepsilon }-x_{t}^{\varepsilon ,(n)}|^{2}\leq |m_{t}^{\varepsilon
}-m_{t}^{\varepsilon ,(n)}|^{2}+4||m^{\varepsilon }-m^{\varepsilon
,(n)}||_{t}(\left\updownarrow {k}^{\varepsilon }\right\updownarrow {_{t}}%
+\updownarrow {\hspace{-0.1cm}}k^{\varepsilon ,\left( n\right) }{\hspace{%
-0.1cm}}\updownarrow _{t}),  \label{thm3.6_3}
\end{equation}%
where%
\begin{equation*}
k^{\varepsilon ,(n)}=\int_{0}^{\cdot }A_{\varepsilon }(x_{s}^{\varepsilon
,(n)})ds,\quad k^{\varepsilon }=\int_{0}^{\cdot }A_{\varepsilon
}(x_{s}^{\varepsilon })ds,~n\in {\mathbb{N}},\;\varepsilon >0.
\end{equation*}%
Since%
\begin{align*}
||{m}^{\varepsilon ,\left( n\right) }||_{T}& \leq \frac{1}{n}%
+||m^{\varepsilon }-m||_{T}+||m||_{T}, \\
|{m}_{t}^{\varepsilon ,\left( n\right) }-{m}_{s}^{\varepsilon ,\left(
n\right) }|& \leq 2||m^{\varepsilon }-m||_{T}+\frac{2}{n}+|{m}_{t}-{m}_{s}|,
\end{align*}%
it is easy to see from the proof of Theorem \ref{prop2.11}, $\left( \mathbf{I%
}-j\right) $ that, there exist $\varepsilon _{0}>0$ and $n_{0}\in \mathbb{N}%
^{\ast }$ such that, for all $0<\varepsilon \leq \varepsilon _{0}$ and $%
n\geq n_{0},$%
\begin{equation*}
||x^{\varepsilon ,\left( n\right) }||_{T}^{2}+\updownarrow {\hspace{-0.1cm}{k%
}^{\varepsilon ,\left( n\right) }\hspace{-0.1cm}}\updownarrow {_{T}}\leq
C(1+\left\Vert m\right\Vert _{T}^{2}).
\end{equation*}%
From this inequality and (\ref{prop3.4}) we deduce that variations $%
\left\updownarrow {k}^{\varepsilon }\right\updownarrow {_{T}}$, $%
\updownarrow {\hspace{-0.1cm}{k}^{\varepsilon ,\left( i\right) }\hspace{%
-0.1cm}}\updownarrow _{T}$ are bounded for any $T\in {\mathbb{R}^{+},}$
hence from (\ref{thm3.6_3}) we obtain%
\begin{equation*}
|x_{t}^{\varepsilon }-x_{t}^{\varepsilon ,(n)}|^{2}\leq C\left( \left\Vert
m\right\Vert _{T}\right) \Big(||m^{\varepsilon
}-m||_{T}^{2}+||m^{\varepsilon }-m||_{T}+\frac{1}{n}\Big),~\forall t\in %
\left[ 0,T\right] .
\end{equation*}%
Taking into account the definitions we have for all $r\in \pi _{n},$%
\begin{align*}
x_{t}^{\varepsilon ,(n)}& =S_{A_{\varepsilon }}\left( t-r\right)
x_{r}^{\varepsilon ,(n)}\;,~\forall t\in (r,r^{\prime }), \\
x_{r}^{\varepsilon ,(n)}& =x_{r-}^{\varepsilon ,(n)}+\Delta
m_{r}^{\varepsilon ,\left( n\right) }
\end{align*}%
and%
\begin{align*}
x_{t}^{(n)}& =S_{A}\left( t-r\right) x_{r}^{(n)}\;,~\forall t\in \left(
r,r^{\prime }\right) , \\
x_{r}^{(n)}& =\Pi _{\overline{\mathrm{D}\left( A\right) }}\big(%
x_{r-}^{(n)}+\Delta m_{r}^{\left( n\right) }\big).
\end{align*}%
By Lemma \ref{lem3.5} it is clear that for any $r\in \pi _{n}$%
\begin{equation}
\lim_{\varepsilon \rightarrow 0}x_{t}^{\varepsilon ,(n)}=\left\{
\begin{array}{ll}
x_{t}^{(n)}, & \text{if }t\in \left( r,r^{\prime }\right) ,\medskip \\
x_{r-}^{(n)}+\Delta m_{r}^{(n)}, & \text{if }t=r%
\end{array}%
\right.  \label{eq3.6}
\end{equation}%
(and the convergence is uniformly with respect to $t\in \left[ 0,T\right] $,
for all $T>0$).

If $t\in \left( r,r^{\prime }\right) ,$%
\begin{equation*}
\left\vert x_{t}^{\varepsilon }-x_{t}\right\vert \leq |x_{t}^{\varepsilon
}-x_{t}^{\varepsilon ,(n)}|+|x_{t}^{\varepsilon ,(n)}-x_{t}^{\left( n\right)
}|+|x_{t}^{(n)}-x_{t}|,
\end{equation*}%
and then%
\begin{equation*}
\limsup_{\varepsilon \rightarrow 0}\Big(\max_{\substack{ r\in \pi _{n}  \\ %
r\leq T}}\sup_{t\in \left( r,r^{\prime }\right) }\left\vert
x_{t}^{\varepsilon }-x_{t}\right\vert \Big)\leq C\frac{1}{n}%
+||x^{(n)}-x||_{T}\,,\quad \forall n\in \mathbb{N}^{\ast }.
\end{equation*}%
If $t=r,$%
\begin{equation*}
\left\vert x_{r}^{\varepsilon }-x_{r-}-\Delta m_{r}\right\vert \leq
|x_{r}^{\varepsilon }-x_{r}^{\varepsilon ,(n)}|+|x_{r}^{\varepsilon
,(n)}-x_{r-}^{(n)}-\Delta m_{r}^{(n)}|+|x_{r-}^{(n)}+\Delta
m_{r}^{(n)}-x_{r-}-\Delta m_{r}|
\end{equation*}%
then%
\begin{equation*}
\limsup_{\varepsilon \rightarrow 0}\Big(\max_{\substack{ r\in \pi _{n}  \\ %
r\leq T}}\left\vert x_{r}^{\varepsilon }-x_{r-}-\Delta m_{r}\right\vert \Big)%
\leq C\frac{1}{n}+||x^{(n)}-x||_{T}\,,~\forall n\in \mathbb{N}^{\ast }.
\end{equation*}%
Let $\bar{x}_{t}:=x_{t}\mathbf{1}_{\left\vert \Delta m_{t}\right\vert
=0}+\left( x_{t-}+\Delta m_{t}\right) \mathbf{1}_{\left\vert \Delta
m_{t}\right\vert >0}$. Consequently%
\begin{equation*}
\sup_{t\in \left[ 0,T\right] }\left\vert x_{t}^{\varepsilon }-\bar{x}%
_{t}\right\vert \leq \max_{\substack{ r\in \pi _{n}  \\ r\leq T}}\left\vert
x_{r}^{\varepsilon }-x_{r-}-\Delta m_{r}\right\vert +\max_{\substack{ r\in
\pi _{n}  \\ r\leq T}}\sup_{t\in \left( r,r^{\prime }\right) }\left\vert
x_{t\wedge T}^{\varepsilon }-x_{t\wedge T}\right\vert
\end{equation*}%
and we get%
\begin{equation*}
\limsup_{\varepsilon \rightarrow 0}\left\vert |x^{\varepsilon }-\bar{x}%
\right\vert |_{T}\leq \delta _{n},\quad \forall n\in \mathbb{N}^{\ast },
\end{equation*}%
where $\delta _{n}\rightarrow 0$, as $n\rightarrow \infty $.

We recall that $x=\mathcal{SP}^{(1)}(A,\Pi _{\overline{\mathrm{D}\left(
A\right) }};m)$ satisfies, for all $t\geq 0,$%
\begin{equation*}
\left\vert \Delta x_{t}\right\vert \leq \left\vert \Delta m_{t}\right\vert
\quad \text{and}\quad x_{t}=\Pi _{\overline{\mathrm{D}\left( A\right) }%
}\left( x_{t-}+\Delta m_{t}\right) \in \overline{\mathrm{D}\left( A\right) }%
\,.
\end{equation*}%
\noindent $(jj)$ Since $\lim_{\varepsilon \rightarrow 0}x_{t}^{\varepsilon }=%
\bar{x}_{t}$, by Proposition \ref{prop Anex 4}-$\left( jj\right) $, we have
uniformly on every interval $\left[ 0,T\right] $,%
\begin{equation*}
J_{\varepsilon }(x_{t}^{\varepsilon })\rightarrow \Pi _{\overline{D(A)}%
}(x_{t-}+\Delta m_{t})=x_{t},\quad t\in \left[ 0,T\right] .
\end{equation*}%
\noindent $\left( jjj\right) $ Is a immediate consequence of the uniform
convergence (\ref{thm3.6_1}).\hfill
\end{proof}

\begin{theorem}
\label{thm3.7}Assume that $m,m^{\varepsilon }\in {\mathbb{D}}\left( {\mathbb{%
R}^{+}},{{\mathbb{H}}}\right) $, $\varepsilon >0$, $m_{0}\in \overline{D(A)}$
and let $x^{\varepsilon }=\mathcal{SP}^{\left( 1\right) }\left(
A_{\varepsilon };m^{\varepsilon }\right) $. If {$m$}${^{\varepsilon
}\longrightarrow }${$m$} in ${\mathbb{D}}\left( {\mathbb{R}^{+}},{{\mathbb{H}%
}}\right) $, as $\varepsilon \rightarrow 0$, then\medskip

\noindent $\left( j\right) $ for any $t\in {\mathbb{R}^{+}}$, $%
x_{t}^{\varepsilon }\longrightarrow x_{t},~$provided that $\Delta m_{t}=0$
and there exist $C>0$ and $\varepsilon _{0}>0$ such that for all $%
0<\varepsilon \leq \varepsilon _{0},$%
\begin{equation}
\int_{0}^{t}|A_{\varepsilon }(x_{s}^{\varepsilon })|ds\leq C<+\infty \,.
\label{thm3.7_2}
\end{equation}

\noindent $\left( jj\right) $ ${(J_{\varepsilon }(x^{\varepsilon
}),m^{\varepsilon })\longrightarrow (x,m)}$ in ${\mathbb{D}}\left( {\mathbb{R%
}^{+}},{{\mathbb{H}}}\times {{\mathbb{H}}}\right) .$\medskip

\noindent$\left( jjj\right) $ if $m$ is continuous and $x=\mathcal{SP}%
^{(1)}(A;m)$ then%
\begin{equation*}
||x^{\varepsilon}-x||_{T}\longrightarrow0,\quad T\in{\mathbb{R}^{+}}.
\end{equation*}
\end{theorem}

\begin{proof}
$\left( j\right) $ Since {$m$}${^{\varepsilon }\longrightarrow }${$m$} in ${%
\mathbb{D}}\left( {\mathbb{R}^{+}},{{\mathbb{H}}}\right) $, there exist $%
\lambda ^{\varepsilon }:\mathbb{R}^{+}\rightarrow \mathbb{R}^{+}$, a
strictly increasing continuous changes of time such that $\lambda
_{0}^{\varepsilon }=0$, $\lambda _{\infty }^{\varepsilon }=+\infty $, $%
\sup_{t\geq 0}|\lambda _{t}^{\varepsilon }-t|\rightarrow 0$ and $%
||m_{\lambda ^{\varepsilon }}^{\varepsilon }-m||_{T}\rightarrow 0$ as $%
\varepsilon \rightarrow 0$, for all $T\in {\mathbb{R}^{+}}$. We have $%
(x_{\lambda ^{\varepsilon }}^{\varepsilon },k_{\lambda ^{\varepsilon
}}^{\varepsilon })=\mathcal{SP}(A_{\varepsilon };m_{\lambda ^{\varepsilon
}}^{\varepsilon })$, where $k_{u}^{\varepsilon }=\int_{0}^{u}A_{\varepsilon
}\left( x_{s}^{\varepsilon }\right) ds$. Let $T>0$ be arbitrary fixed and $%
t\in \left[ 0,T\right] $. Then there exits $\varepsilon _{0}>0$ such that $%
\lambda _{T+1}^{\varepsilon }\geq T$ and, by (\ref{prop3.4}),%
\begin{equation*}
\int_{0}^{t}|A_{\varepsilon }(x_{s}^{\varepsilon })|ds\leq \int_{0}^{\lambda
_{T+1}^{\varepsilon }}|A_{\varepsilon }(x_{s}^{\varepsilon })|ds\leq C\left(
1+||m||_{T+1}\right) ,\;\forall 0<\varepsilon \leq \varepsilon _{0},
\end{equation*}%
which is (\ref{thm3.7_2}). Moreover, by Theorem \ref{thm3.6}, the
convergence follows.

\noindent Conclusion $\left( jj\right) $ follows from the assumption {$m$}${%
^{\varepsilon }\longrightarrow }${$m$} in ${\mathbb{D}}\left( {\mathbb{R}^{+}%
},{{\mathbb{H}}}\right) $ and from Theorem \ref{thm3.6}-$\left( jj\right) $.

\noindent $\left( jjj\right) $ By Remark \ref{remark Annex 6''} we deduce $%
||m^{\varepsilon }-m||_{T}\rightarrow 0$ as $\varepsilon \rightarrow 0$,
which implies $(iii)$, via Theorem \ref{thm3.6}.\hfill
\end{proof}

\begin{remark}
\label{rem3.8}Note that Theorem \ref{thm3.7}-$(j)$ implies that $%
x^{\varepsilon}$ tends to $x$ in the $S$ topology introduced by Jakubowski
in \cite{ja/97}.
\end{remark}

\subsection{Approximation with amortized large jumps}

If in the first case we considered an approximation with free jumps, this
time we define an approximation absorbing the too large jumps with the help
of a generalized projection.

Let $m\in {\mathbb{D}}\left( {\mathbb{R}^{+}},{{\mathbb{H}}}\right) $, $%
m_{0}\in \overline{D(A)}$. We will consider for $0<\varepsilon \leq 1$ the
approximating equation of the form%
\begin{equation}
\begin{array}{rl}
\left( i\right) & x^{\varepsilon }\in {\mathbb{D}}\left( {\mathbb{R}^{+}},{{%
\mathbb{H}}}\right) ,\medskip \\
\left( ii\right) & k^{\varepsilon }\in {\mathbb{D}}\left( {\mathbb{R}^{+}},{{%
\mathbb{H}}}\right) \cap \mathrm{BV}_{loc}\left( \mathbb{R}^{+};\mathbb{H}%
\right) ,\quad k_{0}^{\varepsilon }=0\text{,}\medskip \\
\left( iii\right) & x_{t}^{\varepsilon }+k_{t}^{\varepsilon }=m_{t},\quad
\forall \text{\ }t\geq 0,\medskip \\
\left( iv\right) & k^{\varepsilon }=k^{\varepsilon ,c}+k^{\varepsilon ,d},%
\text{\ }k_{t}^{\varepsilon ,d}=\sum_{0\leq s\leq t}\Delta
k_{s}^{\varepsilon },\medskip \\
\left( v\right) & \displaystyle k_{t}^{\varepsilon
,c}:=\int_{0}^{t}A_{\varepsilon }(x_{s}^{\varepsilon })ds,\quad \forall
\text{\ }t\geq 0,\medskip \\
\left( vi\right) & x_{t}^{\varepsilon }=(x_{t-}^{\varepsilon }+\Delta m_{t})%
\mathbf{1}_{|\Delta m_{t}|\leq \varepsilon }+\Pi (x_{t-}^{\varepsilon
}+\Delta m_{t})\mathbf{1}_{|\Delta m_{t}|>\varepsilon },\quad \forall \text{%
\ }t\geq 0.%
\end{array}
\label{eq3.8}
\end{equation}%
Clearly%
\begin{equation*}
\Delta k_{t}^{\varepsilon }=\Delta k_{t}^{\varepsilon
,d}=\{x_{t-}^{\varepsilon }+\Delta m_{t}-\Pi (x_{t-}^{\varepsilon }+\Delta
m_{t})\}\mathbf{1}_{|\Delta m_{t}|>\varepsilon }~.
\end{equation*}

We will call $x^{\varepsilon }$ the solution of the Yosida problem
associated with the projection $\Pi $ and we will use the notation $\left(
x^{\varepsilon },k^{\varepsilon }\right) =\mathcal{YP}(A_{\varepsilon },\Pi
;m)$, $\varepsilon >0$. Set $t_{0}=0$, $t_{k+1}=\inf \{t>t_{k}:|\Delta
m_{t}|>\varepsilon \}$, $k\in {\mathbb{N}}$ and observe that on every
interval $[t_{k},t_{k+1})$, $x^{\varepsilon }$ satisfies the equation%
\begin{equation*}
x_{t}^{\varepsilon }+\int_{t_{k}}^{t}A_{\varepsilon }(x_{s}^{\varepsilon
})ds=\Pi (x_{t_{k}-}^{\varepsilon }+\Delta m_{t_{k}})+m_{t}-m_{t_{k}}\,,
\end{equation*}%
which has a solution from Theorem \ref{thm2.13} since $A_{\varepsilon }$ is
maximal monotone operator.

The solution is%
\begin{equation}
x_{t}^{\varepsilon }=\left\{
\begin{array}{ll}
\mathcal{SP}^{\left( 1\right) }(A_{\varepsilon };m)_{t}, & t\in \lbrack
0,t_{1}),\medskip \\
\mathcal{SP}^{\left( 1\right) }(A_{\varepsilon };\Pi (x_{t_{{k}-}}+\Delta
m_{t_{k}})+m_{t_{k}+\cdot }-m_{t_{k}})_{t-t_{k}}, & t\in \lbrack
t_{k},t_{k+1}),\,k\in {\mathbb{N}}.%
\end{array}%
\right.  \label{eq3.9}
\end{equation}%
We can now state a simmilar result as in Remark \ref{lem3.2}.

\begin{lemma}
\label{lem3.9}If $m$ is a step function of the form $m_{t}=\sum_{r\in%
\pi}m_{r}\mathbf{1}_{[r,r^{\prime})}\left( t\right) $, where $\pi\in
\mathcal{P}_{\mathbb{R}^{+}}$ is a partition and $m_{0}\in\overline{D(A)}$,
then $x^{\varepsilon}=\mathcal{YP}(A_{\varepsilon},\Pi;m)$ is given by: for
all $r\in\pi$ and $t\in\lbrack r,r^{\prime}),$%
\begin{equation}
x_{t}^{\varepsilon}=\left\{
\begin{array}{ll}
\mathcal{YP}(A_{\varepsilon};x_{r-}+\Delta m_{r})_{t-r}\,, & \text{if }%
|\Delta m_{r}|\leq\varepsilon,\medskip \\
\mathcal{YP}(A_{\varepsilon};\Pi(x_{r-}+\Delta m_{r}))_{t-r}\,, & \text{if }%
|\Delta m_{r}|>\varepsilon.%
\end{array}
\right.  \label{eq3.10}
\end{equation}
\end{lemma}

\begin{proposition}
\label{prop3.10}Let $m,\hat{m}\in{\mathbb{D}}\left( {\mathbb{R}^{+}},{{%
\mathbb{H}}}\right) $ and $m_{0},\hat{m}_{0}\in\overline{D(A)}$.

If $\left( x^{\varepsilon },k^{\varepsilon }\right) =\mathcal{YP}%
(A_{\varepsilon },\Pi ;m)$ and $(\hat{x}^{\varepsilon },\hat{k}^{\varepsilon
})=\mathcal{YP}(A_{\varepsilon },\Pi ;\hat{m})$ then\medskip

\noindent$\left( i\right) $ for any $t\in{\mathbb{R}^{+}}$ and $%
\varepsilon>0 $
\begin{equation*}
|x_{t}^{\varepsilon}-\hat{x}_{t}^{\varepsilon}|^{2}\leq|m_{t}-\hat{m}%
_{t}|^{2}-2\int_{0}^{t}\langle m_{t}-\hat{m}_{t}-m_{s}+\hat{m}%
_{s},dk_{s}^{\varepsilon}-d\hat{k}_{s}^{\varepsilon}\rangle.
\end{equation*}

\noindent$\left( ii\right) $ for any $a\in\mathrm{Int}\left( \mathrm{D}%
\left( A\right) \right) $ and $T\in{\mathbb{R}^{+}}$ there exist $%
\varepsilon_{0}>0$ and $C>0$ such that for any $0<\varepsilon\leq
\varepsilon_{0},$%
\begin{equation*}
||x^{\varepsilon}||_{T}^{2}+\left\updownarrow {k}^{\varepsilon}\right%
\updownarrow {_{T}}\leq C(1+||m||_{T}^{2}).
\end{equation*}
\end{proposition}

\begin{proof}
$(i)$ By (\ref{eq3.3})%
\begin{equation*}
\int_{0}^{t}\langle x_{s}^{\varepsilon }-\hat{x}_{s}^{\varepsilon
},dk_{s}^{\varepsilon ,c}-d\hat{k}_{s}^{\varepsilon ,c}\rangle \geq 0,~t\in {%
\mathbb{R}^{+}}.
\end{equation*}%
Using Remark \ref{lem2.4'}-$\left( i\right) $, it follows that%
\begin{equation*}
\int_{0}^{t}\langle x_{s}^{\varepsilon }-\hat{x}_{s}^{\varepsilon
},dk_{s}^{\varepsilon ,d}-d\hat{k}_{s}^{\varepsilon ,d}\rangle +\frac{1}{2}%
\sum_{s\leq t}|\Delta k_{s}^{\varepsilon }-\Delta \hat{k}_{s}^{\varepsilon
}|^{2}\geq 0,~t\in {\mathbb{R}^{+}}.
\end{equation*}%
Consequently,%
\begin{equation*}
\int_{0}^{t}\langle x_{s}^{\varepsilon }-\hat{x}_{s}^{\varepsilon
},dk_{s}^{\varepsilon }-d\hat{k}_{s}^{\varepsilon }\rangle +\frac{1}{2}%
\sum_{s\leq t}|\Delta k_{s}^{\varepsilon }-\Delta \hat{k}_{s}^{\varepsilon
}|^{2}\geq 0,\quad t\in {\mathbb{R}^{+}}.
\end{equation*}%
Using the same arguments as in the proof of Lemma \ref{lem2.7}-$(ii)$ we
obtain $(i)$.

\noindent $(ii)$ It is sufficient to use (\ref{prop3.4}), to observe that,
as in Remark \ref{lem2.4'}-$\left( ii\right) $,%
\begin{equation*}
r_{0}|\Delta k_{s}^{\varepsilon }|\leq \langle x_{s}^{\varepsilon }-a,\Delta
k_{s}^{\varepsilon }\rangle +\frac{1}{2}|\Delta k_{s}^{\varepsilon
}|^{2},~s\in {\mathbb{R}^{+}}
\end{equation*}%
and to follow the proof of Theorem \ref{prop2.11}.\hfill
\end{proof}

\begin{theorem}
\label{thm3.11}Assume that $m^{\varepsilon }\in {\mathbb{D}}\left( {\mathbb{R%
}^{+}},{{\mathbb{H}}}\right) $, $m_{0}^{\varepsilon }\in \overline{D(A)}$, $%
\left( x^{\varepsilon },k^{\varepsilon }\right) =\mathcal{YP}(A_{\varepsilon
},\Pi ;m^{\varepsilon })$, $\varepsilon >0$ and $\left( x,k\right) =\mathcal{%
SP}(A,\Pi ;m)$.\medskip

\noindent $\left( j\right) $ If $||m^{\varepsilon }-m||_{T}\longrightarrow 0$%
, $T\in {\mathbb{R}^{+},}$ then
\begin{equation*}
||x^{\varepsilon }-x||_{T}\longrightarrow 0,~T\in {\mathbb{R}^{+}}.
\end{equation*}

\noindent $\left( jj\right) $ If $m^{\varepsilon }\longrightarrow m$ in ${%
\mathbb{D}}\left( {\mathbb{R}^{+}},{{\mathbb{H}}}\right) ,$ then%
\begin{equation*}
(x^{\varepsilon },m^{\varepsilon })\longrightarrow (x,m)\quad \text{in }{%
\mathbb{D}}\left( {\mathbb{R}^{+}},{{\mathbb{H}}}\times \mathbb{H}\right) .
\end{equation*}
\end{theorem}

\begin{proof}
We use the notation from the proof of Theorem \ref{thm3.6}. Let $%
(x^{(n)},k^{(n)})=\mathcal{SP}(A,\Pi ;m^{(n)})$, $(x^{\varepsilon
,(n)},k^{\varepsilon ,(n)})=\mathcal{YP}(A_{\varepsilon },\Pi
;m^{\varepsilon ,(n)})$, $n\in {\mathbb{N}}$, $\varepsilon >0$. Following
the same step as in the proof of Theorem \ref{thm3.6} we deduce that, for
all $T\in {\mathbb{R}^{+}},~n\in {\mathbb{N}}$,%
\begin{equation*}
\lim_{\varepsilon \rightarrow 0}||x^{\varepsilon
,(n)}-x^{(n)}||_{T}\rightarrow 0.
\end{equation*}%
Using the same estimates for the differences $m^{\left( n\right) }-m$ and $%
m^{\varepsilon ,\left( n\right) }-m^{\left( n\right) }$ and the arguments
used in the proof of Theorem \ref{thm3.6}-$(j)$ the conclusion $\left(
j\right) $ follows and, in addition, $(jj)$.\hfill
\end{proof}

\begin{corollary}
\label{cor3.13}Let $m\in{\mathbb{D}}\left( {\mathbb{R}^{+}},{{\mathbb{H}}}%
\right) $, $m_{0}\in\overline{D(A)}$ and let $\pi_{\varepsilon}\in\mathcal{P}%
_{\mathbb{R}^{+}}$ be a sequence of partitions. If $m_{t}^{(\varepsilon)}=%
\sum_{r\in\pi_{\varepsilon}}m_{r}\mathbf{1}_{[r,r^{\prime})}\left( t\right) $%
, for $\varepsilon>0$, denotes the sequence of discretizations of $m$ and $%
(x^{(\varepsilon)},k^{\left( \varepsilon\right) })=\mathcal{YP}%
(A_{\varepsilon},\Pi;m^{(\varepsilon)})$, $\varepsilon>0$ then%
\begin{equation*}
(x^{(\varepsilon)},m^{(\varepsilon)})\longrightarrow(x,m)~\text{in }{\mathbb{%
D}}({\mathbb{R}^{+}},\mathbb{H\times H}),
\end{equation*}
where $(x,k)=\mathcal{SP}(A,\Pi;m)$.

Moreover, if $||m^{(\varepsilon)}-m||_{T}\longrightarrow0$, $T\in {\mathbb{R}%
^{+}}$ then
\begin{equation*}
||x^{(\varepsilon)}-x||_{T}\longrightarrow0,~T\in{\mathbb{R}^{+}}.
\end{equation*}
\end{corollary}

\begin{proof}
It is sufficient to remark that $m^{(\varepsilon )}\longrightarrow m$ in ${%
\mathbb{D}}\left( {\mathbb{R}^{+}},{{\mathbb{H}}}\right) $ and to apply
Theorem \ref{thm3.11}.\hfill
\end{proof}

\section{Annexes}

\subsection{Maximal monotone operators on Hilbert spaces}

\label{Maximal monotone}Let $\mathbb{H}$ be a separable real Hilbert space
with the inner product denoted by $\left\langle \cdot ,\cdot \right\rangle $
and the induced norm denoted by $\left\vert \cdot \right\vert $. A
multivalued operator $A:\mathbb{H}\rightrightarrows \mathbb{H}$ (a
point-to-set operator $A:\mathbb{H}\rightarrow 2^{\mathbb{H}}$) will be seen
also as a subset of $\mathbb{H}\times \mathbb{H}$. In fact, we formally
identify the multivalued operator $A:\mathbb{H}\rightrightarrows \mathbb{H}$
with its graph $\mathrm{Gr}\left( A\right) =\left\{ \left( x,y\right) \in
\mathbb{H}\times \mathbb{H}:y\in A\left( x\right) \right\} .$ We denote by $%
\mathrm{D}(A):=\left\{ x\in \mathbb{H}:\,Ax\neq \emptyset \right\} $ the
domain of $A$. We define $A^{-1}:\mathbb{H}\rightrightarrows \mathbb{H}$ to
be the point-to-set operator given by: $x\in A^{-1}\left( y\right) $ if $%
y\in A\left( x\right) $.

We recall some definitions:

\begin{itemize}
\item $A:\mathbb{H}\rightrightarrows\mathbb{H}$ is monotone if $\left\langle
v-y,u-x\right\rangle \geq0,$ for all $\,\left( x,y\right) ,\left( u,v\right)
\in A.$

\item $A:\mathbb{H}\rightrightarrows \mathbb{H}$ is a maximal monotone
operator if $A$ is a monotone operator and it is\ maximal in the set of
monotone operators: that is,
\begin{equation*}
\left\langle v-y,u-x\right\rangle \geq 0,\;\forall \left( x,y\right) \in
A\quad \Rightarrow \quad \left( u,v\right) \in A.
\end{equation*}
\end{itemize}

\begin{proposition}
\label{prop Anex 2}Let $A:\mathbb{H}\rightrightarrows\mathbb{H}$ be a
maximal monotone operator. Then

\noindent $\left( j\right) $ $\overline{\mathrm{D}\left( A\right) }$ is a
convex subset of $\mathbb{H}.$

\noindent $\left( jj\right) $ $Ax$ is a convex closed subset of $\mathbb{H}$%
, for all $x\in \mathrm{D}\left( A\right) .$

\noindent $\left( jjj\right) $ $A$ is locally bounded on $\mathrm{Int}\left(
\mathrm{D}\left( A\right) \right) $ that is: for every $u_{0}\in \mathrm{Int}%
\left( \mathrm{D}\left( A\right) \right) $ there exists $r_{0}>0$ such that%
\begin{equation*}
\overline{B\left( u_{0},r_{0}\right) }:=\left\{ u_{0}+r_{0}v:\left\vert
v\right\vert \leq 1\right\} \subset \mathrm{D}\left( A\right)
\end{equation*}%
and%
\begin{equation*}
A_{u_{0},r_{0}}^{\#}:=\sup \left\{ \left\vert \hat{u}\right\vert :\hat{u}\in
A\left( u_{0}+r_{0}v\right) ,\;\left\vert v\right\vert \leq 1\right\}
<\infty .
\end{equation*}
\end{proposition}

The maximal monotone operator $A$ is uniquely defined by its principal
section $A^{0}x:=\Pi_{Ax}\left( 0\right) $ in the following sense: if $%
\left( x,y\right) \in\overline{\mathrm{D}\left( A\right) }\times\mathbb{H}$
such that%
\begin{equation*}
\left\langle y-A^{0}u,x-u\right\rangle \geq0,\text{ for all }u\in \mathrm{D}%
\left( A\right)
\end{equation*}
then $\left( x,y\right) \in A.$

For each $\varepsilon >0$ the operators%
\begin{equation*}
J_{\varepsilon }x=(I+\varepsilon A)^{-1}(x)\text{ and }A_{\varepsilon
}\left( x\right) =\frac{1}{\varepsilon }(x-J_{\varepsilon }x),
\end{equation*}%
from $\mathbb{H}$ to $\mathbb{H}$ are single-valued. The operator $%
A_{\varepsilon }$ is called Yosida's approximation of the maximal monotone
operator $A.\;$In \cite{ba/76, ba/10} and \cite{br/73} can be found the
proofs of the following properties.

\begin{proposition}
\label{prop Anex 4}Let $A:\mathbb{H}\rightrightarrows\mathbb{H}$ be a
maximal monotone operator. Then:

\noindent $\left( j\right) $ for all $\varepsilon >0$ and for all$\;x,y\in
\mathbb{H}$%
\begin{equation*}
\begin{array}{rl}
\left( i\right) & \left( J_{\varepsilon }x,A_{\varepsilon }x\right) \in
A,\medskip \\
\left( ii\right) & \left\vert J_{\varepsilon }x-J_{\varepsilon }y\right\vert
\leq \left\vert x-y\right\vert ,\medskip \\
\left( iii\right) & \left\vert A_{\varepsilon }x-A_{\varepsilon
}y\right\vert \leq \dfrac{1}{\varepsilon }\left\vert x-y\right\vert ,\medskip
\\
\left( iv\right) & A_{\varepsilon }:\mathbb{H}\rightarrow \mathbb{H}\quad
\text{is a maximal monotone operator;}%
\end{array}%
\end{equation*}

\noindent $\left( jj\right) $ if $x_{\varepsilon },x\in \mathbb{H}$ and $%
\lim\limits_{\varepsilon \searrow 0}x_{\varepsilon }=x$, then $%
\lim\limits_{\varepsilon \searrow 0}J_{\varepsilon }x_{\varepsilon }=\Pi _{%
\overline{\mathrm{D}\left( A\right) }}\left( x\right) ,\;$for all $x\in
\mathbb{H}$, where $\Pi _{\overline{\mathrm{D}\left( A\right) }}\left(
x\right) $ is the orthogonal projection of $x$ on $\overline{\mathrm{D}%
\left( A\right) }\,;$

\noindent $\left( jjj\right) $ $\lim\limits_{\varepsilon \searrow
0}A_{\varepsilon }x=A^{0}x\in Ax,\;$for all $x\in \mathrm{D}\left( A\right)
\,;$

\noindent $\left( jv\right) $ $\left\vert A_{\varepsilon }x\right\vert $ is
monotone decreasing in $\varepsilon >0$, and when $\varepsilon \searrow 0$%
\begin{equation*}
\left\vert A_{\varepsilon }\left( x\right) \right\vert \nearrow \left\{
\begin{array}{ll}
\left\vert A^{0}\left( x\right) \right\vert , & \text{if}~x\in \mathrm{D}%
\left( A\right) ,\medskip \\
+\infty , & \text{if}~x\notin \mathrm{D}\left( A\right) \,;%
\end{array}%
\right.
\end{equation*}

\noindent$\left( v\right) $ For all $x,y\in\mathbb{H}$%
\begin{equation*}
\langle
x-y,A_{\varepsilon}(x)-A_{\varepsilon}(y)\rangle\geq\varepsilon\left(
|A_{\varepsilon}(x)|^{2}+|A_{\varepsilon}(y)|^{2}-2\langle A_{\varepsilon
}(x),A_{\varepsilon}(y)\rangle\right) \geq0.
\end{equation*}
\end{proposition}

\begin{proposition}
\label{prop Anex 5}Let $A:\mathbb{H}\rightrightarrows \mathbb{H}$ be a
maximal monotone operator. Let $r_{0}\geq 0$ and $\overline{B\left(
a,r_{0}\right) }\subset \mathrm{D}\left( A\right) .$ Then for all $x\in
\mathrm{D}\left( A\right) $, $\hat{x}\in A\left( x\right) $ and $\hat{u}\in
A\left( a+r_{0}u\right) ,$ where $u=0,$ if $\hat{x}=0$ and $u=\frac{\hat{x}}{%
\left\vert \hat{x}\right\vert }$ if $\hat{x}\neq 0,$ the following
inequalities hold:%
\begin{equation}
\begin{array}{rl}
\left( j\right) & r_{0}\left\vert \hat{x}\right\vert \leq \left\langle \hat{x%
},x-a\right\rangle +\left\vert \hat{u}\right\vert \left\vert x-a\right\vert
+r_{0}\left\vert \hat{u}\right\vert ,\medskip \\
\left( jj\right) & r_{0}\left\vert \hat{x}\right\vert \leq \left\langle \hat{%
x},x-a\right\rangle +A_{x_{0},r_{0}}^{\#}\left\vert x-a\right\vert
+r_{0}A_{x_{0},r_{0}}^{\#}\,,%
\end{array}
\label{prop Anex 5_1}
\end{equation}%
where $A_{x_{0},r_{0}}^{\#}:=\sup \left\{ |\hat{u}|:\hat{u}\in
A(a+r_{0}v),\;|v|\leq 1\right\} $.

\noindent Moreover for all $\varepsilon \in (0,1]$ and $z\in \mathbb{H}:$%
\begin{equation}
\begin{array}{rl}
\left( jjj\right) & r_{0}\left\vert A_{\varepsilon }z\right\vert \leq
\left\langle A_{\varepsilon }z,z-a\right\rangle +\left\vert \hat{u}%
\right\vert \left\vert z-a\right\vert +\left( \left\vert A^{0}\left(
a\right) \right\vert +r_{0}\right) \left\vert \hat{u}\right\vert ,\medskip
\\
\left( jv\right) & r_{0}\left\vert A_{\varepsilon }z\right\vert \leq
\left\langle A_{\varepsilon }z,z-a\right\rangle
+A_{x_{0},r_{0}}^{\#}\left\vert z-a\right\vert +\left( \left\vert
A^{0}\left( a\right) \right\vert +r_{0}\right) A_{x_{0},r_{0}}^{\#}\,.%
\end{array}
\label{prop Anex 5_2}
\end{equation}
\end{proposition}

\begin{proof}
By monotonicity of $A$ we have
\begin{align*}
r_{0}\left\langle \hat{x},u\right\rangle & \leq r_{0}\left\langle \hat{x}%
,u\right\rangle +\left\langle \hat{x}-\hat{u},x-\left( a+r_{0}u\right)
\right\rangle =\left\langle \hat{x},x-a\right\rangle -\left\langle \hat{u}%
,x-a\right\rangle +r_{0}\left\langle \hat{u},u\right\rangle \\
& \leq \left\langle \hat{x},x-a\right\rangle +\left\vert \hat{u}\right\vert
\left\vert x-a\right\vert +r_{0}\left\vert \hat{u}\right\vert ~,
\end{align*}%
that is (\ref{prop Anex 5_1}). Now since $A_{\varepsilon }\left( x\right)
\in A\left( J_{\varepsilon }\left( x\right) \right) $,%
\begin{align*}
r_{0}\left\vert A_{\varepsilon }x\right\vert & \leq \left\langle
A_{\varepsilon }x,J_{\varepsilon }\left( x\right) -a\right\rangle
+\left\vert \hat{u}\right\vert \left\vert J_{\varepsilon }\left( x\right)
-a\right\vert +r_{0}\left\vert \hat{u}\right\vert \\
& \leq \left\langle A_{\varepsilon }x,x-a\right\rangle +\left\vert \hat{u}%
\right\vert \left[ \left\vert J_{\varepsilon }\left( x\right)
-J_{\varepsilon }\left( a\right) \right\vert +\left\vert J_{\varepsilon
}\left( a\right) -a\right\vert \right] +r_{0}\left\vert \hat{u}\right\vert \\
& \leq \left\langle A_{\varepsilon }x,x-a\right\rangle +\left\vert \hat{u}%
\right\vert \left\vert x-a\right\vert +\left\vert \hat{u}\right\vert
\left\vert A^{0}\left( a\right) \right\vert +r_{0}\left\vert \hat{u}%
\right\vert ~.
\end{align*}%
\hfill \medskip
\end{proof}

If $\varphi :\mathbb{H}\rightarrow (-\infty ,+\infty ]$ is a proper convex,
lower semicontinuous function then the sub{\footnotesize \-}di{\footnotesize %
\-}fferential operator $A=\partial \varphi :\mathbb{H}\rightrightarrows
\mathbb{H}$ defined by
\begin{equation*}
\partial \varphi (x):=\{y\in \mathbb{H}:\langle \hat{x},z-x\rangle +\varphi
(x)\leq \varphi (z),\text{ for all }\,z\in \mathbb{H}\}
\end{equation*}%
is a maximal monotone operator on $\mathbb{H}$. In this case $A_{\varepsilon
}\left( x\right) =\nabla \varphi _{\varepsilon }\left( x\right) $ and $%
J_{\varepsilon }\left( x\right) =x-\varepsilon \nabla \varphi _{\varepsilon
}\left( x\right) ,$ where $\varphi _{\varepsilon }$ is the Moreau-Yosida
regularization of $\varphi $%
\begin{equation*}
\varphi _{\varepsilon }\left( x\right) :=\inf \left\{ \frac{\left\vert
x-z\right\vert ^{2}}{2\varepsilon }+\varphi \left( z\right) :z\in \mathbb{H}%
\right\} .
\end{equation*}%
If $\overline{D}$ is a nonempty closed convex subset of $\mathbb{H},$ then
the convexity indicator function $I_{\overline{D}}:\mathbb{H}\rightarrow
(-\infty ,+\infty ],$%
\begin{equation*}
I_{\overline{D}}\left( x\right) :=\left\{
\begin{array}{ll}
0, & \text{if }x\in \overline{D}, \\
+\infty , & \text{if }x\in \mathbb{H}\setminus \overline{D},%
\end{array}%
\right.
\end{equation*}%
is a proper convex lower semicontinuous and%
\begin{equation*}
\partial {I}_{D}(x)=\left\{
\begin{array}{ll}
\mathcal{N}_{\overline{D}}(x)=\{\nu :\langle \nu ,z-x\rangle \leq 0,\text{
for all }\,z\in \overline{D}\}, & \text{if }x\in \overline{D},\medskip \\
\emptyset , & \text{if }x\notin \overline{D},%
\end{array}%
\right.
\end{equation*}%
where $\mathcal{N}_{D}(x)=\left\{ 0\right\} $ if $x\in \mathrm{Int}\left(
\overline{D}\right) $ and $\mathcal{N}_{D}(x)$ is the closed external normal
cone to $\overline{D}$ if $x\in \mathrm{Bd}\left( \overline{D}\right) $.
Associated to the maximal operator $A=\partial {I}_{D}$ we have%
\begin{equation*}
J_{\varepsilon }\left( x\right) =\Pi _{\overline{D}}\left( x\right) \quad
\text{and}\quad A_{\varepsilon }\left( x\right) =\frac{1}{\varepsilon }%
\left( x-\Pi _{\overline{D}}\left( x\right) \right) \in \partial {I}%
_{D}\left( \Pi _{\overline{D}}\left( x\right) \right) .
\end{equation*}%
Let $r_{0}\geq 0$ and $\overline{B\left( a,r_{0}\right) }\subset \overline{D}%
\,.$ By Proposition \ref{prop Anex 5}, for all $z\in \mathbb{H},$%
\begin{equation}
r_{0}\left\vert z-\Pi _{\overline{D}}\left( z\right) \right\vert \leq
\left\langle z-\Pi _{\overline{D}}\left( z\right) ,\Pi _{\overline{D}}\left(
z\right) -a\right\rangle =-\left\vert z-\Pi _{\overline{D}}\left( z\right)
\right\vert ^{2}+\left\langle z-\Pi _{\overline{D}}\left( z\right)
,z-a\right\rangle ,  \label{prop Anex 5_3}
\end{equation}%
since $\hat{u}=0\in \partial {I}_{D}\left( a+r_{0}u\right) $.

\subsection{Skorohod space}

Let $\left( \mathbb{H},\left\langle \cdot ,\cdot \right\rangle \right) $ be
a separable real Hilbert space with the induced norm $\left\vert \cdot
\right\vert .$

A function $x:\mathbb{R}^{+}\rightarrow \mathbb{H}$ is a c\`{a}dl\`{a}g
function if for every $t\in \mathbb{R}^{+}$ the left limit $%
x_{t-}:=\lim\limits_{s\nearrow t}x_{s}$ and the right limit $%
x_{t+}:=\lim\limits_{s\searrow t}x_{s}$ exist in $\mathbb{H}$ and $%
x_{t+}=x_{t}$ for all $t\geq 0;$ by convention $x_{0-}=x_{0}\,.$

We denote by $\mathbb{D}\left( \mathbb{R}^{+},\mathbb{H}\right) $ the set of
c\`{a}dl\`{a}g functions $x:\mathbb{R}^{+}\rightarrow \mathbb{H}$ and $%
\mathbb{D}\left( \left[ 0,T\right] ,\mathbb{H}\right) \subset \mathbb{D}%
\left( \mathbb{R}^{+},\mathbb{H}\right) $ the subspace of paths $x$ that
stop at the time $T,$ that is $x\in \mathbb{D}\left( \left[ 0,T\right] ,%
\mathbb{H}\right) $ if $x\in \mathbb{D}\left( \mathbb{R}^{+},\mathbb{H}%
\right) $ and $x_{t}=x_{t}^{T}:=x_{t\wedge T}\,,$ for all $t\geq 0.$ The
spaces of continuous functions will be denoted by $\mathcal{C}\left( \mathbb{%
R}^{+},\mathbb{H}\right) $ and $\mathcal{C}\left( \left[ 0,T\right] ,\mathbb{%
H}\right) $ respectively.

We say that $\pi =\left\{ t_{0},t_{1},t_{2},\ldots \right\} $ is a partition
of $\mathbb{R}^{+}$ if%
\begin{equation*}
0=t_{0}<t_{1}<t_{2}<\,\ldots \quad \text{and}\quad t_{n}\rightarrow +\infty .
\end{equation*}%
Let $\pi $ be a partition and $r\in \pi $. We denote by $r^{\prime }$ the
successor of $r$ in the partition $\pi $, i.e.%
\begin{equation}
\text{if }r=t_{i}\text{ then }r^{\prime }:=t_{i+1}.  \label{def r,r'}
\end{equation}%
Write%
\begin{align*}
\left\Vert \pi \right\Vert & :=\sup \left\{ r^{\prime }-r:r\in \pi \right\}
\\
\mathrm{mesh}\left( \pi \right) & :=\inf \left\{ r^{\prime }-r:r\in \pi
\right\} .
\end{align*}%
The set of all partitions of $\mathbb{R}^{+}$ will be denoted by $\mathcal{P}%
_{\mathbb{R}^{+}}\,.$

Given a function $x:\mathbb{R}^{+}\rightarrow\mathbb{H}$ we define

\begin{itemize}
\item the supremum norm by $\left\Vert x\right\Vert _{T}=\sup_{t\in \left[
0,T\right] }\left\vert x_{t}\right\vert \quad $and$\quad \left\Vert
x\right\Vert _{\infty }=\sup_{t\geq 0}\left\vert x_{t}\right\vert \,;$%
\medskip

\item the oscillation of $x$ on a set $F\subset \mathbb{R}^{+}$ by $\mathcal{%
O}_{x}\left( F\right) =\omega _{x}\left( F\right) =\sup_{t,s\in F}\left\vert
x_{t}-x_{s}\right\vert \,;$\medskip

\item the modulus of continuity $\boldsymbol{\upmu}_{x}:\mathbb{R}%
^{+}\rightarrow \mathbb{R}^{+}$ by $\boldsymbol{\upmu}_{x}\left( \varepsilon
\right) =\boldsymbol{\upmu}\left( \varepsilon ;x\right) =\sup_{t\in \mathbb{R%
}^{+}}\mathcal{O}_{x}\left( \left[ t,t+\varepsilon \right] \right) \,;$%
\medskip

\item the c\`{a}dl\`{a}g modulus by \textbf{$\boldsymbol{\gamma }$}$%
_{x}\left( \varepsilon \right) =\omega _{x}^{\prime }\left( \varepsilon
\right) :=\inf_{\pi }\left\{ \max_{r\in \pi }\mathcal{O}_{x}\left(
[r,r^{\prime })\right) :\mathrm{mesh}\left( \pi \right) >\varepsilon
\right\} $, where $\pi \in \mathcal{P}_{\mathbb{R}^{+}}$ is a partition an
\textbf{$\boldsymbol{\gamma }$}$_{x}\left( \varepsilon ,T\right) =\omega
_{x}^{\prime }\left( \varepsilon ,T\right) :=\mathbf{\gamma }_{x^{T}}\left(
\varepsilon \right) .$
\end{itemize}

\begin{remark}
\label{remark Annex 6'}We mention that

\noindent $1.$ function $x:\mathbb{R}^{+}\rightarrow \mathbb{H}$ is
continuous on $\left[ 0,T\right] $ if and only if $\lim\limits_{\varepsilon
\rightarrow 0}\boldsymbol{\upmu}(\varepsilon ;x^{T})=0\,;$

\noindent $2.$ function $x$ is c\`{a}dl\`{a}g on $\left[ 0,T\right] $ if and
only if $\lim_{\varepsilon \rightarrow 0}$\textbf{$\boldsymbol{\gamma }$}$%
_{x}\left( \varepsilon ,T\right) =0\,;$

\noindent $3.$ \textbf{$\boldsymbol{\gamma }$}$_{x}\left( \varepsilon
\right) \leq \boldsymbol{\upmu}_{x}\left( 2\varepsilon \right) \,;$

\noindent $4.$ if $x\in \mathbb{D}\left( \left[ 0,T\right] ,\mathbb{H}%
\right) ,$ then for each $\delta >0$ there exists a partition $\pi \in
\mathcal{P}_{\mathbb{R}^{+}}$ such that $\max_{r\in \pi }\mathcal{O}%
_{x}\left( [r,r^{\prime })\right) <\delta $ and consequently:

$\left( a\right) $ there exists a sequence of partitions $\pi _{n}\in
\mathcal{P}_{\mathbb{R}^{+}}$ such that $\displaystyle\max_{r\in \pi _{n}}%
\mathcal{O}_{x}\left( [r,r^{\prime })\right) <\frac{1}{n}\,.$ Therefore $x$
can be uniformly approximate by simple functions, constant on intervals:%
\begin{equation*}
x_{t}^{n}=\sum_{r\in \pi _{n}}x_{r}\mathbf{1}_{[r,r^{\prime })}\left(
t\right) ,\quad t\geq 0
\end{equation*}%
and $\left\Vert x^{n}-x\right\Vert _{T}\rightarrow 0$ as $\left\Vert \pi
_{n}\right\Vert \rightarrow 0\,;$

$\left( b\right) $ for each $\delta >0$ there exist a finite number of
points $t\in \left[ 0,T\right] $ such that%
\begin{equation*}
\left\vert x_{t}-x_{t-}\right\vert \geq \delta \,;
\end{equation*}

$\left( c\right) $ $\left\Vert x\right\Vert _{T}<\infty $ and the closure of
$\left\{ x_{t}:t\in \left[ 0,T\right] \right\} $ is compact.
\end{remark}

Let $\Lambda $ the collection of the strictly increasing functions $\lambda :%
\left[ 0,T\right] \rightarrow \left[ 0,T\right] $ such that $\lambda \left(
0\right) =0$ and $\lambda \left( T\right) =T$ (the space of time scale
transformations). The Skorohod topology on $\mathbb{D}\left( \left[ 0,T%
\right] ,\mathbb{H}\right) $ it the topology defined by one of the two
topologically equivalent metrics: for all $x,y\in \mathbb{D}\left( \left[ 0,T%
\right] ,\mathbb{H}\right) ,$%
\begin{align*}
d^{0}\left( x,y\right) & =\inf_{\lambda \in \Lambda }\left( \left\Vert
x-y\circ \lambda \right\Vert _{T}\vee \left\Vert \lambda -I\right\Vert
_{T}\right) ,\quad \text{and} \\
d^{1}\left( x,y\right) & =\inf_{\lambda \in \Lambda }\Big(\left\Vert
x-y\circ \lambda \right\Vert _{T}\vee \sup_{s<t}\Big|\ln \Big(\frac{\lambda
\left( t\right) -\lambda \left( s\right) }{t-s}\Big)\Big|\Big),
\end{align*}%
where $a\vee b:=\max \left\{ a,b\right\} $ and $I:\mathbb{R}^{+}\rightarrow
\mathbb{R}^{+}$ is the identity function.

\begin{remark}
\label{remark Annex 6''}We have:

\noindent $1.\;d^{0}\left( x,y\right) \leq \left\Vert x-y\right\Vert _{T}\,,$
for all $x,y\in \mathbb{D}\left( \left[ 0,T\right] ,\mathbb{H}\right) \,;$

\noindent $2.\;\left( \mathbb{D}\left( \left[ 0,T\right] ,\mathbb{H}\right)
,d^{0}\right) $ is a separable metric space, but not complete;

\noindent $3.\;\left( \mathbb{D}\left( \left[ 0,T\right] ,\mathbb{H}\right)
,d^{1}\right) $ is separable complete metric space (Polish space);

\noindent $4.\;x^{n}\rightarrow x$ in $\mathbb{D}\left( \left[ 0,T\right] ,%
\mathbb{H}\right) $ if and only if there exists $\lambda _{n}\in \Lambda $
such that%
\begin{equation*}
\left\Vert \lambda _{n}-I\right\Vert _{T}+\left\Vert x^{n}\circ \lambda
_{n}-x\right\Vert _{T}\rightarrow 0,\text{ as }n\rightarrow \infty \,;
\end{equation*}

\noindent $5.\;$if $x^{n}\rightarrow x$ in $\mathbb{D}\left( \left[ 0,T%
\right] ,\mathbb{H}\right) $, then $x_{n}(t)\rightarrow x\left( t\right) $
holds for continuity points $t$ of $x;$ moreover if $x$ is continuous, then $%
\left\Vert x^{n}-x\right\Vert _{T}\rightarrow 0\,;$

\noindent $6.\;$if $x^{n}\rightarrow x$ in $\mathbb{D}\left( \left[ 0,T%
\right] ,\mathbb{H}\right) $, then%
\begin{equation*}
\sup_{n\in \mathbb{N}}\left\Vert x^{n}\right\Vert _{T}<\infty \quad \text{and%
}\quad \lim_{\varepsilon \rightarrow 0}\sup_{n\in \mathbb{N}}\mathbf{%
\boldsymbol{\gamma }}_{x^{n}}\left( \varepsilon ,T\right) <\infty \,;
\end{equation*}

\noindent $7.\;\mathbb{D}\left( \mathbb{R}^{+},\mathbb{H}\right) $ is a
Polish space under the metric%
\begin{equation*}
d\left( x,y\right) =\sum_{n\in \mathbb{N}}\frac{1}{2^{n}}\Big(1\wedge
d^{1}\left( x^{n},y^{n}\right) \Big),\quad x,y\in \mathbb{D}\left( \mathbb{R}%
^{+},\mathbb{H}\right) ~.
\end{equation*}%
The topology $\tau $ generated by $d$ on $\mathbb{D}\left( \mathbb{R}^{+},%
\mathbb{H}\right) $ is called the Skorohod topology (and is denoted by $%
J_{1} $) and coincides on $\mathcal{C}\left( \mathbb{R}^{+},\mathbb{H}%
\right) $ with the topology of the uniform convergence on bounded intervals.
Also the Skorohod topology coincides on $\mathbb{D}\left( \left[ 0,T\right] ,%
\mathbb{H}\right) \subset \mathbb{D}\left( \mathbb{R}^{+},\mathbb{H}\right) $
with the topology generated of $d^{0}$ or $d^{1}.$
\end{remark}

Using \cite[Chapter 3, Proposition 6.5]{et-ku/86} it is easy to see that:

\begin{proposition}
\label{proposition Annex 6}Let $m\in\mathbb{D}\left( \mathbb{R}^{+},\mathbb{H%
}\right) $ and $\pi_{n}\in\mathcal{P}_{\mathbb{R}^{+}}\,$, $n\in\mathbb{N}$,
be a partition such that $\left\Vert \pi_{n}\right\Vert \rightarrow0$. If%
\begin{equation*}
m_{t}^{n}:=\sum_{r\in\pi_{n}}m_{r}\mathbf{1}_{[r,r^{\prime})}\left( t\right)
,
\end{equation*}
then%
\begin{equation*}
m^{n}\longrightarrow m\text{ in }\mathbb{D}\left( \mathbb{R}^{+},\mathbb{H}%
\right) .
\end{equation*}
\end{proposition}

\subsection{Bounded variation functions}

Let $\left[ a,b\right] $ be a closed interval from $\mathbb{R}$ and $%
\mathcal{P}_{\left[ a,b\right] }$ be the set of the partitions%
\begin{equation*}
\pi =\left\{ a=t_{0}<t_{1}<\cdots <t_{n}=b\right\} ,~n\in \mathbb{N}^{\ast }.
\end{equation*}%
We denote by $||\pi ||=\sup \left\{ t_{i+1}-t_{i}:0\leq i\leq n-1\right\} .$

We define the variation of a function $k:\left[ a,b\right] \rightarrow
\mathbb{H}$ corresponding to the partition $\pi \in \mathcal{P}_{\left[ a,b%
\right] }$ by%
\begin{equation*}
V_{\pi }\left( k\right) :=\sum_{i=0}^{n_{\pi }-1}\left\vert
k_{t_{i+1}}-k_{t_{i}}\right\vert
\end{equation*}%
and the total variation of $k$ on $\left[ a,b\right] $ by%
\begin{equation*}
\left\updownarrow k\right\updownarrow _{\left[ a,b\right] }=\sup_{\pi \in
\mathcal{P}_{\left[ a,b\right] }}V_{\pi }\left( k\right) =\sup \left\{
\sum_{i=0}^{n_{\pi }-1}\left\vert k_{t_{i+1}}-k_{t_{i}}\right\vert :\pi \in
\mathcal{P}_{\left[ a,b\right] }\right\} .
\end{equation*}%
If $\left[ a,b\right] =\left[ 0,T\right] ,$ then $\left\updownarrow
k\right\updownarrow _{T}:=\left\updownarrow k\right\updownarrow _{\left[ 0,T%
\right] }\,.$

\begin{proposition}
If $k\in \mathbb{D}\left( \left[ 0,T\right] ,\mathbb{H}\right) $ and $%
\overline{\pi }_{N}\in \mathcal{P}_{\left[ 0,T\right] }$ is given by%
\begin{equation*}
\overline{\pi }_{N}:\left\{ 0=\frac{0}{2^{N}}T<\frac{1}{2^{N}}T<\cdots <%
\frac{2^{N}-1}{2^{N}}T<\frac{2^{N}}{2^{N}}T=T\right\} ,
\end{equation*}%
then%
\begin{equation*}
V_{\overline{\pi }_{N}}\left( k\right) \nearrow \left\updownarrow
k\right\updownarrow _{T}\,,\quad \text{as }N\rightarrow \infty .
\end{equation*}
\end{proposition}

\begin{proof}
Clearly $V_{\overline{\pi }_{N}}\left( k\right) $ is increasing with respect
to $N$ and $V_{\overline{\pi }_{N}}\left( k\right) \leq \left\updownarrow
k\right\updownarrow _{T}\,.$ Let $\pi \in \mathcal{P}_{\left[ 0,T\right] }$
be arbitrary chosen, $\pi =\left\{ 0=t_{0}<t_{1}<\cdots <t_{n_{\pi
}}=T\right\} $ and let $j_{i,N}$ be the integer smallest integer grater or
equal to $\frac{t_{i}}{T}2^{N}.$ Then $j_{i,N}\frac{T}{2^{N}}\geq t_{i}$ and
$\lim_{N\rightarrow \infty }\left( j_{i,N}\frac{T}{2^{N}}\right) =t_{i}.$ We
have%
\begin{equation*}
\begin{array}{l}
\displaystyle V_{\pi }\left( k\right) =\sum\limits_{i=0}^{n_{\pi
}-1}\left\vert k_{t_{i+1}}-k_{t_{i}}\right\vert \medskip \\
\displaystyle\leq \sum\limits_{i=0}^{n_{\pi }-1}\left[ \Big|%
k_{t_{i+1}}-k_{j_{i+1,N}\frac{T}{2^{N}}}\Big|+\Big|k_{j_{i+1,N}\frac{T}{2^{N}%
}}-k_{j_{i,N}\frac{T}{2^{N}}}\Big|+\Big|k_{j_{i,N}\frac{T}{2^{N}}}-k_{t_{i}}%
\Big|\right] \medskip \\
\displaystyle\leq 2\sum\limits_{i=0}^{n_{\pi }}\Big|k_{j_{i,N}\frac{T}{2^{N}}%
}-k_{t_{i}}\Big|+V_{\overline{\pi }_{N}}\left( k\right)%
\end{array}%
\end{equation*}%
and passing to the limit for $N\rightarrow \infty $ we obtain%
\begin{equation*}
V_{\pi }\left( k\right) \leq \lim_{N\rightarrow \infty }V_{\overline{\pi }%
_{N}}\left( k\right) \leq \left\updownarrow k\right\updownarrow _{T},\quad
\forall \pi \in \mathcal{P}_{\left[ a,b\right] }\,.
\end{equation*}%
Hence $\lim\limits_{N\rightarrow \infty }V_{\overline{\pi }_{N}}\left(
k\right) =\left\updownarrow k\right\updownarrow _{T}\,.$\hfill
\end{proof}

\begin{definition}
A function $k:\left[ a,b\right] \rightarrow \mathbb{H}$ has bounded
variation on $\left[ a,b\right] $ if $\left\updownarrow k\right\updownarrow
_{\left[ a,b\right] }<\infty .$ The space of bounded variation functions on $%
\left[ a,b\right] $ will be denoted by $\mathrm{BV}\left( \left[ a,b\right] ;%
\mathbb{H}\right) .$ By $\mathrm{BV}_{loc}\left( \mathbb{R}^{+};\mathbb{H}%
\right) $ we denote the space of the functions $k:\mathbb{R}^{+}\rightarrow
\mathbb{H}$ such that $\left\updownarrow k\right\updownarrow _{T}<\infty $
for all $T>0.$
\end{definition}

Let $k\in \mathbb{D}\left( \mathbb{R}^{+},\mathbb{H}\right) \cap \mathrm{BV}%
_{loc}\left( \mathbb{R}^{+};\mathbb{H}\right) .$ The function $k$ has the
decomposition $k_{t}=k_{t}^{c}+k_{t}^{d}$ where%
\begin{equation*}
k_{t}^{d}=\sum_{0\leq s\leq t}\Delta k_{s},\quad \text{with }\Delta
k_{s}:=k_{s}-k_{s-}
\end{equation*}%
is a pure jump function and $k_{t}^{c}:=k_{t}-k_{t}^{d}$ is a continuous
function. The series which define $k_{t}^{d}$ is convergent since $%
\sum_{0\leq s\leq t}\left\vert \Delta k_{s}\right\vert \leq
\left\updownarrow k\right\updownarrow _{t}<\infty $.

If $x\in \mathbb{D}\left( \mathbb{R}^{+},\mathbb{H}\right) $ and $k\in
\mathbb{D}\left( \mathbb{R}^{+},\mathbb{H}\right) \cap \mathrm{BV}%
_{loc}\left( \mathbb{R}^{+};\mathbb{H}\right) ,$ we define%
\begin{equation*}
\left[ x,k\right] _{t}:=\sum_{0\leq s\leq t}\left\langle \Delta x_{s},\Delta
k_{s}\right\rangle
\end{equation*}%
Remark that the above series is well defined since%
\begin{equation*}
\big|\left[ x,k\right] _{t}\big|=\Big|\sum_{0\leq s\leq t}\left\langle
\Delta x_{s},\Delta k_{s}\right\rangle \Big|\leq \sum_{0\leq s\leq
t}\left\vert \Delta x_{s}\right\vert \left\vert \Delta k_{s}\right\vert \leq
2\big(\sup_{0\leq s\leq t}\left\vert x_{s}\right\vert \big)\left\updownarrow
k\right\updownarrow _{t}\,.
\end{equation*}%
We recall now some results due to \cite{ge/38}. If $k\in \mathbb{D}\left(
\mathbb{R}^{+},\mathbb{H}\right) \cap \mathrm{BV}_{loc}\left( \mathbb{R}^{+};%
\mathbb{H}\right) $ then there exists a unique $\mathbb{H}$-valued, $\sigma $%
-finite measure $\mu _{k}:\mathcal{B}_{\mathbb{R}^{+}}\rightarrow \mathbb{H}$
such that%
\begin{equation*}
\mu _{k}\left( (s,t]\right) =k_{t}-k_{s},\quad \text{for all }0\leq s<t.
\end{equation*}%
and the total variation measure is uniquely defined by%
\begin{equation*}
\left\vert \mu _{k}\right\vert \left( (s,t]\right) =\left\updownarrow
k\right\updownarrow _{t}-\left\updownarrow k\right\updownarrow _{s}.
\end{equation*}%
The Lebesgue-Stieltjes integral on $(s,t]$%
\begin{align*}
\int_{s}^{t}\left\langle x_{r},dk_{r}\right\rangle &
:=\int_{(s,t]}\left\langle x_{r},\mu _{k}\left( dr\right) \right\rangle
=\int_{(s,t]}\left\langle x_{r},\mu _{k^{c}}\left( dr\right) \right\rangle
+\int_{(s,t]}\left\langle x_{r},\mu _{k^{d}}\left( dr\right) \right\rangle \\
& =\int_{(s,t]}\left\langle x_{r},\mu _{k^{c}}\left( dr\right) \right\rangle
+\sum_{s<r\leq t}\left\langle x_{r},\Delta k_{r}\right\rangle
\end{align*}%
is defined for all Borel measurable function such that $\displaystyle%
\int_{s}^{t}\left\vert x_{r}\right\vert d\left\updownarrow
k\right\updownarrow _{r}:=\int_{(s,t]}\left\vert x_{r}\right\vert \left\vert
\mu _{k}\right\vert \left( dr\right) <\infty $, and in this case%
\begin{equation*}
\left\vert \int_{s}^{t}\left\langle x_{r},dk_{r}\right\rangle \right\vert
\leq \int_{s}^{t}\left\vert x_{r}\right\vert d\left\updownarrow
k\right\updownarrow _{r}\leq \big(\sup_{s<r\leq t}\left\vert
x_{r}\right\vert \big)\left( \left\updownarrow k\right\updownarrow
_{t}-\left\updownarrow k\right\updownarrow _{s}\right) .
\end{equation*}%
We add that the Lebesgue-Stieltjes integral on $\left\{ t\right\} $%
\begin{equation*}
\int_{\left\{ t\right\} }\left\langle x_{r},dk_{r}\right\rangle
:=\left\langle x_{t},\Delta k_{t}\right\rangle .
\end{equation*}

Let $\pi _{n}\in \mathcal{P}_{\mathbb{R}^{+}}$ be a sequence of partitions
such that $||\pi _{n}||\rightarrow 0$ as $n\rightarrow \infty .$ Let us
denote%
\begin{equation*}
\left\lfloor r\right\rfloor _{n}=\left\lfloor r\right\rfloor _{\pi
_{n}}:=\max \left\{ s\in \pi _{n}:s<r\right\} \quad \text{and}\quad
\left\lceil r\right\rceil _{n}=\left\lceil r\right\rceil _{\pi _{n}}:=\min
\left\{ s\in \pi _{n}:s\geq r\right\} .
\end{equation*}%
Then for all $x\in \mathbb{D}\left( \mathbb{R}^{+},\mathbb{H}\right) $ and $%
r\geq 0$ we have $x_{\left\lfloor r\right\rfloor _{n}}\rightarrow x_{r-}$
and $x_{\left\lceil r\right\rceil _{n}}\rightarrow x_{r}\,.$ By the Lebesgue
dominated convergence theorem as $n\rightarrow \infty $%
\begin{align*}
\sum_{s\in \pi _{n}}\left\langle x_{s\wedge t},k_{s^{\prime }\wedge
t}-k_{s\wedge t}\right\rangle & =\int_{0}^{t}\langle x_{\left\lfloor
r\right\rfloor _{n}},dk_{r}\rangle \;\rightarrow \;\int_{0}^{t}\left\langle
x_{r-},dk_{r}\right\rangle , \\
\sum_{s\in \pi _{n}\,}\left\langle x_{s^{\prime }\wedge t},k_{s^{\prime
}\wedge t}-k_{s\wedge t}\right\rangle & =\int_{0}^{t}\langle x_{\left\lceil
r\right\rceil _{n}\wedge t},dk_{r}\rangle \;\rightarrow
\;\int_{0}^{t}\left\langle x_{r},dk_{r}\right\rangle ,
\end{align*}%
and%
\begin{align}
\int_{0}^{t}\left\langle x_{r},dk_{r}\right\rangle &
=\int_{0}^{t}\left\langle x_{r-},dk_{r}\right\rangle +\left[ x,k\right] _{t}
\label{connection betw int} \\
& =\lim_{||\pi _{n}||\rightarrow 0}\sum_{s\in \pi _{n}\,}\left\langle
x_{s\wedge t},k_{s^{\prime }\wedge t}-k_{s\wedge t}\right\rangle +\left[ x,k%
\right] _{t}~.  \notag
\end{align}%
Remark that, if $k\in \mathbb{D}\left( \mathbb{R}^{+},\mathbb{H}\right) \cap
\mathrm{BV}_{loc}\left( \mathbb{R}^{+};\mathbb{H}\right) $, then%
\begin{align*}
\left\vert k_{t}\right\vert ^{2}-\left\vert k_{0}\right\vert ^{2}&
=\sum_{s\in \pi _{n}\,}\left\langle k_{s^{\prime }\wedge t},k_{s^{\prime
}\wedge t}-k_{s\wedge t}\right\rangle +\sum_{s\in \pi _{n}\,}\left\langle
k_{s\wedge t},k_{s^{\prime }\wedge t}-k_{s\wedge t}\right\rangle \\
& \rightarrow \int_{0}^{t}\left\langle x_{r},dk_{r}\right\rangle
+\int_{0}^{t}\left\langle x_{r-},dk_{r}\right\rangle
=2\int_{0}^{t}\left\langle k_{r},dk_{r}\right\rangle -\left[ k,k\right] _{t}
\end{align*}%
Therefore we proved

\begin{lemma}
\label{lemma Annex 9}If $k\in\mathbb{D}\left( \mathbb{R}^{+},\mathbb{H}%
\right) \cap\mathrm{BV}_{loc}\left( \mathbb{R}^{+};\mathbb{H}\right) $ then%
\begin{equation*}
\int_{0}^{t}\left\langle k_{r},dk_{r}\right\rangle =\frac{1}{2}\left\vert
k_{t}\right\vert ^{2}-\frac{1}{2}\left\vert k_{0}\right\vert ^{2}+\frac{1}{2}%
\left[ k,k\right] _{t}~.
\end{equation*}
\end{lemma}

\begin{lemma}
\label{lemma Annex 10}Let $k:\mathbb{R}^{+}\rightarrow \mathbb{H}$ and $%
k^{n}\in \mathbb{D}\left( \mathbb{R}^{+},\mathbb{H}\right) \cap \mathrm{BV}%
_{loc}\left( \mathbb{R}^{+};\mathbb{H}\right) $. If for all $T\geq 0,$%
\begin{equation*}
\left\Vert k_{t}^{n}-k_{t}\right\Vert _{T}\rightarrow 0\quad \text{and}\quad
\sup\limits_{n\in \mathbb{N}^{\ast }}\left\updownarrow
k^{n}\right\updownarrow _{T}=M<+\infty ,
\end{equation*}%
then $k\in \mathbb{D}\left( \mathbb{R}^{+},\mathbb{H}\right) \cap \mathrm{BV}%
\left( \left[ 0,T\right] ;\mathbb{D}\left( \mathbb{R}^{+},\mathbb{H}\right)
\right) $ and $\left\updownarrow k\right\updownarrow _{T}\leq M.$
\end{lemma}

\begin{proof}
Let $0<\varepsilon \leq 1.$ Then%
\begin{equation*}
\left\vert k_{t+\varepsilon }-k_{t}\right\vert \leq 2\left\Vert
k-k^{n}\right\Vert _{t+1}+\left\vert k_{t+\varepsilon
}^{n}-k_{t}^{n}\right\vert
\end{equation*}%
and therefore%
\begin{equation*}
\limsup_{\varepsilon \searrow 0}\left\vert k_{t+\varepsilon
}-k_{t}\right\vert \leq 2\left\Vert k-k^{n}\right\Vert _{t+1},\quad \text{%
for all }n\in \mathbb{N}^{\ast }.
\end{equation*}%
Hence $k_{t+}=k_{t}\,.$ Let a sequence $\pi _{N}\in \mathcal{P}_{\left[ 0,T%
\right] }$ such that $V_{\pi _{N}}\left( k\right) \nearrow \left\updownarrow
k\right\updownarrow _{T}\quad $as $N\rightarrow \infty .$ From the
definition of $\left\updownarrow \cdot \right\updownarrow _{T}$ we have $%
V_{\pi _{N}}\left( k^{n}\right) \leq \left\updownarrow
k^{n}\right\updownarrow _{T}\leq M.$ Since $k_{t}^{n}\rightarrow k_{t}$ for
all $t\in \left[ 0,T\right] $, $V_{\pi _{N}}\left( k^{n}\right) \rightarrow
V_{\pi _{N}}\left( k\right) .$ Hence $V_{\pi _{N}}\left( k\right) \leq M$,
for all $N\in \mathbb{N}^{\ast }$ and passing to the limit as $N\rightarrow
\infty $ we obtain $\left\updownarrow k\right\updownarrow _{T}\leq M.$\hfill
\end{proof}

\begin{theorem}[Helly-Bray]
\label{Helly-Bray} Let $n\in \mathbb{N}^{\ast }$, $x^{n},x,k\in \mathbb{D}%
\left( \mathbb{R}^{+},\mathbb{H}\right) $ and $k^{n}\in \mathbb{D}\left(
\mathbb{R}^{+},\mathbb{H}\right) \cap \mathrm{BV}_{loc}\left( \mathbb{R}^{+};%
\mathbb{H}\right) $, such that for all $T$%
\begin{equation*}
\begin{array}{rl}
\left( i\right) & \left\Vert x^{n}-x\right\Vert _{T}\rightarrow 0,\text{ as }%
n\rightarrow \infty ,\medskip \\
\left( ii\right) & \left\Vert k^{n}-k\right\Vert _{T}\rightarrow 0,\text{ as
}n\rightarrow \infty ,\quad \text{and}\medskip \\
\left( iii\right) & \sup\limits_{n\in \mathbb{N}^{\ast }}\left\updownarrow
k^{n}\right\updownarrow _{T}=M<+\infty .%
\end{array}%
\end{equation*}%
Then $k\in \mathbb{D}\left( \mathbb{R}^{+},\mathbb{H}\right) \cap \mathrm{BV}%
\left( \left[ 0,T\right] ;\mathbb{R}^{d}\right) ,$ $\left\updownarrow
k\right\updownarrow _{T}\leq M$ and uniformly with respect to $s,t\in \left[
0,T\right] ,$ $s\leq t,$%
\begin{equation}
\int_{s}^{t}\left\langle x_{r}^{n},dk_{r}^{n}\right\rangle \rightarrow
\int_{s}^{t}\left\langle x_{r},dk_{r}\right\rangle ,\;\text{as }n\rightarrow
\infty .  \label{HB1}
\end{equation}%
Moreover,%
\begin{equation}
\int_{s}^{t}\left\vert x_{r}\right\vert d\left\updownarrow
k\right\updownarrow _{r}\leq \liminf\limits_{n\rightarrow +\infty
}\int_{s}^{t}\left\vert x_{r}^{n}\right\vert d\left\updownarrow
k^{n}\right\updownarrow _{r}\,,\text{ for all }0\leq s\leq t\leq T
\label{HB2}
\end{equation}%
and there exist a a subsequence $n_{i}\rightarrow \infty $ and a sequence $%
\delta _{i}\searrow 0$ as $i\rightarrow \infty $ such that uniformly with
respect to $s,t\in \left[ 0,T\right] ,$ $s\leq t:$%
\begin{equation}
\int_{s}^{t}\mathbf{1}_{\left\vert \Delta k_{r}^{n_{i}}\right\vert >\delta
_{i}}\big\langle x_{r}^{n_{i}},dk_{r}^{d,n_{i}}\big\rangle\rightarrow
\int_{s}^{t}\big\langle x_{r},dk_{r}^{d}\big\rangle,\;\text{as }i\rightarrow
\infty .  \label{HB3}
\end{equation}
\end{theorem}

\begin{proof}
From Lemma \ref{lemma Annex 10} we deduce that $k\in \mathbb{D}\left(
\mathbb{R}^{+},\mathbb{H}\right) \cap \mathrm{BV}\left( \left[ 0,T\right] ;%
\mathbb{R}^{d}\right) $ and $\left\updownarrow k\right\updownarrow _{T}\leq
M $.$\smallskip $

\noindent \textbf{Step 1.} Let $\varepsilon >0$ and $\pi _{\varepsilon }\in
\mathcal{P}_{\mathbb{R}^{+}}$ be such that $\max_{r\in \pi _{\varepsilon }}%
\mathcal{O}_{x}\left( [r,r^{\prime })\right) <\varepsilon .$ We denote by $%
x_{u}^{\varepsilon }=\sum_{r\in \pi _{\varepsilon }}x_{r}\mathbf{1}%
_{[r,r^{\prime })}\left( u\right) ,$ for any $u\geq 0.$ We have $\left\vert
x_{u}-x_{u}^{\varepsilon }\right\vert <\varepsilon $ for all $u\geq 0\ $ and
consequently%
\begin{equation*}
\begin{array}{l}
\displaystyle\Big|\int_{s}^{t}\left\langle x_{r}^{n},dk_{r}^{n}\right\rangle
-\int_{s}^{t}\left\langle x_{r},dk_{r}\right\rangle \Big|\medskip \\
\displaystyle\leq \Big|\int_{s}^{t}\left\langle
x_{r}^{n}-x_{r},dk_{r}^{n}\right\rangle \Big|+\Big|\int_{s}^{t}\left\langle
x_{r}-x_{r}^{\varepsilon },dk_{r}^{n}-dk_{r}\right\rangle \Big|+\Big|%
\int_{0}^{t}\left\langle x_{r}^{\varepsilon },dk_{r}^{n}-dk_{r}\right\rangle %
\Big|\medskip \\
\displaystyle\leq \left\Vert x^{n}-x\right\Vert _{T}~\left\updownarrow
k^{n}\right\updownarrow _{T}+2M\varepsilon +2\left\Vert x\right\Vert
_{T}\left\Vert k^{n}-k\right\Vert _{T}\,\mathrm{card}\left\{ r\in \pi
_{\varepsilon }:r\leq T\right\} ,%
\end{array}%
\end{equation*}%
because%
\begin{equation*}
\int_{0}^{t}\left\langle x_{r}^{\varepsilon },dk_{r}\right\rangle
=\sum_{r\in \pi _{\varepsilon }}\left\langle x_{r\wedge t},k_{r^{\prime
}\wedge t}-k_{r\wedge t}\right\rangle .
\end{equation*}%
Hence
\begin{equation*}
\limsup_{n\rightarrow \infty }\left[ \sup_{0\leq s\leq t\leq T}\Big|%
\int_{s}^{t}\left\langle x_{r}^{n},dk_{r}^{n}\right\rangle
-\int_{s}^{t}\left\langle x_{r},dk_{r}\right\rangle \Big|\right] \leq
2M\varepsilon ,\quad \forall \varepsilon >0.
\end{equation*}%
which yields (\ref{HB1}).$\smallskip $

\noindent \textbf{Step 2. }Let $\alpha \in C\left( \left[ 0,T\right] ;%
\mathbb{R}^{d}\right) $ such that $\left\Vert \alpha \right\Vert _{T}\leq 1.$
Then%
\begin{equation*}
\int_{s}^{t}\left\vert x_{r}\right\vert \left\langle \alpha
_{r},dk_{r}\right\rangle =\lim_{n\rightarrow \infty }\int_{s}^{t}\left\vert
x_{r}^{n}\right\vert \left\langle \alpha _{r},dk_{r}^{n}\right\rangle \leq
\liminf_{n\rightarrow +\infty }\int_{s}^{t}\left\vert x_{r}^{n}\right\vert
d\left\updownarrow k^{n}\right\updownarrow _{r}
\end{equation*}%
and passing to $\sup_{\left\Vert \alpha \right\Vert _{T}\leq 1}$ we get (\ref%
{HB2}).$\smallskip $

\noindent \textbf{Step 3. }Recall that the set $\left\{ \left\vert \Delta
k_{r}\right\vert :r\in \left[ 0,T\right] \right\} $ is at most countable.
Let $\delta \notin \left\{ \left\vert \Delta k_{r}\right\vert :r\in \left[
0,T\right] \right\} .$ Since%
\begin{equation*}
\left\vert \int_{s}^{t}\mathbf{1}_{\left\vert \Delta k_{r}\right\vert
>\delta }\langle x_{r},dk_{r}^{d}\rangle -\int_{s}^{t}\langle
x_{r},dk_{r}^{d}\rangle \right\vert \leq \int_{0}^{T}\mathbf{1}_{\left\vert
\Delta k_{r}\right\vert \leq \delta }\left\vert x_{r}\right\vert
d\updownarrow {\hspace{-0.1cm}}k^{d}{\hspace{-0.1cm}}\updownarrow _{r}\,,
\end{equation*}%
by the Lebesgue dominated convergence theorem there exists a sequence $%
\delta _{i}\searrow 0,$ $\delta _{i}\notin \left\{ \left\vert \Delta
k_{r}\right\vert :r\in \left[ 0,T\right] \right\} $ such that%
\begin{equation*}
\sup_{0\leq s\leq t\leq T}\left\vert \int_{s}^{t}\mathbf{1}_{\left\vert
\Delta k_{r}\right\vert >\delta _{i}}\langle x_{r},dk_{r}^{d}\rangle
-\int_{s}^{t}\langle x_{r},dk_{r}^{d}\rangle \right\vert <\frac{1}{2i},\quad
\text{for all }i\in \mathbb{N}.
\end{equation*}%
Note that for each $i$ fixed the set $Q_{i}=\left\{ r\geq 0:\left\vert
\Delta k_{r}\right\vert >\delta _{i}\right\} $ is finite and consequently,
using the uniform convergence of $k^{n},$ there exists $\ell _{i}\in \mathbb{%
N}$ such that $\left\vert \Delta k_{r}^{n}\right\vert >\delta _{i}$ for all $%
n\geq \ell _{i}$ and for all $r\in Q_{i}$.

We have%
\begin{equation*}
\Big|\int_{s}^{t}\mathbf{1}_{\left\vert \Delta k_{r}^{n}\right\vert >\delta
_{i}}\langle x_{r}^{n},dk_{r}^{d,n}\rangle -\int_{s}^{t}\langle
x_{r},dk_{r}^{d}\rangle \Big|\leq D_{i,n}+E_{i,n}+F_{i}
\end{equation*}%
with%
\begin{align*}
D_{i,n}& =\Big|\int_{s}^{t}\left( \mathbf{1}_{\left\vert \Delta
k_{r}^{n}\right\vert >\delta _{i}}-\mathbf{1}_{\left\vert \Delta
k_{r}\right\vert >\delta _{i}}\right) \langle x_{r}^{n},dk_{r}^{d,n}\rangle %
\Big|, \\
E_{i,n}& =\Big|\int_{s}^{t}\mathbf{1}_{\left\vert \Delta k_{r}\right\vert
>\delta _{i}}\langle x_{r}^{n},dk_{r}^{d,n}\rangle -\int_{s}^{t}\mathbf{1}%
_{\left\vert \Delta k_{r}\right\vert >\delta _{i}}\langle
x_{r},dk_{r}^{d}\rangle \Big|,\quad \text{and} \\
F_{i}& =\Big|\int_{s}^{t}\mathbf{1}_{\left\vert \Delta k_{r}\right\vert
>\delta _{i}}\langle x_{r},dk_{r}^{d}\rangle -\int_{s}^{t}\langle
x_{r},dk_{r}^{d}\rangle \Big|\leq \frac{1}{2i}\,.
\end{align*}%
Since%
\begin{align*}
\left\vert \mathbf{1}_{\left\vert \Delta k_{r}^{n}\right\vert >\delta _{i}}-%
\mathbf{1}_{\left\vert \Delta k_{r}\right\vert >\delta _{i}}\right\vert &
\leq \mathbf{1}_{\left\vert \Delta k_{r}^{n}-\Delta k_{r}\right\vert >\delta
_{i}-\delta _{i+1}}+\mathbf{1}_{\left\vert \Delta k_{r}\right\vert >\delta
_{i+1}}\left\vert \mathbf{1}_{\left\vert \Delta k_{r}^{n}\right\vert >\delta
_{i}}-\mathbf{1}_{\left\vert \Delta k_{r}\right\vert >\delta _{i}}\right\vert
\\
& \leq \mathbf{1}_{2\left\Vert k^{n}-k\right\Vert _{T}\geq \delta
_{i}-\delta _{i+1}}+\mathbf{1}_{\left\vert \Delta k_{r}\right\vert >\delta
_{i+1}}\left\vert \mathbf{1}_{\left\vert \Delta k_{r}^{n}\right\vert >\delta
_{i}}-\mathbf{1}_{\left\vert \Delta k_{r}\right\vert >\delta
_{i}}\right\vert ,
\end{align*}%
we deduce%
\begin{align*}
D_{i,n}& \leq \left\Vert x^{n}\right\Vert _{T}\left\updownarrow
k^{n}\right\updownarrow _{T}\mathbf{1}_{2\left\Vert k^{n}-k\right\Vert
_{T}\geq \delta _{i}-\delta _{i+1}}+\sum_{r\in Q_{i+1}}\left\vert \mathbf{1}%
_{\left\vert \Delta k_{r}^{n}\right\vert >\delta _{i}}-\mathbf{1}%
_{\left\vert \Delta k_{r}\right\vert >\delta _{i}}\right\vert \left\langle
x_{r}^{n},\Delta k_{r}^{n}\right\rangle \\
& \leq \Big[\mathbf{1}_{2\left\Vert k^{n}-k\right\Vert _{T}\geq \delta
_{i}-\delta _{i+1}}+\sum_{r\in Q_{i+1}}\left\vert \mathbf{1}_{\left\vert
\Delta k_{r}^{n}\right\vert >\delta _{i}}-\mathbf{1}_{\left\vert \Delta
k_{r}\right\vert >\delta _{i}}\right\vert \Big]\left\Vert x^{n}\right\Vert
_{T}\left\updownarrow k^{n}\right\updownarrow _{T}
\end{align*}%
and%
\begin{align*}
E_{i,n}& \leq \sum_{r\in Q_{i}}\left\vert \left\langle x_{r}^{n},\Delta
k_{r}^{n}\right\rangle -\left\langle x_{r},\Delta k_{r}\right\rangle
\right\vert \\
& \leq \left\Vert x^{n}-x\right\Vert _{T}\left\updownarrow
k^{n}\right\updownarrow _{T}+2\left\Vert x\right\Vert _{T}\left\Vert
k^{n}-k\right\Vert _{T}\mathrm{card}\left( Q_{i}\right) .
\end{align*}%
It follows there exists $n_{i}\geq \ell _{i}$ such that
\begin{equation*}
D_{i,n}=0\text{ and }E_{i,n}<\frac{1}{2i}\;\;\text{for all }n\geq n_{i}~.
\end{equation*}%
Hence%
\begin{equation*}
\Big|\int_{s}^{t}\mathbf{1}_{\left\vert \Delta k_{r}^{n_{i}}\right\vert
>\delta _{i}}\langle x_{r}^{n_{i}},dk_{r}^{d,n_{i}}\rangle
-\int_{s}^{t}\langle x_{r},dk_{r}^{d}\rangle \Big|<\frac{1}{i},\quad \text{%
for all }i\in \mathbb{N}^{\ast }.
\end{equation*}%
\hfill \medskip
\end{proof}

We also presents other auxiliary results used throughout the paper.

\begin{proposition}
\label{prop Anex 13}Let $A:\mathbb{H}\rightrightarrows \mathbb{H}$ be a
maximal monotone operator and $\mathcal{A}:\mathbb{D}\left( \mathbb{R}^{+},%
\mathbb{H}\right) \rightrightarrows \mathbb{D}\left( \mathbb{R}^{+},\mathbb{H%
}\right) \cap \mathrm{BV}\left( \left[ 0,T\right] ;\mathbb{D}\left( \mathbb{R%
}^{+},\mathbb{H}\right) \right) $ be defined by: $\left( x,k\right) \in
\mathcal{A\quad }if\quad x\in \mathbb{D}\left( \mathbb{R}^{+},\mathbb{H}%
\right) ,\;\;k\in \mathbb{D}\left( \mathbb{R}^{+},\mathbb{H}\right) \cap
\mathrm{BV}\left( \left[ 0,T\right] ;\mathbb{D}\left( \mathbb{R}^{+},\mathbb{%
H}\right) \right) $ and%
\begin{equation}
\int_{s}^{t}\left\langle x_{r}-z,dk_{r}-\hat{z}dr\right\rangle \geq 0,\quad
\forall \,\left( z,\hat{z}\right) \in A,\;\forall \,0\leq s\leq t.
\label{prop Anex 13_1}
\end{equation}%
Then relation (\ref{prop Anex 13_1}) is equivalent to the following one: for
all $u,\hat{u}\in \mathbb{D}\left( \mathbb{R}^{+},\mathbb{H}\right) $ such
that $\left( u_{r},\hat{u}_{r}\right) \in A,$ for any $r\geq 0$, we have%
\begin{equation}
\int\nolimits_{s}^{t}\left\langle x_{r}-u_{r},dk_{r}-\hat{u}%
_{r}dr\right\rangle \geq 0,\quad \forall \,0\leq s\leq t.
\label{prop Anex 13_2}
\end{equation}%
In addition $\mathcal{A}$ is a monotone operator, i.e. for all $\left(
x,k\right) ,$ $\left( y,\ell \right) \in \mathcal{A},$%
\begin{equation*}
\int_{s}^{t}\left\langle x_{r}-y_{r},dk_{r}-d\ell _{r}\right\rangle \geq
0,\quad \forall \,0\leq s\leq t
\end{equation*}%
and $\mathcal{A}$ is maximal in the set of monotone operators.
\end{proposition}

\begin{proof}
In order to obtain the implication (\ref{prop Anex 13_1})$\,\Rightarrow \,$(%
\ref{prop Anex 13_2}) let $u,\hat{u}\in \mathbb{D}\left( \mathbb{R}^{+},%
\mathbb{H}\right) $ such that $\left( u_{r},\hat{u}_{r}\right) \in A,\,$for
any $r\geq 0.$ Let $\pi _{n}$ be the partition $0<\frac{1}{n}<\frac{2}{n}%
<\cdots $ and%
\begin{equation*}
\left\lceil r\right\rceil _{n}=\frac{\min \left\{ j\in \mathbb{N}:r\leq
j\right\} }{n}\,.
\end{equation*}%
Therefore%
\begin{equation*}
\int\nolimits_{s}^{t}\left\langle x_{r}-u_{r},dk_{r}-\hat{u}%
_{r}dr\right\rangle =\lim_{n\rightarrow \infty }\int\nolimits_{s}^{t}\langle
x_{r}-u_{\left\lceil r\right\rceil _{n}},dk_{r}-\hat{u}_{\left\lceil
r\right\rceil _{n}}dr\rangle \geq 0.
\end{equation*}

The implication (\ref{prop Anex 13_2})$\,\Rightarrow \,$(\ref{prop Anex 13_1}%
) is obtained by taking $u_{r}=z$ and $\hat{u}_{r}=\hat{z}$.$\smallskip $

Let now $\left( x,k\right) ,(y,\ell )\in \mathcal{A}$ be arbitrary but
fixed. Then for all $u,\hat{u}\in \mathbb{D}\left( \mathbb{R}^{+},\mathbb{H}%
\right) $ such that $\left( u_{r},\hat{u}_{r}\right) \in A,$ for all $r\geq
0 $ we have for all $0\leq s\leq t,$
\begin{equation*}
\int\nolimits_{s}^{t}\left\langle y_{r}-u_{r},d\ell _{r}-\hat{u}%
_{r}dr\right\rangle \geq 0\quad \text{and}\quad
\int\nolimits_{s}^{t}\left\langle x_{r}-u_{r},dk_{r}-\hat{u}%
_{r}dr\right\rangle \geq 0.
\end{equation*}%
We take%
\begin{equation*}
u_{r}=J_{\varepsilon }\Big(\frac{x_{r}+y_{r}}{2}\Big)=\frac{x_{r}+y_{r}}{2}%
-\varepsilon A_{\varepsilon }\Big(\frac{x_{r}+y_{r}}{2}\Big)
\end{equation*}%
and $\hat{u}_{r}=A_{\varepsilon }\left( \frac{x_{r}+y_{r}}{2}\right) $.
Since $A$ is a maximal operator on $\mathbb{H}$, $\overline{D(A)}$ is convex
and $\lim\limits_{\varepsilon \rightarrow 0}\varepsilon A_{\varepsilon
}\left( u\right) \rightarrow 0,$ for all $u\in \overline{D(A)}$. Also for
all $a\in \mathrm{D}\left( A\right) $
\begin{equation*}
\varepsilon \left\vert A_{\varepsilon }\left( u\right) \right\vert \leq
\varepsilon \left\vert A_{\varepsilon }\left( u\right) -A_{\varepsilon
}\left( a\right) \right\vert +\varepsilon \left\vert A_{\varepsilon }\left(
a\right) \right\vert \leq \left\vert u-a\right\vert +\varepsilon \left\vert
A^{0}\left( a\right) \right\vert .
\end{equation*}%
Adding member by member the inequalities we obtain%
\begin{equation*}
0\leq \dfrac{1}{2}\int\nolimits_{s}^{t}\left\langle y_{r}-x_{r},d\ell
_{r}-dk_{r}\right\rangle +\varepsilon \int\nolimits_{s}^{t}\big\langle %
A_{\varepsilon }\Big(\frac{x_{r}+y_{r}}{2}\Big),d\ell _{r}+dk_{r}\big\rangle.
\end{equation*}%
Passing to $\lim_{\varepsilon \searrow 0}$ we obtain $\int\nolimits_{s}^{t}%
\left\langle y_{r}-x_{r},d\ell _{r}-dk_{r}\right\rangle \geq 0.\smallskip $

$\mathcal{A}$ is a maximal monotone operator since if $\left( y,\ell \right)
\in \mathbb{D}\left( \mathbb{R}^{+},\mathbb{H}\right) \times \left[ \mathbb{D%
}\left( \mathbb{R}^{+},\mathbb{H}\right) \cap \mathrm{BV}_{loc}\left(
\mathbb{R}^{+};\mathbb{H}\right) \right] $ satisfies
\begin{equation*}
\int\nolimits_{s}^{t}\left\langle y_{r}-x_{r},d\ell _{r}-dk_{r}\right\rangle
\geq 0,\quad \forall \,\left( x,k\right) \in \mathcal{A},
\end{equation*}%
then this last inequality is satisfied for all $\left( x,k\right) $ of the
form $\left( x_{t},k_{t}\right) =\left( z,\hat{z}\right) ,$ where $\left( z,%
\hat{z}\right) \in A,$ and consequently, from the definition of $\mathcal{A}%
, $ $\left( y,\ell \right) \in \mathcal{A}$ . The proof is complete
now.\hfill
\end{proof}

\begin{definition}
\label{definition Annex 14}We write $dk_{t}\in A\left( x_{t}\right) \left(
dt\right) $ if%
\begin{equation*}
\begin{array}{ll}
\left( a_{1}\right) & x\in \mathbb{D}\left( \mathbb{R}^{+},\mathbb{H}\right)
\text{ and }x_{t}\in \overline{\mathrm{D}(A)}\text{ for all }t\geq 0\medskip
\\
\left( a_{2}\right) & k\in \mathbb{D}\left( \mathbb{R}^{+},\mathbb{H}\right)
\cap \mathrm{BV}_{loc}\left( \mathbb{R}^{+};\mathbb{H}\right) ,\quad
k_{0}=0,\medskip \\
\left( a_{3}\right) & \left\langle x_{t}-u,\,dk_{t}-\hat{u}dt\right\rangle
\geq 0,\quad \text{on }\mathbb{R}^{+},\;\;\forall \,\left( u,\hat{u}\right)
\in A.%
\end{array}%
\end{equation*}
\end{definition}

\begin{proposition}
\label{prop Anex 15}Let $A\subset \mathbb{H}\times \mathbb{H}$ be a maximal
subset and $\mathcal{A}$ be the realization of $A$ on $\mathbb{D}\left(
\mathbb{R}^{+},\mathbb{H}\right) \times \left[ \mathbb{D}\left( \mathbb{R}%
^{+},\mathbb{H}\right) \cap \mathrm{BV}_{loc}\left( \mathbb{R}^{+};\mathbb{R}%
^{d}\right) \right] $ defined by (\ref{prop Anex 13_1}). Assume that $%
\mathrm{Int}\left( \mathrm{D}\left( A\right) \right) \neq \emptyset .$ Let $%
a\in \mathrm{Int}\left( \mathrm{D}\left( A\right) \right) $ and $r_{0}>0$ be
such that $\overline{B\left( a,r_{0}\right) }=\left\{ u\in \mathbb{H}%
:\left\vert u-a\right\vert \leq r_{0}\right\} \subset \mathrm{D}\left(
A\right) .$ Then%
\begin{equation*}
A_{a,r_{0}}^{\#}:=\sup \left\{ \left\vert \hat{u}\right\vert :\hat{u}\in
Au,\;u\in \overline{B\left( a,r_{0}\right) }\right\} <\infty
\end{equation*}%
and for all $\left( x,k\right) \in \mathcal{A},$%
\begin{equation}
r_{0}d\left\updownarrow k\right\updownarrow _{t}\leq \left\langle
x_{t}-a,dk_{t}\right\rangle +A_{a,r_{0}}^{\#}\left( \left\vert
x_{t}-a\right\vert +r_{0}\right) dt  \label{bvintdom}
\end{equation}%
as signed measure on $\mathbb{R}^{+}$.

Moreover for all $0\leq s\leq t$, $y\in \mathbb{D}\left( \mathbb{R}^{+},%
\mathbb{H}\right) $ and $0<\varepsilon \leq 1:$
\begin{equation}
\displaystyle r_{0}\int_{s}^{t}\left\vert A_{\varepsilon }y_{r}\right\vert
dr\leq \int_{s}^{t}\left\langle y_{r}-a,A_{\varepsilon }y_{r}\right\rangle
dr+A_{a,r_{0}}^{\#}\int_{s}^{t}\left[ \left\vert y_{r}-a\right\vert
+\left\vert A^{0}\left( a\right) \right\vert +r_{0}\right] dr\,.
\label{bvintdom-e}
\end{equation}
\end{proposition}

\begin{proof}
Since $A$ is locally bounded on $\mathrm{Int}\left( \mathrm{D}\left(
A\right) \right) $, for $a\in \mathrm{Int}\left( \mathrm{D}\left( A\right)
\right) ,$ there exist $r_{0}>0$ such that $a+r_{0}v\in \mathrm{Int}\left(
\mathrm{D}\left( A\right) \right) $ for all $\left\vert v\right\vert \leq 1$
and%
\begin{equation*}
A_{a,r_{0}}^{\#}:=\sup \left\{ \left\vert \hat{z}\right\vert :\hat{z}\in
Az,\;z\in \overline{B\left( a,r_{0}\right) }\right\} <\infty .
\end{equation*}%
Let $0\leq s=t_{0}<t_{1}<...<t_{n}=t\leq T$ such that $\max_{i}\left(
t_{i+1}-t_{i}\right) =\delta _{n}\rightarrow 0.$ We put in (\ref{prop Anex
13_1}) $z=a+r_{0}v.$ Then%
\begin{equation*}
\int_{t_{i}}^{t_{i+1}}\left\langle x_{r}-\left( a+r_{0}v\right) ,dk_{r}-\hat{%
z}dr\right\rangle \geq 0,\quad \forall \,\left\vert v\right\vert \leq
1,\;\forall \,0\leq s\leq t\leq T
\end{equation*}%
and we get%
\begin{equation*}
r_{0}\left\langle k_{t_{i+1}}-k_{t_{i}},v\right\rangle \leq
\int_{t_{i}}^{t_{i+1}}\left\langle x_{r}-a,dk_{r}\right\rangle
+A_{a,r_{0}}^{\#}\int_{t_{i}}^{t_{i+1}}\left\vert x_{r}-a\right\vert
dr+r_{0}A_{a,r_{0}}^{\#}\left( t_{i+1}-t_{i}\right) ,
\end{equation*}%
for all $\left\vert v\right\vert \leq 1.$

Hence%
\begin{equation*}
r_{0}\left\vert k_{t_{i+1}}-k_{t_{i}}\right\vert \leq
\int_{t_{i}}^{t_{i+1}}\left\langle x_{r}-a,dk_{r}\right\rangle
+A_{a,r_{0}}^{\#}\int_{t_{i}}^{t_{i+1}}\left\vert x_{r}-a\right\vert
dr+r_{0}A_{a,r_{0}}^{\#}\left( t_{i+1}-t_{i}\right)
\end{equation*}%
and adding member by member for $i=\overline{0,n-1}$ the inequality
\begin{equation*}
r_{0}\sum_{i=0}^{n-1}\left\vert k_{t_{i+1}}-k_{t_{i}}\right\vert \leq
\int_{s}^{t}\left\langle x_{t}-a,dk_{t}\right\rangle
+A_{a,r_{0}}^{\#}\int_{s}^{t}\left\vert x_{r}-a\right\vert dr+\left(
t-s\right) r_{0}A_{a,r_{0}}^{\#}
\end{equation*}%
holds and clearly (\ref{bvintdom}) follows.

Since $\left( J_{\varepsilon }y,\int_{0}^{\cdot }A_{\varepsilon
}y_{r}dr\right) \in \mathcal{A}$, from (\ref{bvintdom}) we easily obtain (%
\ref{bvintdom-e}).\hfill \medskip
\end{proof}

\noindent \textbf{Acknowledgement.} \textit{The work of authors L.M. and
A.R. was supported by the projects ERC-Like, no. 1ERC/02.07.2012 and IDEAS
no. 241/05.10.2011. The work of author L.S. was supported by Polish NCN,
grant no. 2012/07/B/ST1/03508.}

\end{document}